\documentclass[USenglish, a4paper, 10pt]{article}
\usepackage[utf8]{inputenc} 
\usepackage{babel} 
\usepackage{lmodern} 
\usepackage[T1]{fontenc} 
\usepackage{amssymb, amsmath, mathrsfs, mathabx, bbm} 
\DeclareMathAlphabet{\mathbbold}{U}{bbold}{m}{n}

\usepackage{titling} 
\usepackage{titlesec} 
		\titlelabel{\thetitle.\quad} 
		\titleformat*{\section}{\center\large} 
		\titleformat*{\subsection}{\sf\large} 
		\titleformat*{\subsubsection}{\sf\it} 
		
\usepackage[colorlinks=true,linkcolor=black, citecolor=Cyan,urlcolor=MidnightBlue]{hyperref} 
\usepackage[usenames,dvipsnames]{color} 

\usepackage{verbatim} 
\usepackage[all]{xy} 
\usepackage[usenames,dvipsnames]{color} 
\usepackage{enumerate} 
\usepackage{leftidx}

\numberwithin{equation}{section} 
\usepackage{array} 
\usepackage{amsthm} 
\swapnumbers 
\theoremstyle{plain}
\newtheorem{theoSec}{Theorem}[section] 

\newtheorem{lemSec}[theoSec]{Lemma}
\newtheorem{proSec}[theoSec]{Proposition}
\newtheorem{corSec}[theoSec]{Corollary}

\theoremstyle{remark}
\newtheorem{remSec}[theoSec]{Remark}
\newtheorem{remsSec}[theoSec]{Remarks}
\newtheorem{noteSec}[theoSec]{Note}

\theoremstyle{plain}
\newtheorem{theo}{Theorem}[subsection] 

\newtheorem{lem}[theo]{Lemma}
\newtheorem{pro}[theo]{Proposition}
\newtheorem{cor}[theo]{Corollary}
\newtheorem{defi}[theo]{Definition}
\theoremstyle{remark}
\newtheorem{rem}[theo]{Remark}

\newtheorem{note}[theo]{Note}

\theoremstyle{plain}
\newtheorem{theodefiSubSub}[theo]{Theorem-Definition}

\theoremstyle{remark}

\title{\textbf{The Baum-Connes property for a quantum (semi-)direct product}}
\author{\textsc{Rubén Martos}}
\date{}
\begin{document}
\maketitle
\renewcommand{\abstractname}{}
\vspace{-2.5cm}
\begin{abstract}
\textsc{Abstract}. The well known ``associativity property'' of the crossed product by a semi-direct product of discrete groups is generalized into the context of discrete \emph{quantum} groups. This decomposition allows to define an appropriate triangulated functor relating the Baum-Connes property for the quantum semi-direct product to the Baum-Connes property for the discrete quantum groups involved in the construction. The corresponding stability result for the Baum-Connes property generalizes the result \cite{Chabert} of J. Chabert for a quantum semi-direct product under torsion-freeness assumption. The $K$-amenability connexion between the discrete quantum groups involved in the construction is investigated as well as the torsion phenomena. The analogous strategy can be applied for the dual of a quantum direct product. In this case, we obtain, in addition, a connection with the \emph{K\"{u}nneth formula}, which is the quantum counterpart to the result \cite{ChabertEchterhoffOyono} of J. Chabert, S. Echterhoff and H. Oyono-Oyono. Again the $K$-amenability connexion between the discrete quantum groups involved in the construction is investigated as well as the torsion phenomena.

\end{abstract}

\section{\textsc{Introduction}}
The Baum-Connes conjecture has been formulated in 1982 by P. Baum and A. Connes. We still do not know any counter example to the original conjecture but it is known that the one with coefficients is false. For this reason we refer to the Baum-Connes conjecture with coefficients as the \emph{Baum-Connes property}. The goal of the conjecture was to understand the link between two operator $K$-groups of different nature that would establish a strong connexion between geometry and topology in a more abstract and general index-theory context. More precisely, if $G$ is a (second countable) locally compact group and $A$ is a (separable) $G$-$C^*$-algebra, then the Baum-Connes property for $G$ with coefficients in $A$ claims that the assembly map $\mu^G_A: K^{top}_{*}(G; A)\longrightarrow K_{*}(G\underset{r}{\ltimes}A)$ is an isomorphism, where $K^{top}_{*}(G; A)$ is the equivariant $K$-homology with compact support of $G$ with coefficients in $A$ and $K_{*}(G\underset{r}{\ltimes}A)$  is the $K$-theory of the reduced crossed product $G\underset{r}{\ltimes}A$. This property has been proved for a large class of groups; let us mention the remarkable work of N. Higson and G. Kasparov \cite{HigsonKasparovHaagerup} about groups with Haagerup property and the one of V. Lafforgue \cite{Lafforgue} about hyperbolic groups.

The equivariant $K$-homology with compact support $K^{top}_{*}(G; A)$ is, of course, the geometrical object obtained from the classifying space of proper actions of $G$, thus it is, \emph{a priori}, easier to calculate than the group $K_{*}(G\underset{r}{\ltimes}A)$ which is the one of analytical nature and then less flexible in its structure. Nevertheless, sometimes the group $K^{top}_{*}(G; A)$ creates non-trivial troubles. This is why R. Meyer and R. Nest provide in 2006 a new formulation of the Baum-Connes property in a well-suited category framework \cite{MeyerNest}. More precisely, if $\mathscr{K}\mathscr{K}^{G}$ is the $G$-equivariant Kasparov category and $F(A):=K_{*}(G\underset{r}{\ltimes}A)$ is the homological functor over $\mathscr{K}\mathscr{K}^{G}$ defining the right-hand side of the Baum-Connes assembly map, then Meyer and Nest show in \cite{MeyerNest} that the assembly map $\mu^G_A$ is equivalent to the natural transformation $\eta^G_A: \mathbb{L}F(A)\longrightarrow F(A)$, where $\mathbb{L}F$ is the localisation of the functor $F$ with respect to an appropriated complementary pair of (localizing) subcategories $(\mathscr{L},\mathscr{N})$; namely, $\mathscr{L}$ is the subcategory of $\mathscr{K}\mathscr{K}^{G}$ of \emph{compactly induced $G$-$C^*$-algebras} and $\mathscr{N}$ is the subcategory of $\mathscr{K}\mathscr{K}^{G}$ of \emph{compactly contractible $G$-$C^*$-algebras}.

This reformulation allows, particularly, to avoid any geometrical construction and thus to replace $G$ by a locally compact \emph{quantum} group $\mathbb{G}$. The problem is the torsion structure of such a $\mathbb{G}$. Indeed, if $G$ is a discrete group, its torsion is completely described in terms of the finite subgroups of $G$ whereas for the case in which $\widehat{\mathbb{G}}$ is a discrete quantum group, the notion of torsion is a non trivial problem that has been introduced firstly by R. Meyer and R. Nest (see \cite{MeyerNestTorsion} and \cite{MeyerNestHomological2}) and recently re-interpreted by Y. Arano and K. De Commer in terms of fusion rings (see \cite{YukiKenny}). But torsion phenomena in the quantum setting are far from been completely understood, so that the current proper formulation of the \emph{quantum} Baum-Connes property concerns only \emph{torsion-free discrete quantum groups $\widehat{\mathbb{G}}$}.

\bigskip
In this article, we study the stability of the Baum-Connes property under the semi-direct product construction. Namely, consider a semi-direct product of locally compact groups $F:=\Gamma\ltimes G$ such that $F$ is equipped with a $\gamma$-element and for every compact subgroup $\Lambda< \Gamma$ the semi-direct product $F_{\Lambda}:=\Lambda\ltimes G$ satisfies the Baum-Connes property. In this situation, J. Chabert shows in \cite{Chabert} that if the Baum-Connes property with coefficients holds for $\Gamma$, so it does for $F$. The strategy consists in using the canonical $*$-isomorphism $F\underset{r}{\ltimes} A\cong \Gamma\underset{r}{\ltimes} (G\underset{r}{\ltimes} A)$ for any $F$-$C^*$-algebra $A$ with the goal of constructing, in a natural way, partial descent homomorphisms and thus to translate the assembly map for $F$ into an assembly map for $\Gamma$ through the transitional group $G$. The assumption of the existence of a $\gamma$-element for $F$ is actually unnecessary (as it is shown in \cite{ChabertPermanence}). In fact, all technical problems we have to deal with, in order to get such a translation, appear in the treatment of the equivariant $K$-homology with compact support in relation with the associativity above. Hence, our strategy is to apply the Meyer-Nest machinery thanks to which all these shortcomings will be encoded in functorial and localization functors properties obtaining that 
the Baum-Connes property with coefficients holds for $F$ if and only if it holds for $\Gamma$ under the same assumptions concerning the compact subgroups of $\Gamma$ (Theorem \ref{theo.FinalTheorem}). It is important to say that H. Oyono-Oyono has already shown in \cite{HerveExtensions} this ``if and only if'' result as a consequence of the stability of the Baum-Connes property for extensions of \emph{discrete} groups.

Our goal is to generalize this result when we have a \emph{quantum} semi-direct product $\mathbb{F}:=\Gamma\ltimes \mathbb{G}$ where $\Gamma$ is a discrete group acting by quantum automorphisms in the compact quantum group $\mathbb{G}$ (such a construction is due to S. Wang \cite{WangSemidirect}). Observe that the $*$-isomorphism $F\underset{r}{\ltimes} A\cong \Gamma\underset{r}{\ltimes} (G\underset{r}{\ltimes} A)$ in the classical case is the key tool to reach all subsequent constructions; therefore our main technical problem is to obtain such associativity in the quantum setting. Namely, we prove that there exists a canonical $*$-isomorphism $\widehat{\mathbb{F}}\underset{r}{\ltimes} A\cong \Gamma\underset{r}{\ltimes} (\widehat{\mathbb{G}}\underset{r}{\ltimes} A)$ for any $\widehat{\mathbb{F}}$-$C^*$-algebra $A$ (Theorem \ref{theo.AssociativityandFunctorZ}). With this aim in mind, it is convenient to analyze the structure of the reduced crossed products by a discrete (quantum) group and, in this sense, we establish an universal property for such a crossed product (Theorem \ref{theo.QuantumReducedCrossedProduct}), which will be very useful throughout all the article. This decomposition allows to define a triangulated functor $\mathscr{K}\mathscr{K}^{\widehat{\mathbb{F}}}\longrightarrow \mathscr{K}\mathscr{K}^{\Gamma}$, which translates the Baum-Connes property for $\widehat{\mathbb{F}}$ into the Baum-Connes property for $\Gamma$ through the transitional quantum group $\widehat{\mathbb{G}}$.

Notice however that, according to the above discussion about the torsion phenomena of quantum groups, we need $\widehat{\mathbb{F}}$ to be torsion-free in order to give a meaning to the \emph{quantum} Baum-Connes property. In this way, we give a more precise picture of the torsion nature of such a quantum semi-direct product and obtain the natural result that we expect: if $\Gamma$ and $\widehat{\mathbb{G}}$ are torsion-free, then $\widehat{\mathbb{F}}$ is torsion-free (Theorem \ref{theo.TorsionQuantumsemi-direct}). Since we work with torsion-freeness assumption in order to legitimate the formulation of the quantum Baum-Connes property, we have in particular that the only finite subgroup of $\Gamma$ is the trivial one as opposed to the classical case of \cite{Chabert}. A last property of own interest is studied: $K$-amenability (see Theorem \ref{theo.KAmenability}). Namely, we prove that $\widehat{\mathbb{F}}$ is $K$-amenable if and only if $\Gamma$ and $\widehat{\mathbb{G}}$ are $K$-amenable. 
A crucial observation is done concerning a \emph{compact bicrossed product} in the sense of \cite{FimaBiproduit}. Namely, if $\mathbb{F}=\Gamma {_\alpha}\bowtie_{\beta}G$ is such a compact quantum group, the torsion-freeness assumption about $\widehat{\mathbb{F}}$ implies that the action $\beta$ is trivial (Proposition \ref{pro.BicrossedBetaTrivial}). Hence $\mathbb{F}$ becomes a quantum semi-direct product, for which the stabilization property has already been established.

Finally, we observe that the preceding categorical strategy can be adapted for the study of the stabilization of the Baum-Connes property under other constructions of compact quantum groups. For instance we examine the quantum direct product construction (due to S. Wang \cite{WangSemidirect}). If $\mathbb{F}:=\mathbb{G}\times\mathbb{H}$ is a direct product of two compact quantum groups and $A_0\in Obj(\mathscr{K}\mathscr{K}^{\widehat{\mathbb{G}}})$ is a fixed object, we define a triangulated functor $\mathscr{K}\mathscr{K}^{\widehat{\mathbb{H}}}\longrightarrow \mathscr{K}\mathscr{K}^{\widehat{\mathbb{F}}}$, which allows to establish the Baum-Connes property for $\widehat{\mathbb{F}}$ (with coefficients in tensor products) whenever $\widehat{\mathbb{G}}$ and $\widehat{\mathbb{H}}$ satisfy the \emph{strong} Baum-Connes property (Theorem \ref{theo.SBCDirectProduct}). Moreover the usual Baum-Connes property for $\widehat{\mathbb{F}}$ is closely related to the \emph{K\"{u}nneth formula} in the analogous way as established in \cite{ChabertEchterhoffOyono} by J. Chabert, S. Echterhoff and H. Oyono-Oyono (Corollary \ref{theo.BCDirectProduct}). Again the torsion phenomena and the $K$-amenability property are studied. We prove that $\widehat{\mathbb{G}}$ and $\widehat{\mathbb{H}}$ are \emph{strong} torsion-free if and only if $\widehat{\mathbb{F}}$ is \emph{strong} torsion-free (Proposition \ref{pro.StrongTorsionQuantumDirect}). We prove that $\widehat{\mathbb{F}}$ is $K$-amenable if and only if $\widehat{\mathbb{G}}$ and $\widehat{\mathbb{H}}$ are $K$-amenable (Theorem \ref{theo.KAmenabilityQDirectProduct}).

\bigskip
\textsc{Acknowledgements.} I would like to thank sincerely my advisor Pierre Fima because of his very useful remarks, comments and rectifications. The rewarding and passionate discussions with him have been a continuous motivation. It is a pleasure to thank as well Roland Vergnioux who has read very carefully my PhD manuscript in order to improve in a capital way the present paper.

\section{\textsc{Preliminary results}}
	\subsection{Notations and conventions}\label{sec.Notations}
	First of all, let us fix the notations we use throughout the whole article.
	
	We denote by $\mathcal{B}(H)$ the space of all linear operators of the Hilbert space $H$ and by $\mathcal{L}_{A}(H)$ the space of all adjointable operators of the Hilbert $A$-module $H$. All our $C^*$-algebras are supposed to be \emph{separables} and all our Hilbert modules are supposed to be \emph{countably generated}. Hilbert $A$-modules are considered to be \emph{right $A$-modules}, so that the corresponding inner products are considered to be conjugate-linear on the left and linear on the right. We use systematically the leg and Sweedler notations. We denote by $\mathscr{A}b$ the abelian category of abelian groups and by $\mathscr{A}b^{\mathbb{Z}/2}$ the abelian category of $\mathbb{Z}/2$-graded groups of $\mathscr{A}b$. The symbol $\otimes$ stands for the minimal tensor product of $C^*$-algebras, the exterior/interior tensor product of Hilbert modules depending on the context. The symbol $\underset{\max}{\otimes}$ stands for the maximal tensor product of $C^*$-algebras. If $M$ and $N$ are two $R$-modules for some ring $R$, the symbol $\odot$ stands for their algebraic tensor product over $R$, and we write $M\underset{R}{\odot} N$. If $S, A$ are $C^*$-algebras, we denote by $M(A)$ the multiplier algebra of $A$ and we put $\widetilde{M}(A\otimes S):=\{x\in M(A\otimes S)\ |\ x(id_A\otimes S)\subset A\otimes S\mbox{ and } (id_A\otimes S)x\subset A\otimes S\}$, which contains $M(A)\otimes S$.
	
	If $H$ is a finite dimensional Hilbert space, we denote by $\overline{H}$ its dual or conjugate vector space, so that if $\{\xi_1,\ldots,\xi_{dim(H)}\}$ is an orthonormal basis for $H$ and $\{\omega_1,\ldots,\omega_{dim(H)}\}$ its dual basis in $\overline{H}$, we denote by $*$ the usual homomorphism between $H$ and $\overline{H}$ such that $\xi_i^*=\omega_i$, for all $i=1,\ldots,dim(H)$.
	
	If $\mathbb{G}=(C(\mathbb{G}), \Delta)$ is a compact quantum group, the set of all unitary equivalence classes of irreducible unitary finite dimensional representations of $\mathbb{G}$ is denoted by $Irr(\mathbb{G})$. The collection of all finite dimensional unitary representations of $\mathbb{G}$ is denoted by $Rep(\mathbb{G})$ and the $C^*$-tensor category of representations of $\mathbb{G}$ is denoted by $\mathscr{R}ep(\mathbb{G})$. The trivial representation of $\mathbb{G}$ is denoted by $\epsilon$. If $x\in Irr(\mathbb{G})$ is such a class, we write $w^x\in\mathcal{B}(H_x)\otimes C(\mathbb{G})$ for a representative of $x$ and $H_x$ for the finite dimensional Hilbert space on which $w^x$ acts (we write $dim(x):= n_x$ for the dimension of $H_x$). The linear span of matrix coefficients of all finite dimensional representations of $\mathbb{G}$ is denoted by $Pol(\mathbb{G})$. Given $x,y\in Irr(\mathbb{G})$, the tensor product of $x$ and $y$ is denoted by $x\otop y$. Given $x\in Irr(\mathbb{G})$, there exists a unique class of irreducible unitary finite dimensional representation of $\mathbb{G}$ denoted by $\overline{x}$ such that $Mor(\epsilon, x\otop \overline{x})\neq 0\neq Mor(\epsilon, \overline{x}\otop x)$ and it is called contragredient or conjugate representation of $x$. Consequently, there exist non-trivial invariant vectors in $H_x\otimes H_{\overline{x}}$ and $H_{\overline{x}}\otimes H_x$ denoted by $E_x$ and $E_{\overline{x}}$, respectively; which regarded as intertwiner operators, they will be denoted by $\Phi_x$ and $\Phi_{\overline{x}}$, respectively and called \emph{canonical intertwiners}. In this way, there exists an antilinear isomorphism $J_x: H_x\longrightarrow H_{\overline{x}}$. We define the operator $Q_x:=J_x^*J_x$, which is an invertible positive self-adjoint operator unique up to multiplication of a real number. We choose this number such that $J_x$ is normalized meaning that $Tr(J_x^*J_x)=Tr((J_x^*J_x)^{-1})$. Thus, the quantum dimension of a class $x\in Irr(\mathbb{G})$ is defined by $dim_q(x)=Tr(Q_x)$. Let $\{\xi^x_1,\ldots, \xi^x_{n_x}\}$ be an orthonormal basis of $H_x$ which diagonalizes $Q_x$ and let $\{\omega^x_1,\ldots, \omega^x_{n_x}\}$ be its dual basis in $\overline{H}_x$. If $\{\xi^{\overline{x}}_1,\ldots, \xi^{\overline{x}}_{n_{\overline{x}}}\}$ is an orthonormal basis of $H_{\overline{x}}$ and $\{\omega^{\overline{x}}_1,\ldots, \omega^{\overline{x}}_{n_{\overline{x}}}\}$ denotes its dual basis in $\overline{H}_{\overline{x}}$ , then we identify systematically $\overline{H}_{\overline{x}}$ and $H_{x}$ via the linear map $\omega^{\overline{x}}_k\mapsto \frac{1}{\sqrt{\lambda^x_k}}\xi^x_k$, for all $k=1,\ldots, n_x$, where $\{\lambda^x_1, \ldots, \lambda^x_{n_x}\}$ are the eigenvalues of $Q_x$. The fundamental unitary of $\mathbb{G}$ is denoted by $W_{\mathbb{G}}\in\mathcal{B}(L^2(\mathbb{G})\otimes L^2(\mathbb{G}))$ and we recall that $\Delta(a)=W^*_{\mathbb{G}}(1\otimes a)W_{\mathbb{G}}$, for all $a\in C(\mathbb{G})$ and $\widehat{\Delta}(a)=\widehat{W}^*_{\mathbb{G}}(1\otimes a)\widehat{W}_{\mathbb{G}}$, for all $a\in c_0(\widehat{\mathbb{G}})$ where $\widehat{W}_{\mathbb{G}}:=\Sigma W^*_{\mathbb{G}}\Sigma$. The Haar state of $\mathbb{G}$ is denoted by $h_{\mathbb{G}}$ and the corresponding GNS construction by $(L^2(\mathbb{G}), \lambda, \Omega)$. We adopt the standard convention for the inner product on $L^2(\mathbb{G})$, which means that $\langle \lambda(a)\Omega, \lambda(b)\Omega\rangle:=h_{\mathbb{G}}(a^*b)$ for all $a,b\in C(\mathbb{G})$. We denote by $\widehat{h}_{L}$ the left Haar weight of $\widehat{\mathbb{G}}$ and by $\widehat{\lambda}: c_0(\widehat{\mathbb{G}})\longrightarrow \mathcal{B}(L^2(\mathbb{G}))$ the corresponding left regular representation. For more details of these definitions and constructions we refer to \cite{Woronowicz}. If $\mathbb{H}$ is another compact quantum group such that $\widehat{\mathbb{H}}<\widehat{\mathbb{G}}$ is a discrete quantum subgroup, then we define an equivalence relation in $Irr(\mathbb{G})$ in the following way: $x,y\in Irr(\mathbb{G})$, $x\sim y\Leftrightarrow$ there exists an irreducible representation $z\in Irr(\mathbb{H})$ such that $y\otop \overline{x}\supset z$. We say that $\widehat{\mathbb{H}}$ is divisible in $\widehat{\mathbb{G}}$ if for each $\alpha\in \sim\backslash Irr(\mathbb{G})$ there exists a representation $l_{\alpha}\in \alpha$ such that $s\otop l_{\alpha}$ is irreducible for all $s\in Irr(\mathbb{H})$ and $s\otop l_\alpha\cong s'\otop l_\alpha$ implies $s\cong s'$, for all $s,s'\in Irr(\mathbb{H})$. This is equivalent to say that for each $\alpha\in Irr(\mathbb{G})/\sim$ there exists a representation $r_{\alpha}\in \alpha$ such that $r_{\alpha}\otop s$ is irreducible for all $s\in Irr(\mathbb{H})$ and $r_\alpha\otop s\cong r_\alpha\otop s'$ implies $s\cong s'$, for all $s,s'\in Irr(\mathbb{H})$. This is again equivalent to say that there exists a $\widehat{\mathbb{H}}$-equivariant $*$-isomorphism $c_0(\widehat{\mathbb{G}})\cong c_0(\widehat{\mathbb{H}})\otimes c_0(\widehat{\mathbb{H}}\backslash\widehat{\mathbb{G}})$ (see \cite{VoigtBaumConnesUnitaryFree} for a proof). The trivial quantum subgroup of $\widehat{\mathbb{G}}$ is denoted by $\mathbb{E}$. 
	
	
	If $\alpha: A\longrightarrow \widetilde{M}(c_0(\widehat{\mathbb{G}})\otimes A)$ is a \emph{left} action of $\widehat{\mathbb{G}}$ on a $C^*$-algebra $A$ and $x\in Irr(\mathbb{G})$, we write $\alpha^x(a):=\alpha(a)(p_x\otimes id_A)\in \mathcal{B}(H_x)\otimes A$, for all $a\in A$ where $p_x$ is the central projection of $c_0(\widehat{\mathbb{G}})$ on $\mathcal{B}(H_x)$ so that $\alpha(a)=\underset{x\in Irr(\mathbb{G})}{\oplus^{c_0}}\alpha^x(a)$. If $\{\xi^x_1,\ldots,\xi^x_{n_x}\}$ is an orthonormal basis of $H_x$ and $\omega_{\xi_i,\xi_j}$ is the linear form of $\mathcal{B}(H_x)$ defined by $\omega_{\xi_i,\xi_j}(T):=\langle \xi_i, T(\xi_j) \rangle$ for all $T\in\mathcal{B}(H_x)$, we define $\alpha^x_{i,j}(a):=(\omega_{\xi^x_i,\xi^x_j}\otimes id_A)(\alpha^x(a))\in A$, for all $a\in A$ and all $i,j=1\ldots,n_x$. Hence, if $\{m^x_{i,j}\}_{i,j=1,\ldots,n_x}$ are the matrix units in $\mathcal{B}(H_x)$ associated to the basis $\{\xi^x_1,\ldots,\xi^x_{n_x}\}$, then we have $\alpha^x(a)=\overset{n_x}{\underset{i,j=1}{\sum}} m^x_{i,j}\otimes \alpha^x_{i,j}(a)$. In an analogous way, if $U\in M(c_0(\widehat{\mathbb{G}})\otimes C)$ for some $C^*$-algebra $C$, then we define $U^x:=U(p_x\otimes id_C)\in \mathcal{B}(H_x)\otimes C$ and $U^x_{i,j}:=(\omega_{\xi^x_i,\xi^x_j}\otimes id_C)U^x\in C$, for all $i,j=1\ldots,n_x$. We use systematically the well known one-to-one correspondence between unitary representations $U\in M(c_0(\widehat{\mathbb{G}})\otimes C)$ and non-degenerate $*$-homomorphisms $\phi_U: C_m(\mathbb{G})\longrightarrow M(C)$ which is such that $\phi_U(w^x_{i,j}):=U^x_{i,j}$ for all $x\in Irr(\mathbb{G})$ and all $i,j=1,\ldots, n_x$. If $u\in Rep(\mathbb{G})$ is any finite dimensional unitary representation of $\mathbb{G}$, then we can define $\alpha^u(a)\in \mathcal{B}(H_u)\otimes A$ and $U^u\in \mathcal{B}(H_u)\otimes C$ using the decomposition of $u$ into direct sum of irreducible representations. It is worth mentioning that giving an action $\alpha: A\longrightarrow \widetilde{M}(c_0(\widehat{\mathbb{G}})\otimes A)$ is equivalent to give a family of $*$-homomorphisms $\alpha^x:A \longrightarrow \mathcal{B}(H_x)\otimes A$, for every $x\in Irr(\mathbb{G})$ satisfying $i)$ $(\Phi\otimes id_A)\alpha^x(a)=(id_{H_z}\otimes \alpha^y)(\alpha^z(a))(\Phi\otimes id_A)$ and $ii)$ $[\alpha^x(A)(H_x\otimes A)]=H_x\otimes A$, for all $a\in A$, all $x,y,z\in Irr(\mathbb{G})$ and all $\Phi\in Mor(x, y\otop z)$. This can be checked straightforwardly by using the above definitions.
	
	Let $(A,\delta)$ be a right $\mathbb{G}$-$C^*$-algebra; the fixed points space of $A$ is denoted by $A^{\delta}$ and we say that $\delta$ is a \emph{torsion action} if $\delta$ is ergodic and $A$ is finite dimensional (hence unital). Given an irreducible representation $x\in Irr(\mathbb{G})$, we put $K_x:=\{X\in \overline{H}_x\otimes A\ |\ (id\otimes \delta)(X)=\big[X\big]_{12}\big[w^x\big]_{13})\}$ and use systematically the natural identification $K_x\cong Mor(x, \delta)$, with $Mor(x, \delta):=\{T:H_x\longrightarrow A \ |\ T \mbox{ is linear such that } \delta(T(\xi))=(T\otimes id_{C(\mathbb{G})})w^x(\xi\otimes 1_{C(\mathbb{G})})\}$. The corresponding $x$-spectral subspace of $A$ is denoted by $\mathcal{A}_x$ so that the corresponding Podle\'{s} subalgebra of $A$ is denoted by $\mathcal{A}_{\mathbb{G}}$. By abuse of language, both $K_x$ and $Mor(x,\delta)$ are called \emph{spectral subspaces} as well. Finally, from the general theory of spectral subspaces (see \cite{KennyActions} for the details) we recall that each $x$-spectral subspace $\mathcal{A}_x$ is finite dimensional with $dim(\mathcal{A}_x)\leqslant dim_q(x)$ whenever $\delta$ is ergodic; and that there always exists a conditional expectation $\mathbb{E}_{\delta}:A \longrightarrow A^{\delta}$ which is faithful on $\mathcal{A}_{\mathbb{G}}$ so that we have the following decomposition $\mathcal{A}_{\mathbb{G}}=\underset{x\in Irr(\mathbb{G})}{\bigoplus}\mathcal{A}_x$. 
	
	For our purposes it is convenient to introduce the following operations in $\mathcal{A}_{\mathbb{G}}$: given irreducible representations of $\mathbb{G}$, say $x,y,z\in Irr(\mathbb{G})$, we define the following element $X\underset{\Phi}{\otimes} Y:= (\big[X\big]_{13}\big[Y\big]_{23})(\Phi\otimes 1_A)\in \overline{H}_z\otimes A$, for all $X\in K_{x}$ and $Y\in K_{y}$ and all intertwiner $\Phi\in Mor(z, x\otop y)$; where $\big[X\big]_{13}$ and $\big[Y\big]_{23}$ are the corresponding legs of $X$ and $Y$ in $\overline{H}_x\otimes \overline{H}_y\otimes A$. It is straightforward to check that $X\underset{\Phi}{\otimes} Y\in K_z$. Precisely, we have the following formula $(id\otimes\delta)\Big(\big[X\big]_{13}\big[Y\big]_{23}\Big)=\Big[\big[X\big]_{13}\big[Y\big]_{23}\Big]_{12}\big[w^{x\otop y}\big]_{13}$, where the legs of the right hand side of the identity are considered in $\overline{H}_{x\otop y}\otimes A\otimes C(\mathbb{G})$. Hence, the decomposition of $x\otop y$ into direct sum of irreducible representations, say $\{z_k\}_{k=1,\ldots, r}$ for some $r\in \mathbb{N}$ yields that $\big[X\big]_{13}\big[Y\big]_{23}\in \overset{r}{\underset{k=1}{\bigoplus}}\ K_{z_k}$. The element $X\underset{\Phi}{\otimes} Y$ is called \emph{spectral product of $X$ by $Y$ with respect to $\Phi$}.
	
	Given an irreducible representation $x\in Irr(\mathbb{G})$ and an element $X\in K_x$, we define the \emph{spectral conjugate of $X$} as the element $X^\#:=(*\circ J_x\otimes id_A)(X^*)\in K_{\overline{x}}$. If $\{\xi^x_1,\ldots,\xi^x_{dim(x)}\}$ is an orthonormal basis of $H_x$ which diagonalizes the canonical operator $Q_x$ and $\{\omega^x_1,\ldots,\omega^x_{dim(x)}\}$ is its dual basis in the dual space $\overline{H}_{x}$, then the intertwiner $\Phi_x$ can be written in coordinates under the expression $\Phi_x=\overset{n_x}{\underset{i=1}{\sum}}\sqrt{\lambda^x_i}\ \xi^x_i\otimes \xi^{\overline{x}}_i$, where $\lambda^x_i\in\mathbb{R}^+$ is the eigenvalue of $Q_x$ associated to the vector $\xi^x_i$, for each $i=1,\ldots, n_x\equiv dim(x)$. In this situation, the element $X\in K_x$ can be written in coordinates under the form $X=\overset{n_x}{\underset{i=1}{\sum}} \omega^x_i\otimes a_i$ as element of $\overline{H}_x\otimes A$, for some $a_i\in A$, for all $i=1,\ldots, n_x$ and its spectral conjugate is given by $X^\#:=\overset{n_x}{\underset{i=1}{\sum}} J_x(\xi^x_i)^*\otimes a^*_i\in \overline{H_{\overline{x}}}\otimes A$. A straightforward computation shows that the association $X\longmapsto X^\#$ is antilinear and that $X^\#\in K_{\overline{x}}$.
	\begin{rem}
		We can do thus the multiplication over $\Phi_x$ and over $\Phi_{\overline{x}}$ defined above. Namely, the previous constructions show that $X\underset{\Phi_x}{\otimes} X^\#\in K_\epsilon\mbox{ and } X^\#\underset{\Phi_{\overline{x}}}{\otimes}X\in K_{\epsilon}$. Observe that $K_\epsilon=A^{\delta}$. Hence, if $\delta$ is an ergodic action, $X \underset{\Phi_x}{\otimes}X^\#$ and $X^\#\underset{\Phi_{\overline{x}}}{\otimes}X$ are scalar multiples of $1_A$. It will be useful to write down explicit formulas in coordinates for these products: $X\underset{\Phi_x}{\otimes} X^\#=\overset{n_x}{\underset{i=1}{\sum}} \lambda^x_i\ a_ia^*_i\mbox{ ; } X^\# \underset{\Phi_{\overline{x}}}{\otimes} X=\overset{n_x}{\underset{i=1}{\sum}} a^*_ia_i$.
	\end{rem}
	
	Concerning these constructions, the following technical observation will help to conclude later on our torsion analysis for a quantum semi-direct product.
	\begin{lem}\label{lem.NonZeroProducts}
		Let $\mathbb{G}$ be a compact quantum group and $(A,\delta)$ a right ergodic $\mathbb{G}$-$C^*$-algebra. Given irreducible representations $x,y\in Irr(\mathbb{G})$ and non-zero elements $X\in K_x$, $Y\in K_y$; then there exist an irreducible representation $z\in Irr(\mathbb{G})$ and an intertwiner $\Phi\in Mor(z, x\otop y)$ such that $X\underset{\Phi}{\otimes} Y\neq 0$.
	\end{lem}
	\begin{proof}			
		Let us fix orthonormal basis $\{\xi^x_1,\ldots,\xi^x_{dim(x)}\}$ of $H_x$ and $\{\xi^y_1,\ldots,\xi^y_{dim(y)}\}$ of $H_y$ that diagonalise the canonical operators $Q_x=J_x^*J_x$ and $Q_y=J_y^*J_y$, respectively; with eigenvalues $\{\lambda^x_i\}_{i=1,\ldots,n_x}$ and $\{\mu^y_j\}_{j=1,\ldots,n_y}$, respectively. Denote by $\{\omega^x_1,\ldots,\omega^x_{dim(x)}\}$ and $\{\omega^y_1,\ldots,\omega^y_{dim(y)}\}$ the corresponding dual basis of $\overline{H}_x$ and $\overline{H}_y$, respectively.
					
		Suppose that for all irreducible representation $z\in Irr(\mathbb{G})$ and all intertwiner $\Phi\in Mor(z, x\otop y)$ we have $X\otimes_{\Phi} Y=0$, that is, $X\otimes_{\Phi} Y=\underset{i,j}{\sum}\omega^x_i\otimes \omega^y_j\circ \Phi\otimes a_ib_j=0$ where we use the coordinate expressions for $X$ and $Y$ as above.
			
		Multiplying by $\Phi^*\otimes 1_A$, this is still zero,
		\begin{equation}\label{eq.FormulaZero}
			\underset{i,j}{\sum}\omega^x_i\otimes \omega^y_j\circ \Phi\Phi^*\otimes a_ib_j=0
		\end{equation}
			
		This is true for \emph{every} irreducible representation $z\in Irr(\mathbb{G})$ and \emph{every} intertwiner $\Phi\in Mor(z, x\otop y)$. Let us consider the decomposition in direct sum of irreducible representations of $x\otop y$, say $\{z_l\}_l$; and denote by $\{p_l\}_l\subset \mathcal{B}(H_x\otimes H_y)$ the corresponding family of mutually orthogonal finite-dimensional projections with sum $id_{H_x\otimes H_y}$. Likewise, for every $k=1,\ldots, dim(Mor(z,x\otop y))$ consider the corresponding intertwiners $\Phi_k\in Mor(z_l, x\otop y)$ for each $l$ which are such that $\Phi^*_k\Phi_k=id\mbox{ and } \overset{dim(Mor(z,x\otop y))}{\underset{k=1}{\sum}}\Phi_k\Phi^*_k= p_l$.
			
		Hence we consider the identity (\ref{eq.FormulaZero}) above for these intertwiners $\Phi_k$ for each $k=1,\ldots, dim\big(Mor(z, x\otop y)\big)$ and next we sum over $k$. We get $\underset{i,j}{\sum}\omega^x_i\otimes \omega^y_j\ p_l\otimes a_ib_j=0$. Next, we can sum over $l$ and we get $\underset{i,j}{\sum}\omega^x_i\otimes \omega^y_j \otimes a_ib_j=0$, what implies that $a_ib_j=0$, for all $i=1,\ldots, n_x$ and all $j=1,\ldots, n_y$.
			
		Since $\delta$ is an ergodic action by assumption we have that $X^\# \underset{\Phi_{\overline{x}}}{\otimes} X=\overset{n_x}{\underset{i=1}{\sum}} a^*_ia_i=\lambda\ 1_A$ and $Y \underset{\Phi_{y}}{\otimes} Y^\#=\overset{n_y}{\underset{j=1}{\sum}}\mu^y_j\ b_jb^*_j=\mu\ 1_A$, for some $\lambda, \mu\in\mathbb{C}$. Since $X$ and $Y$ are supposed to be non-zero, there exist at least one $i=1,\ldots, n_x$ and one $j=1,\ldots, n_y$ such that $a_i\neq 0$ and $b_j\neq 0$. Consequently, $X^\# \underset{\Phi_{\overline{x}}}{\otimes} X\neq 0$ and $Y \underset{\Phi_{y}}{\otimes} Y^\#\neq 0$. In other words, we have $\lambda\neq 0$ and $\mu\neq 0$.
			Using equation the preceding discussion, we write
			$$0\neq\lambda\ \mu=\Big(\overset{n_x}{\underset{i=1}{\sum}}\ a_i^*a_i\Big)\Big(\overset{n_y}{\underset{j=1}{\sum}}\mu^y_j\ b_jb^*_j\Big)=\underset{i,j}{\sum}\mu^y_j\ a^*_ia_ib_jb^*_j=0\mbox{, }$$ \enlargethispage{0.5cm}a contradiction.	
		\end{proof}
		
	\subsection{Fusion rings and (\emph{strong}) torsion-freeness}\label{sec.Torsion}
		The notion of torsion-freeness for a discrete quantum group was initially introduced by R. Meyer and R. Nest (see \cite{MeyerNestTorsion}, \cite{MeyerNestHomological2} and \cite{VoigtBaumConnesAutomorphisms} for more details). 
		\begin{defi}
			Let $\mathbb{G}$ be a compact quantum group. We say that $\widehat{\mathbb{G}}$ is torsion-free if one of the following equivalent condition holds:
				\begin{enumerate}[i)]
					\item Any torsion action of $\mathbb{G}$ is $\mathbb{G}$-equivariantly Morita equivalent to the trivial $\mathbb{G}$-$C^*$-algebra $\mathbb{C}$.
					\item Let $(A,\delta)$ be a finite dimensional left $\mathbb{G}$-$C^*$-algebra.
						\begin{enumerate}[a)]
							\item If $\delta$ is ergodic, then $A$ is simple.
							\item If $A$ is simple, then there exists a finite dimensional unitary representation $(H_u,u)$ of $\mathbb{G}$ such that $A\cong \mathcal{K}(H_u)$ as $\mathbb{G}$-$C^*$-algebras.
						\end{enumerate}
				\end{enumerate}
		\end{defi}
		\begin{rem}\label{rem.TorsionFiniteSubgroup}
			If $\widehat{\mathbb{G}}$ is a discrete quantum group that has a non-trivial finite discrete quantum subgroup, then $\widehat{\mathbb{G}}$ is \emph{not} torsion-free because the co-multiplication of such a non-trivial finite discrete quantum group $\Lambda$ would define an ergodic action of $\mathbb{G}$ on $C(\Lambda)$. 
		\end{rem}
		
		It is important to observe that torsion-freeness in the sense of Meyer-Nest is \emph{not} preserved, in general, by discrete quantum subgroups. For instance, consider $\widehat{SO_q(3)}<\widehat{SU_q(2)}$. While $\widehat{SU_q(2)}$ is torsion-free by \cite{VoigtBaumConnesFree}, $\widehat{SO_q(3)}$ is \emph{not} torsion-free by \cite{SoltanSO3}.
	
	In relation with the results obtained in \cite{RubenAmauryTorsion}, we can consider an other example more complicated. Let $\mathbb{G}$ be a compact quantum group such that $\widehat{\mathbb{G}}$ is torsion-free. Then the dual of the free product $\mathbb{G}\ast SU_q(2)$ is torsion-free (because $\widehat{SU_q(2)}$ is torsion-free for all $q\in (-1,1)\backslash \{0\}$ as it is shown in \cite{VoigtBaumConnesFree} and torsion-freeness is preserved by free product as it is shown in \cite{YukiKenny}). Consider the Lemeux-Tarrago's discrete quantum subgroup $\widehat{\mathbb{H}}_q<\widehat{\mathbb{G}\ast SU_q(2)}$ which is such that $\mathbb{H}_q$ is monoidally equivalent to the free wreath product $\mathbb{G}\wr_*S^+_N$ (see \cite{TarragoWreath} for more details). It is explained in \cite{RubenAmauryTorsion} that the dual of $\mathbb{G}\wr_*S^+_N$ is \emph{never} torsion-free. Hence $\widehat{\mathbb{H}}_q$ neither (because torsion-freeness is preserved under monoidally equivalence as it is shown in \cite{VoigtBaumConnesFree} or \cite{RijdtMonoidalProbabilistic}). 

	
	It is reasonable to expect that torsion-freeness is preserved under \emph{divisible} discrete quantum subgroups. Inspired by the study carried out in \cite{RubenAmauryTorsion}, we expect to be able to apply the techniques from \cite{YukiKenny} for proving the following stability result: given a compact quantum group $\mathbb{G}$, $\widehat{\mathbb{G}}$ is torsion-free if and only if every divisible discrete quantum subgroup $\widehat{\mathbb{H}}<\widehat{\mathbb{G}}$ is torsion-free.
		
	\bigskip
		Recently \cite{YukiKenny}, Y. Arano and K. De Commer have re-interpreted this notion in terms of fusion rings giving a \emph{strong} torsion-freeness definition that implies the Meyer-Nest's one. Let us recall briefly the corresponding fusion rings theory (we refer to \cite{YukiKenny} for all the details and further properties).
		
		Let $(I, \mathbbm{1})$ be an involutive pointed set and $J$ any set. Let $(\mathbb{Z}_I, \oplus, \otimes)$ be a fusion ring with fusion rules given by the collection of natural numbers $\{N_{\beta, \gamma}^{\alpha}\}_{\alpha,\beta,\gamma\in I}$ and $(\mathbb{Z}_J, \oplus, \otimes)$ a $J$-based co-finite module with respect to $\mathbb{Z}_I$ defined by a collection of natural numbers $\{N_{\alpha, b}^{c}\}_{\alpha\in I, b,c\in J}$. We say that $(\mathbb{Z}_I, \oplus, \otimes)$ is trivial if $I=\{\mathbbm{1}\}$. We say that $(\mathbb{Z}_J, \oplus, \otimes)$ is \emph{connected} if for any $b,c\in J$, there exists $\alpha\in I$ such that $N_{\alpha,b}^{c}\neq 0$. Remark that $(\mathbb{Z}_I, \oplus, \otimes)$ is a natural $I$-based (connected) co-finite module with respect to itself with left $\otimes$-multiplication and bilinear form given by $\langle \alpha,\beta \rangle:=\alpha\otimes\overline{\beta}\mbox{,}$ for all $\alpha,\beta\in I$. It is called \emph{the standard $I$-based module}.
		
		If $\widehat{\mathbb{G}}$ be a discrete quantum group, we denote by $R(\mathbb{G}):=Fus(\widehat{\mathbb{G}})$ the usual fusion ring of $\widehat{\mathbb{G}}$ given by the irreducible representations of $\mathbb{G}$. If $(\mathbb{Z}_{I_1}, \oplus, \otimes, d_1)$ and $(\mathbb{Z}_{I_2}, \oplus, \otimes, d_2)$ are two fusion rings, we define the tensor product of $\mathbb{Z}_{I_1}$ and $\mathbb{Z}_{I_2}$ as the free $\mathbb{Z}$-module $\mathbb{Z}_{I_1}\underset{\mathbb{Z}}{\odot} \mathbb{Z}_{I_2}$ which is naturally a fusion ring denoted by $\mathbb{Z}_{I_1}\otimes\mathbb{Z}_{I_2}$.

		In this situation, a fusion ring $(\mathbb{Z}_I, \oplus, \otimes, d)$ is called torsion-free if any based connected co-finite module is isomorphic to the standard based module. In particular, we have 
		\begin{defi}
			Let $\mathbb{G}$ be a compact quantum group. We say that $\widehat{\mathbb{G}}$ is strong torsion-free if $Fus(\widehat{\mathbb{G}})$ is torsion-free.
		\end{defi}
		\begin{rem}\label{rem.StrongTorsionFreeSubroups}
			It is known that the strong torsion-freeness is not preserved by quantum subgroups (see \cite{YukiKenny} for a counter example). However, we can show (see \cite{YukiKenny} for a proof) that if $\mathbb{G}$ and $\mathbb{H}$ are compact quantum groups such that $\widehat{\mathbb{H}}<\widehat{\mathbb{G}}$ is a \emph{divisible} discrete quantum subgroup, then $\widehat{\mathbb{H}}$ is strong torsion-free whenever $\widehat{\mathbb{G}}$ is strong torsion-free. 
		\end{rem}
		
			

		Observe that torsion-freeness of fusion rings is \emph{not} preserved, in general, by tensor product. More precisely, we have the following result (see \cite{YukiKenny} for a proof).
		\begin{theo}\label{theo.TensorProductFusionRings}
			Let $(\mathbb{Z}_{I_1}, \otimes)$ and $(\mathbb{Z}_{I_2}, \otimes)$ be torsion-free fusion rings and assume that $(\mathbb{Z}_{I_1}\odot \mathbb{Z}_{I_2}, \otimes)$ is \emph{not} torsion-free. Then $(\mathbb{Z}_{I_1}, \otimes)$ and $(\mathbb{Z}_{I_2}, \otimes)$ have non-trivial isomorphic finite fusion subrings.
		\end{theo} 
			
	\subsection{Quantum crossed products}
	We recall here the crossed product construction. The next result is well known to specialists. Since we could not find any reference in the literature, we include the complete proof for the reader's convenience.
	\begin{theodefiSubSub}\label{theo.QuantumReducedCrossedProduct}
	Let $\mathbb{G}$ be a compact quantum group and $(A, \alpha)$ a left $\widehat{\mathbb{G}}$-$C^*$-algebra. There exists a unique (up to a canonical isomorphism) $C^*$-algebra $P$ with a non-degenerate $*$-homomorphism $\pi: A\longrightarrow P$, a unitary representation $U\in M(c_0(\widehat{\mathbb{G}})\otimes P)$ and a non-degenerate completely positive KSGNS-faithful map $E: P\longrightarrow M(A)$ such that
	\begin{enumerate}[i)]
		\item $\pi(a)U^u_{i,j}=\overset{dim(u)}{\underset{k=1}{\sum}}U^u_{i,k}\pi(\alpha^x_{k,j}(a))$, for all $u\in Rep(\mathbb{G})$, all $a\in A$ and all $i,j=1,\ldots, dim(u)$.
		\item $P=C^*\langle \pi(a)U^u_{i,j}: a\in A, u\in Rep(\mathbb{G}), i,j=1,\ldots, dim(u) \rangle$
		\item $E(\pi(a)U^u_{i,j})=\delta_{u,\epsilon}a$ for all $u\in Irr(\mathbb{G})$ and all $a\in A$.
	\end{enumerate} 
	
	In addition, $P$ is unique up to a canonical isomorphism meaning that for any $C^*$-algebra $Q$ with a triple $(\rho,V,E')$ where $\rho: A\longrightarrow Q$ is a non-degenerate $*$-homomorphism, $V\in M(c_0(\widehat{\mathbb{G}})\otimes Q)$ is a unitary representation and $E': Q\longrightarrow M(A)$ is a strict completely positive KSGNS-faithful map satisfying the analogous properties $(i), (ii)$ and $(iii)$ above, there exists a (necessarily unique) $*$-isomorphism $\psi: P\longrightarrow Q$ such that $\psi (\pi(a)U^u_{i,j})=\rho(a)V^u_{i,j}$, for all $u\in Rep(\mathbb{G})$, all $a\in A$ and all $i,j=1,\ldots, dim(u)$. Moreover, $E'$ is a non-degenerate map and we have $E=E'\circ \psi$.
	
	The $C^*$-algebra $P$ constructed in this way is called \emph{reduced crossed product of $A$ by $\widehat{\mathbb{G}}$} and is denoted by $\widehat{\mathbb{G}} \underset{\alpha,r}{\ltimes}A$.
\end{theodefiSubSub}
\begin{note}
	We should notice that the conditions in the preceding theorem can be written coordinate-free as follows, which gives a more common and conceptual statement:
	\begin{enumerate}[i)]
		\item $U^*(id\otimes \pi(a))U=(id\otimes \pi)\alpha(a)$, for all $a\in A$,
		\item $P=C^*(\pi(A)\phi_U(C_m(\mathbb{G})))$,
		\item $(id\otimes E)\big((1\otimes \pi(a))U\big)=p_{\epsilon}\otimes a$, for all $a\in A$.
	\end{enumerate}
\end{note}
\begin{proof}
	First of all, notice that the statement is proven once it is proven for any $x\in Irr(\mathbb{G})$. 
	
	$P$ will be a sub-$C^*$-algebra of $\mathcal{L}_A(L^2(\mathbb{G})\otimes A)$. For the non-degenerate $*$-homomorphism $\pi$ we consider the representation of $A$ on $L^2(\mathbb{G})\otimes A$ ``twisting'' by the action $\alpha$. Precisely, take the GNS representation $(L^2(\mathbb{G}), \widehat{\lambda}, \Omega)$ associated to the left Haar weight $\widehat{h}_L$ of $\widehat{\mathbb{G}}$. So we have that $\widehat{\lambda}\otimes id_A: c_0(\widehat{\mathbb{G}})\otimes A\longrightarrow \mathcal{L}_A(L^2(\mathbb{G})\otimes A)$ is a non-degenerate $*$-homomorphism. We define the non-degenerate $*$-homomorphism $\pi:A \longrightarrow \mathcal{L}_A(L^2(\mathbb{G})\otimes A)$ by $\pi(a):= (\widehat{\lambda}\otimes id_A)\circ \alpha (a)$, for all $a\in A$.

	For $U$ we consider the unitary representation of $\widehat{\mathbb{G}}$ on $L^2(\mathbb{G})\otimes A$ induced by $\lambda$. Precisely, take the unitary representation $\mathscr{V}:=\underset{x\in Irr(\mathbb{G})}{\bigoplus^{c_0}}w^x\in M(c_0(\widehat{\mathbb{G}})\otimes C_r(\mathbb{G}) )$ with $w^x\in \mathcal{B}(H_x)\otimes C(\mathbb{G})$, for all $x\in Irr(\mathbb{G})$. We define the unitary $U:= ( id_{c_0(\widehat{\mathbb{G}})}\otimes\lambda)(\mathscr{V})\otimes id_A\in M(c_0(\widehat{\mathbb{G}})\otimes \mathcal{L}_A(L^2(\mathbb{G})\otimes A))$.
	
	A straightforward computation yields the following expressions
	$$U^x=(id_{\mathcal{B}(H_x)}\otimes \lambda)\otimes w^x\otimes id_A\in \mathcal{B}(H_x)\otimes \mathcal{L}_A(L^2(\mathbb{G})\otimes A)$$ 
	$$U^x_{i,j}= \lambda(w^x_{i,j})\otimes id_A\in \mathcal{L}_A(L^2(\mathbb{G})\otimes A)$$
	for all $x\in Irr(\mathbb{G})$ and all $i,j=1,\ldots,dim(x)$.				

	In this situation, we can check the formula $\pi(a)U^x_{i,j}=\overset{n_x}{\underset{k=1}{\sum}}U^x_{i,k}\pi(\alpha^x_{k,j}(a))$, for all $x\in Irr(\mathbb{G})$, all $a\in A$ and all $i,j=1,\ldots, dim(x)$. Indeed,
	\begin{equation*}
		\begin{split}
			\overset{n_x}{\underset{k=1}{\sum}}U^x_{i,k}&\pi(\alpha^x_{k,j}(a))=\overset{n_x}{\underset{k=1}{\sum}}(\omega_{\xi^x_i,\xi^x_k}\otimes id)\Big(((id\otimes \lambda)(\mathscr{V})\otimes id_A) (p_x\otimes id)\Big)\big((\widehat{\lambda}\otimes id_A)\circ \alpha (\alpha^x_{k,j}(a))\big)\\
			&=\overset{n_x}{\underset{k=1}{\sum}}(\omega_{\xi^x_i,\xi^x_k}\otimes id)\Big(((\widehat{\lambda}\otimes \lambda)(\mathscr{V})\otimes id_A) (p_x\otimes id)\Big)\big(\alpha (\alpha^x_{k,j}(a))\big)\\
			&=\overset{n_x}{\underset{k=1}{\sum}}(\omega_{\xi^x_i,\xi^x_k}\otimes id)\Big(((\widehat{\lambda}\otimes \lambda)(\mathscr{V})\otimes id_A) (p_x\otimes id)\Big)\big(\alpha \big((\omega_{\xi^x_k, \xi^x_j}\otimes id)\alpha(a)(p_x\otimes id)\big)\big)\\
			&=\overset{n_x}{\underset{k=1}{\sum}}(\omega_{\xi^x_i,\xi^x_k}\otimes id)(\omega_{\xi^x_k, \xi^x_j}\otimes id)\Big((p_x\otimes id)((\widehat{\lambda}\otimes \lambda)(\mathscr{V})\otimes id_A) \Big)\big((id\otimes\alpha)\alpha(a)(p_x\otimes id)\big)\big)\\
			&\overset{(1)}{=}\overset{n_x}{\underset{k=1}{\sum}}(\omega_{\xi^x_i,\xi^x_k}\otimes id)(\omega_{\xi^x_k, \xi^x_j}\otimes id)\Big((p_x\otimes id)((\widehat{\lambda}\otimes \lambda)(\mathscr{V})\otimes id_A) \Big)\big((\widehat{\Delta}\otimes id)\alpha(a)(p_x\otimes id)\big)\big)\\
			&\overset{(2)}{=}\overset{n_x}{\underset{k=1}{\sum}}(\omega_{\xi^x_i,\xi^x_k}\otimes id)(\omega_{\xi^x_k, \xi^x_j}\otimes id)\Big((p_x\otimes id)(1\otimes \alpha(a))(\widehat{\lambda}\otimes \lambda)(\mathscr{V})\otimes id_A) \Big)(p_x\otimes id)\\
			&=(\omega_{\xi^x_i,\xi^x_j}\otimes id\otimes id)\Big((1\otimes \alpha(a))(\widehat{\lambda}\otimes \lambda)(\mathscr{V})\otimes id_A) \Big)(p_x\otimes id)\\
			&=(\widehat{\lambda}\otimes id)\alpha(a)(\omega_{\xi^x_i,\xi^x_j}\otimes id\otimes id)\Big((id\otimes \lambda)(\mathscr{V})\otimes id_A) \Big)(p_x\otimes id)=\pi(a)U^x_{ij}\mbox{,}
		\end{split}
	\end{equation*}
	where the equality $(1)$ holds because $\alpha$ is a left action of $\widehat{\mathbb{G}}$ on $A$ and the equality $(2)$ holds because of the definition of the co-multiplication $\widehat{\Delta}$ of $\widehat{\mathbb{G}}$ in terms of its fundamental unitary (observe that $\widehat{W}_{\mathbb{G}}=(id\otimes \lambda)(\mathscr{V})$). 
	
	Thus we define $P:=C^*\langle \pi(a)U^x_{i,j}: a\in A, x\in Irr(\mathbb{G}), i,j=1,\ldots, dim(x) \rangle$ which is a sub-$C^*$-algebra of $\mathcal{L}_A(L^2(\mathbb{G})\otimes A)$.
	
	To conclude the construction of $P$ as in the statement, we have to define a non-degenerate completely positive KSGNS-faithful map $E: P\longrightarrow M(A)=\mathcal{L}_A(A)$ satisfying the formula $E(\pi(a)U^x_{i,j})=a\delta_{x, \epsilon}$ for all $x\in Irr(\mathbb{G})$, all $a\in A$ and all $i,j=1,\ldots, dim(x)$. Namely, let us define the linear map $\Upsilon:A\longrightarrow  L^2(\mathbb{G})\otimes A$ by $\Upsilon(a):= \Omega\otimes a$, for all $a\in A$. It is actually an adjointable map between $A$ and $L^2(\mathbb{G})\otimes A$ whose adjoint is such that $\Upsilon^*(\lambda(w^x_{i,j})\Omega\otimes a)=h_{\mathbb{G}}(w^x_{i,j})a$, for all $x\in Irr(\mathbb{G})$, all $i,j=1,\ldots, dim(x)$ and all $a\in A$. Thus $E(X):=\Upsilon^*\circ X\circ \Upsilon$, for all $X\in P$ defines a completely positive map from $P$ into $M(A)$.
	
	We claim that the triple $(L^2(\mathbb{G})\otimes A, id, \Upsilon)$ is the KSGNS construction for $E$. We only have to prove that $L^2(\mathbb{G})\otimes A=\overline{span\{P\Upsilon(A)\}}$; but, by construction, it suffices to show that $\lambda(w^x_{i,j})\Omega\otimes a\in \overline{P\Upsilon(A)}$ for all $a\in A$, all $x\in Irr(\mathbb{G})$ and all $i,j=1,\ldots,dim(x)$; which is straightforward.

	Finally, an easy computation shows that the formula $E(\pi(a)U^x_{i,j})=a\delta_{x, \epsilon}$ holds for all $x\in Irr(\mathbb{G})$, all $a\in A$ and all $i,j=1,\ldots, dim(x)$. For, fix an orthonormal basis $\{\xi^x_1,\ldots,\xi^x_{n_x}\}$ of $H_{x}$ diagonalizing the canonical operator $Q_x$ with eigenvalues $\{\lambda^x_{j}\}_{j=1,\ldots, n_x}$, so that the formula $\lambda(w^x_{i,j})\Omega=\frac{\sqrt{\lambda^x_j}}{\sqrt{dim_q(x)}}\xi^x_i\otimes \omega^x_j$ holds for all $i,j=1,\ldots,n_x$ where $\{\omega^x_1,\ldots,\omega^x_{n_x}\}$ is the dual basis of $\{\xi^x_1,\ldots,\xi^x_{n_x}\}$ in the dual space $H_{\overline{x}}$. We write
	\begin{equation*}
		\begin{split}
			E(&\pi(a)U^x_{i,j})(b)=\Upsilon^*\big(\pi(a)U^x_{i,j}(\Upsilon(b))\big)=\Upsilon^*\big(\pi(a)U^x_{i,j}(\Omega\otimes b)\big)\\
			&=\Upsilon^*\big( \pi(a)(\lambda(w^x_{i,j})\otimes id_A)(\Omega\otimes b)\big)=\Upsilon^*\big( \pi(a)\big(\lambda(w^x_{i,j})\Omega\otimes b\big)\big)\\
			&=\Upsilon^*\Big( (\widehat{\lambda}\otimes id_A)\circ \alpha (a)\big(\lambda(w^x_{i,j})\Omega\otimes b\big)\Big) \\
			&=\Upsilon^*\Big( (p_x\otimes id_A)\Big[(\widehat{\lambda}\otimes id_A)\circ \alpha (a)\Big](p_x\otimes id_A)\big(\lambda(w^x_{i,j})\Omega\otimes b\big)\Big) \\
			&=\Upsilon^*\Big((p_x\otimes id_A)\Big[(\widehat{\lambda}\otimes id_A)\circ \alpha^x (a)\Big]\big(\lambda(w^x_{i,j})\Omega\otimes b\big)\Big)\\
			&=\Upsilon^*\Big( (p_x\otimes id_A)\Big[(\widehat{\lambda}\otimes id_A)\circ \overset{n_x}{\underset{i,j=1}{\sum}}m^x_{i,j}\otimes \alpha^x_{i,j}(a) \Big]\Big(\Big(\frac{\sqrt{\lambda^x_j}}{\sqrt{dim_q(x)}}\xi^x_i\otimes \omega^x_j\Big)\otimes b\Big)\Big)\\
			&=\Upsilon^*\Big(\overset{n_x}{\underset{i,j=1}{\sum}}\big(m^x_{i,j}\otimes id_{H_{\overline{x}}}\otimes \alpha^x_{i,j}(a) \big)\Big(\Big(\frac{\sqrt{\lambda^x_j}}{\sqrt{dim_q(x)}}\xi^x_i\otimes \omega^x_j\Big)\otimes b\Big)\Big)\\
			&=\Upsilon^*\Big(\frac{\sqrt{\lambda^x_j}}{\sqrt{dim_q(x)}}\overset{n_x}{\underset{i,j=1}{\sum}}\delta_{j,i}\xi^x_i\otimes \omega^x_j\otimes \alpha^x_{i,j}(a) b\Big)\\
			&=\Upsilon^*\Big(\frac{\sqrt{\lambda^x_i}}{\sqrt{dim_q(x)}}\xi^x_i\otimes \omega^x_i\otimes \alpha^x_{i,i}(a) b\Big)=\Upsilon^*\Big(\lambda(w^x_{i,i})\Omega\otimes \alpha^x_{i,i}(a)b\Big)\\
			&=h_{\mathbb{G}}(w^x_{i,i})\alpha^x_{i,i}(a)b=\alpha^x_{i,i}(a)b\delta_{x, \epsilon}=ab\delta_{x, \epsilon}\mbox{,}
		\end{split}
	\end{equation*}	
	where we use the orthogonality relations (and the definition of the KSGNS construction). Since it is true for all $b\in B$, we conclude the required formula.
	
	Observe that by KSGNS construction, $E$ is just a \emph{strict} completely positive map (see \cite{Lance} for the details). But, thanks to the property $E(\pi(a))=a$, for all $a\in A$ that we've just proved, it is clear that $E$ is actually a \emph{non-degenerate} completely positive map as assured in the statement.
	
	Next, let us establish the uniqueness of such a construction. Suppose $Q$ is another $C^*$-algebra with a triple $(\rho, V, E')$ where $\rho: A\longrightarrow Q$ is a non degenerate $*$-homomorphism, $V\in M(c_0(\widehat{\mathbb{G}})\otimes Q)$ is a unitary representation and $E': Q\longrightarrow M(A)$ is a strict completely positive KSGNS-faithful map satisfying the analogous properties $(i), (ii)$ and $(iii)$ of the triple $(\pi, U, E)$ associated to $P$. We have to show that there exists a (unique) $*$-isomorphism $\psi: P\longrightarrow Q$ such that $\psi (\pi(a)U^x_{i,j})=\rho(a)V^x_{i,j}\mbox{,}$ for all $x\in Irr(\mathbb{G})$, all $a\in A$ and all $i,j=1,\ldots, dim(x)$.
	
	Given the strict completely positive KSGNS-faithful maps $E: P\longrightarrow M(A)$ and $E': Q\longrightarrow M(A)$, consider their KSGNS constructions; say $(L^2(\mathbb{G})\otimes A, id, \Upsilon)$ and $(K,\sigma,\Upsilon')$, respectively. This means in particular that $L^2(\mathbb{G})\otimes A=\overline{span\{P\Upsilon(A)\}}$, $\sigma: Q\longrightarrow \mathcal{L}_A(K)$ is a non-degenerate faithful $*$-homomorphism such that $K=\overline{span\{\sigma(Q)\Upsilon'(A)\}}$ and that $E'(Y)=(\Upsilon')^*\circ \sigma(Y)\circ\Upsilon'$, for all all $Y\in Q$.
	
	Define a unitary operator $\mathscr{U}:L^2(\mathbb{G})\otimes A \longrightarrow K$. If such an operator exists, it must verify the formula $\mathscr{U}\big(X\Upsilon(b)\big)=\sigma(Y)\Upsilon'(b)$, for all $X=\pi(a) U^{x}_{i,j}\in P$, $Y= \rho(a) V^{x}_{i,j}\in Q$ and all $b\in A$.
	
	Actually, a straightforward computation shows that the formula above defines an isometry. Indeed, doing the identification $Q\cong\sigma(Q)$ (by virtue of the faithfulness of the KSGNS construction), let us take $X=\pi(a) U^{x}_{i,j}, X'=\pi(a') U^{x'}_{i,j}\in P$, $Y= \rho(a) V^{x}_{i,j}, \\Y'= \rho(a') V^{x'}_{i,j}\in Q$, $b, b'\in A$ and write
	\begin{equation*}
		\begin{split}
			\langle\mathscr{U}&\Big(X\Upsilon(b)\Big), \mathscr{U}\Big(X'\Upsilon(b')\Big)\rangle=\langle Y\Upsilon'(b), Y'\Upsilon'(b') \rangle\\
			&=\langle \rho(a) V^{x}_{i,j}\Upsilon'(b), \rho(a') V^{x'}_{i,j}\Upsilon'(b') \rangle=\langle \Upsilon'(b), (V^{x}_{i,j})^*\rho(a^*)\rho(a') V^{x'}_{i,j}\Upsilon'(b') \rangle\\
			&=\langle b, (\Upsilon')^*\big((V^{x}_{i,j})^*\rho(a^*a') V^{x'}_{i,j}\Upsilon'(b')\big) \rangle=\langle b, E'\big((V^{x}_{i,j})^*\rho(a^*a') V^{x'}_{i,j}\big)(b') \rangle\\
			&=\langle b, E'\big(V^{\overline{x}}_{i,j}\rho(a^*a') V^{x'}_{i,j}\big)(b') \rangle=\langle b, E'\Big(\overset{n_{x}}{\underset{k=1}{\sum}}\rho(\alpha^{x}_{j,k}(a^*a')) V^{\overline{x}}_{i,k}V^{x'}_{i,j}\Big)(b') \rangle\\
			&\overset{(1)}{=}\langle b, E'\Big(\Big(\overset{n_{x}}{\underset{k=1}{\sum}}\rho(\alpha^{x\otop\epsilon}_{j,k}(a^*a'))V^{\overline{x}\otop x'}_{r,t}\Big)(b') \rangle=\langle b, \underset{t}{\sum} \alpha^{x\otop\epsilon}_{j,k}(a^*a') \delta_{\overline{x}\otop x',\epsilon}\ b' \rangle\\
			&=\langle b, E\Big(\overset{n_{x}}{\underset{k=1}{\sum}}\pi(\alpha^{x\otop\epsilon}_{j,k}(a^*a'))U^{\overline{x}\otop x'}_{r,t}\Big)(b') \rangle\\
			&=\langle b, E\Big(\overset{n_{x}}{\underset{k=1}{\sum}}\pi(\alpha^{x}_{j,k}(a^*a'))U^{\overline{x}}_{i,k}\ U^{x'}_{i,j}\Big)(b') \rangle=\langle b, E\big(U^{\overline{x}}_{i,j}\pi(a^*a') U^{x'}_{i,j}\big)(b') \rangle\\
			&=\langle b, E\big((U^{x}_{i,j})^*\pi(a^*a') U^{x'}_{i,j}\big)(b') \rangle=\langle b, \Upsilon^*\big((U^{x}_{i,j})^*\pi(a^*a') U^{x'}_{i,j}\Upsilon(b')\big) \rangle\\
			&=\langle \Upsilon(b), (U^{x}_{i,j})^*\pi(a^*a') U^{x'}_{i,j}\Upsilon(b') \rangle=\langle \Upsilon(b), (U^{x}_{i,j})^*\pi(a^*)\pi(a') U^{x'}_{i,j}\Upsilon(b') \rangle\\
			&=\langle \pi(a) U^{x}_{i,j}\Upsilon(b), \pi(a') U^{x'}_{i,j}\Upsilon(b') \rangle=\langle X\Upsilon(b), X'\Upsilon(b') \rangle\mbox{,}\\
		\end{split}
	\end{equation*}
	where it should be noticed that in $(1)$ we use the index notation $r:= (i,i)$, $t:= (k,j)$ in order to write down properly the coefficients for the tensor product $\overline{x}\otop x'$.
	
	Doing again the identification $Q\cong\sigma(Q)$, we define
	$$
		\begin{array}{rccl}
			\psi:&P& \longrightarrow & Q\\
			&X & \longmapsto &\psi(X):= \mathscr{U}\circ X\circ \mathscr{U}^*
		\end{array}
	$$
	
	It is clear that $\psi$ is a $*$-isomorphism and the formula $\psi (\pi(a)U^x_{i,j})=\rho(a)V^x_{i,j}$ for all $x\in Irr(\mathbb{G})$, all $a\in A$ and all $i,j=1,\ldots, dim(x)$ is easily checked. 
	
	Moreover, by assumption we have $E'(\rho(a)V^x_{i,j})=a\delta_{\gamma,e}$ for all $x\in Irr(\mathbb{G})$, all $a\in A$ and all $i,j=1,\ldots, dim(x)$ and so $E'(\rho(a))=a$ for all $a\in A$; then it is clear that $E'$ is in fact a \emph{non-degenerate} map. Finally, the relation $E=E'\circ \psi$ holds by construction.
\end{proof}
	
	Applying the universal property of the preceding theorem, we get the following results.
	\begin{cor}\label{cor.DirectProductTensorProduct}
		Let $\mathbb{G}$, $\mathbb{H}$ be two compact quantum groups and let $\mathbb{F}:=\mathbb{G}\times \mathbb{H}$ be the corresponding quantum direct product of $\mathbb{G}$ and $\mathbb{H}$.
		
		If $(A,\alpha)$ is a left $\widehat{\mathbb{G}}$-$C^*$-algebra and $(B,\beta)$ is a left $\widehat{\mathbb{H}}$-$C^*$-algebra, then there exists a canonical $*$-isomorphism $\widehat{\mathbb{F}}\underset{\delta, r}{\ltimes} C\cong \widehat{\mathbb{G}}\underset{\alpha,r}{\ltimes} A\otimes \widehat{\mathbb{H}}\underset{\beta,r}{\ltimes} B$, where $C:=A\otimes B$ is the $\widehat{\mathbb{F}}$-$C^*$-algebra with action $\delta:=\alpha\otimes \beta$.
	\end{cor}
	\begin{note}
		By abuse of notation we denote by $\alpha\otimes\beta$ the composition \begin{equation*}
			\begin{split}
				A\otimes B\overset{\alpha\otimes\beta}{\longrightarrow}&\widetilde{M}(c_0(\widehat{\mathbb{G}})\otimes A)\otimes \widetilde{M}(c_0(\widehat{\mathbb{H}})\otimes B)\subset \widetilde{M}(c_0(\widehat{\mathbb{G}})\otimes A\otimes c_0(\widehat{\mathbb{H}})\otimes B)\\
				&\overset{\Sigma}{\cong} \widetilde{M}(c_0(\widehat{\mathbb{G}})\otimes c_0(\widehat{\mathbb{H}}) \otimes A\otimes B)= \widetilde{M}(c_0(\widehat{\mathbb{F}})\otimes A\otimes B)\mbox{,}
			\end{split}
		\end{equation*} which is a left action of $\widehat{\mathbb{F}}$ on $A\otimes B$.
	\end{note}
	
	\begin{cor}\label{pro.CrossedProductTensorProduct}
		Let $\mathbb{G}$ be a compact quantum group. If $(A,\alpha)$ is a left $\widehat{\mathbb{G}}$-$C^*$-algebra and $B$ is any $C^*$-algebra, then we have a canonical $*$-isomorphism $\widehat{\mathbb{G}}\underset{id\otimes\alpha, r}{\ltimes} (B\otimes A)\cong B\otimes \widehat{\mathbb{G}}\underset{{\alpha, r}}{\ltimes} A$, where $id\otimes \alpha: B\otimes A\longrightarrow \widetilde{M}(c_0(\widehat{\mathbb{G}})\otimes B\otimes A)$ denotes, by abuse of notation, the action given by the composition $(\Sigma_{12}\otimes id_A)(id_B\otimes \alpha)$.
	\end{cor}

	\begin{rem}\label{rem.Functoriality}
		Let $\mathbb{G}$ be a compact quantum group, $(A,\alpha)$, $(B,\beta)$ two left $\widehat{\mathbb{G}}$-$C^*$-algebras and $\varphi: A\longrightarrow B$ a $\widehat{\mathbb{G}}$-equivariant $*$-homomorphism. 
		
		On the one hand, there exists a $*$-homomorphism $\mathcal{Z}(\varphi):= id\ltimes \varphi: \widehat{\mathbb{G}}\underset{\alpha,r}{\ltimes} A\longrightarrow \widehat{\mathbb{G}}\underset{\beta,r}{\ltimes} B$ such that $\mathcal{Z}(\varphi)(\pi_{\alpha}(a)(U^\alpha)^{x}_{i,j})=\pi_{\beta}(\varphi(a))(U^\beta)^x_{i,j}$, for all $a\in A$, all $x\in Irr(\mathbb{G})$ and all $i,j=1,\ldots, dim(x)$ where $(\pi_{\alpha}, U^{\alpha}, E_{\alpha})$ and $(\pi_{\beta}, U^{\beta}, E_{\beta})$ are the canonical triples associated to the reduced crossed products $\widehat{\mathbb{G}}\underset{\alpha,r}{\ltimes} A$ and $\widehat{\mathbb{G}}\underset{\beta,r}{\ltimes} B$, respectively.
		
		The $*$-homomorphism $\mathcal{Z}(\varphi)$ above is nothing but the restriction of $\mathcal{L}_{A}(L^2(\mathbb{G})\otimes A) \longrightarrow \mathcal{L}_{B}(L^2(\mathbb{G})\otimes B)$ defined by $T \longmapsto\mathscr{U}_{\varphi}(T\otimes_{\varphi} id_B)\mathscr{U}_{\varphi}^{-1}$, for all $T\in \mathcal{L}_{A}(L^2(\mathbb{G})\otimes A)$, where $\mathscr{U}_{\varphi}:L^2(\mathbb{G})\otimes A\otimes_{\varphi} B\overset{\sim}{\longrightarrow} L^2(\mathbb{G})\otimes B$ is the canonical isometry of Hilbert modules such that $\mathscr{U}_{\varphi}(\xi\otimes a\otimes_{\varphi} b)=\xi\otimes \varphi(a)b$, for all $\xi\in L^2(\mathbb{G})$, all $a\in A$ and all $b\in B$.
		
		Finally, observe that $\mathcal{Z}(\varphi)=id\ltimes\varphi$ is, by construction, compatible with the elements of the canonical triples in the following sense
		$$\mathcal{Z}(\varphi)(\pi_{\alpha}(a))=\pi_{\beta}(\varphi(a))\mbox{, }\mathcal{Z}(\varphi)((U^\alpha)^{x}_{i,j})=(U^\beta)^{x}_{i,j}\mbox{, }E_{\beta}\circ\mathcal{Z}(\varphi)=E_{\alpha}\circ\varphi\mbox{,}$$
		for all $a\in A$, all $x\in Irr(\mathbb{G})$ and all $i,j=1,\ldots, dim(x)$.
		
		On the other hand, the cone $C_\varphi$ is a $\widehat{\mathbb{G}}$-$C^*$-algebra with action 
		$$
			\begin{array}{rccl}
				\delta:&C_{\varphi}& \longrightarrow &M(c_0(\widehat{\mathbb{G}})\otimes C_{\varphi})\\
				&(a,h) & \longmapsto &\delta(a,h):=(\alpha(a), \beta\circ h)
			\end{array}
		$$
		
		Given an irreducible representation $x\in Irr(\mathbb{G})$, the matrix coefficients of $\delta^x$ with respect to an orthonormal basis of $H_x$ are given by $\delta^{x}_{i,j}(a,h):=\big(\alpha^{x}_{i,j}(a),\beta^{x}_{i,j}\circ h\big)\in C_{\varphi}$, for all $(a,h)\in C_{\varphi}$ and all $i,j=1,\ldots, n_x$.
	\end{rem}
	
	The following is a straightforward but useful result for our purpose.
	\begin{pro}\label{pro.ConesTensorProduct}
		Let $\mathbb{G}$ and $\mathbb{H}$ be two compact quantum groups and put $\mathbb{F}:=\mathbb{G}\times \mathbb{H}$. If $A_0$ is a $\widehat{\mathbb{G}}$-$C^*$-algebra and $\phi: B\longrightarrow B'$ is a $\widehat{\mathbb{H}}$-equivariant $*$-homomorphism, then there exists a canonical $\widehat{\mathbb{F}}$-equivariant $*$-isomorphism $A_0\otimes C_{\phi}\cong C_{id\otimes\phi}$, where $C_{\phi}$ denotes the cone of the $*$-homomorphism $\phi$ and $C_{id\otimes\phi}$ the cone of the induced $*$-homomorphism $id_{A_0}\otimes\phi: A_0\otimes B\longrightarrow A_0\otimes B'$.
	\end{pro}
	
	\begin{pro}\label{pro.ConesCrossedProduct}
		Let $\mathbb{G}$ be a compact quantum group and $(A,\alpha)$, $(B,\beta)$ two left $\widehat{\mathbb{G}}$-$C^*$-algebras. If $\varphi: A\longrightarrow B$ is any $\widehat{\mathbb{G}}$-equivariant $*$-homomorphism, then there exists a canonical $*$-isomorphism $\widehat{\mathbb{G}}\underset{r}{\ltimes} C_{\varphi}\cong C_{id\ltimes\varphi}$, where $C_{\varphi}$ denotes the cone of the $*$-homomorphism $\varphi$ and $C_{id\ltimes\varphi}$ the cone of the induced $*$-homomorphism $id\ltimes \varphi: \widehat{\mathbb{G}}\underset{\alpha,r}{\ltimes} A\longrightarrow \widehat{\mathbb{G}}\underset{\beta,r}{\ltimes} B$.
	\end{pro}
	\begin{proof}
			First, recall the definitions of our cones: $C_{\varphi}:=\{(a,h)\in A\times C_0\big((0,1], B\big)\ |\ \varphi(a)=h(1)\}$ and \\$C_{id\ltimes\varphi}:=\{(X, \tilde{h})\in \widehat{\mathbb{G}}\underset{\alpha,r}{\ltimes} A\times C_0\big((0,1], \widehat{\mathbb{G}}\underset{\beta,r}{\ltimes} B\big)\ |\ id\ltimes\varphi(X)=\tilde{h}(1)\}$. Observe that if $(A,\alpha)$, $(B,\beta)$ are left $\widehat{\mathbb{G}}$-$C^*$-algebras, then $(C_{\varphi},\delta)$ is again a left $\widehat{\mathbb{G}}$-$C^*$-algebra in the obvious way.
			
			In order to show the canonical $*$-isomorphism $\widehat{\mathbb{G}}\underset{r}{\ltimes} C_{\varphi}\cong C_{id\ltimes\varphi}$, we are going to show that the $C^*$-algebra $C_{id\ltimes\varphi}$ satisfies the universal property of the reduced crossed product $\widehat{\mathbb{G}}\underset{r}{\ltimes} C_{\varphi}$. To do so, we have to define a triple $(\overline{\rho}, \overline{V}, \overline{E})$ associated to $C_{id\ltimes\varphi}$ in the sense of Theorem \ref{theo.QuantumReducedCrossedProduct}. 
			
			Given the reduced crossed products $\widehat{\mathbb{G}}\underset{\alpha,r}{\ltimes} A$ and $\widehat{\mathbb{G}}\underset{\beta, r}{\ltimes} B$, consider the corresponding canonical associated triples $(\pi_{\alpha}, U^{\alpha}, E_{\alpha})$ and $(\pi_{\beta}, U^{\beta}, E_{\beta})$, respectively; and define the non-degenerate $*$-homomorphism $\overline{\rho}:C_{\varphi}\longrightarrow C_{id\ltimes\varphi}$ by $\overline{\rho}(a,h):=(\pi_{\alpha}(a), \pi_{\beta}\circ h)$, for all $(a,h)\in C_{\varphi}$; the unitary representation $\overline{V}\in M(c_0(\widehat{\mathbb{G}})\otimes C_{id\ltimes\varphi})$ is defined by the non-degenerate $*$-homomorphism $\phi_{\overline{V}}: C_{m}(\mathbb{G})\longrightarrow M(C_{id\ltimes\varphi})$ as $\phi_{\overline{V}}(c):=\big(\phi_{U^{\alpha}}(c)\cdot\ , \phi_{U^{\beta}}(c)\cdot\ \big)$, for all $c\in C_m(\mathbb{G})$	and observe that by construction we have $\overline{V}^{x}_{i,j}=\big((U^{\alpha})^{x}_{i,j}\cdot\ , (U^{\beta})^x_{i,j}\cdot\ \big)\in M(C_{id\ltimes\varphi})$, for all $x\in Irr(\mathbb{G})$ and all $i,j=1,\ldots,dim(x)$. The strict completely positive KSGNS-faithful map $\overline{E}:C_{id\ltimes\varphi}\longrightarrow M(C_{\varphi})=\mathcal{L}_{C_{\varphi}}(C_{\varphi})$ is defined by $\overline{E}(X,\tilde{h}):=\big(E_{\alpha}(X)\cdot\ , E_{\beta}\circ \tilde{h}\cdot\ \big)$, for all $(X,\tilde{h})\in C_{id\ltimes\varphi}$.			
	
			To conclude the proof we have to check the following
			\begin{enumerate}[i)]
				\item $\overline{\rho}(a,h)\overline{V}^x_{i,j}=\overset{dim(x)}{\underset{k=1}{\sum}}\overline{V}^x_{i,k}\overline{\rho}(\delta^x_{k,j}(a,h))$, for all $(a,h)\in C_{\varphi}$, all $x\in Irr(\mathbb{G})$ and all $i,j=1,\ldots, dim(x)$; which is a routine computation.
				\item $\overline{E}$ is always a KSGNS-faithful map such that $\overline{E}(\overline{\rho}(a,h)\overline{V}^{x}_{i,j})=(a,h)\delta_{x,\epsilon}$, for all $(a,h)\in C_{\varphi}$, all $x\in Irr(\mathbb{G})$ and all $i,j=1,\ldots,dim(x)$. The formula is straightforward and concerning the KSGNS-faithfulness we are going to exhibit directly the KSGNS-construction for our $\overline{E}:C_{id\ltimes\varphi}\longrightarrow M(C_{\varphi})=\mathcal{L}_{C_{\varphi}}(C_{\varphi})$. In order to do so, recall that $(L^2(\mathbb{G})\otimes A, id, \Upsilon_{\alpha})$ and $(L^2(\mathbb{G})\otimes B, id, \Upsilon_{\beta})$ are the KSGNS constructions for $E_{\alpha}$ and $E_{\beta}$, respectively. First, we need an appropriated Hilbert $C_{\varphi}$-module. Let us take
				$$\mathscr{H}:=\{(\xi,\eta)\in L^2(\mathbb{G})\otimes A\times C_0\big((0,1], L^2(\mathbb{G})\otimes B\big)\ |\ \eta(1)=\mathscr{U}_{\varphi}(\xi\otimes_{\varphi} id_B)\}\mbox{,}$$
				where $\mathscr{U}_{\varphi}$ is the canonical isometry between $L^2(\mathbb{G})\otimes A\otimes_{\varphi}B$ and $L^2(\mathbb{G})\otimes B$ of Remark \ref{rem.Functoriality} above.
						
				Next, consider the adjointable operator $\Upsilon:C_{\varphi} \longrightarrow \mathscr{H}$ defined by \\$\Upsilon(a,h):=\big(\Upsilon_{\alpha}(a), \Upsilon_{\beta}\circ h\big)$, for all $(a,h)\in C_{\varphi}$ and the representation $\sigma$ (faithful, thanks to the faithfulness of the KSGNS constructions of $E_{\alpha}$ and $E_{\beta}$) of $C_{id\ltimes\varphi}$ on $\mathscr{H}$ given by $\sigma(X,\tilde{h}):=\big(X\cdot\ , \tilde{h}\cdot\ \big)$, for all $(X,\tilde{h})\in C_{id\ltimes\varphi}$.				
				
				In this way, it is easy to check that the triple $(\mathscr{G}, \sigma, \Upsilon)$ with $\mathscr{G}:=\overline{span\{\sigma(C_{id\ltimes\varphi})\Upsilon(C_{\varphi})\}}$ is the KSGNS construction of our $\overline{E}$ and then $\overline{E}$ is KSGNS-faithful (observe by the way that our $\overline{E}$ above is defined exactly through $\Upsilon$ by construction).
		\end{enumerate}
	\end{proof}

	\subsection{The Baum-Connes property for discrete quantum groups}\label{sec.BCDiscreteQuantumGroups}
		We collect here the main categorical framework for the formulation of the Baum-Connes property for \emph{torsion-free} discrete quantum groups. We refer to \cite{MeyerNest} or \cite{Jorgensen} for a complete presentation of the subject.
	
	Let $\widehat{\mathbb{G}}$ be a discrete quantum group and consider the corresponding equivariant Kasparov category, $\mathscr{K}\mathscr{K}^{\widehat{\mathbb{G}}}$, with canonical suspension functor denoted by $\Sigma$. The word \emph{homomorphism (resp., isomorphism)} will mean \emph{homomorphism (resp., isomorphism) in the corresponding Kasparov category}; it will be a true homomorphism (resp., isomorphism) between $C^*$-algebras or any Kasparov triple between $C^*$-algebras (resp., any $KK$-equivalence between $C^*$-algebras). From now on, in order to formulate the Baum-Connes property for a discrete quantum group, we assume that $\widehat{\mathbb{G}}$ is \emph{torsion-free}. In that case, consider the usual complementary pair of localizing subcategories in $\mathscr{K}\mathscr{K}^{\widehat{\mathbb{G}}}$, $(\mathscr{L}_{\widehat{\mathbb{G}}}, \mathscr{N}_{\widehat{\mathbb{G}}})$. Denote by $(L,N)$ the canonical triangulated functors associated to this complementary pair. More precisely we have that $\mathscr{L}_{\widehat{\mathbb{G}}}$ is defined as the \emph{localizing subcategory of $\mathscr{K}\mathscr{K}^{\widehat{\mathbb{G}}}$ generated by the objects of the form $Ind^{\widehat{\mathbb{G}}}_{\mathbb{E}}(C)=c_0(\widehat{\mathbb{G}})\otimes C$ with $C$ any $C^*$-algebra in the Kasparov category $\mathscr{K}\mathscr{K}$} and $\mathscr{N}_{\widehat{\mathbb{G}}}$ is defined as the \emph{localizing subcategory of objects which are isomorphic to $0$ in $\mathscr{K}\mathscr{K}$}.
	$$\mathscr{L}_{\widehat{\mathbb{G}}}:=\langle\{Ind^{\widehat{\mathbb{G}}}_{\mathbb{E}}(C)=c_0(\widehat{\mathbb{G}})\otimes C\ |\ C\in Obj.(\mathscr{K}\mathscr{K})\}\rangle$$
	$$\mathscr{N}_{\widehat{\mathbb{G}}}=\{A\in Obj.(\mathscr{K}\mathscr{K}^{\widehat{\mathbb{G}}})\ |\ Res^{\widehat{\mathbb{G}}}_{\mathbb{E}}(A)=0\}=\{A\in Obj.(\mathscr{K}\mathscr{K}^{\widehat{\mathbb{G}}})\ |\ L(A)=0\}$$
	
	\begin{note}
		The following nomenclature is useful. Given $A\in Obj.(\mathscr{K}\mathscr{K}^{\widehat{\mathbb{G}}})$ consider a $(\mathscr{L}_{\widehat{\mathbb{G}}}, \mathscr{N}_{\widehat{\mathbb{G}}})$-triangle associated to $A$, say $\Sigma N(A)\longrightarrow L(A)\overset{D}{\longrightarrow} A\longrightarrow N(A)$. We know that such triangles are distinguished and unique up to isomorphism. The homomorphism $D:L(A)\longrightarrow A$ is called \emph{Dirac homomorphism for $A$}. In particular, we consider the Dirac homomorphism for $\mathbb{C}$ (as trivial $\widehat{\mathbb{G}}$-$C^*$-algebra), $D_{\mathbb{C}}:L(\mathbb{C})\longrightarrow \mathbb{C}$. We refer to $D_{\mathbb{C}}$ simply as \emph{Dirac homomorphism}.
		
		More generally, if $(\mathcal{T},\Sigma)$ is any triangulated category and $(\mathscr{L}_{\mathcal{T}}, \mathscr{N}_{\mathcal{T}})$ is a complementary pair of localizing subcategories in $\mathcal{T}$ with canonical functors $(L,N)$, then for all object $X\in Obj.(\mathcal{T})$ there exists a distinguished triangle (unique up to isomorphism) of the form $\Sigma(N(X))\longrightarrow L(X)\overset{u}{\longrightarrow} X\longrightarrow N(X)$. By abuse of language, the homomorphism $u$ is called \emph{Dirac homomorphism for $X$}. Given a functor $F:\mathcal{T}\longrightarrow \mathscr{C}$, where $\mathscr{C}$ is some category, we define $\mathbb{L}F:=F\circ L$ and $\eta_X:=F(u)$, for all $X\in Obj.(\mathcal{T})$. The latter yields a natural transformation $\eta: \mathbb{L}F\longrightarrow F$.
	\end{note}
	
	Finally, consider the homological functor defining the \emph{quantum} Baum-Connes assembly map for $\widehat{\mathbb{G}}$,
	$$
		\begin{array}{rccl}
			F:&\mathscr{K}\mathscr{K}^{\widehat{\mathbb{G}}}& \longrightarrow &\mathscr{A}b^{\mathbb{Z}/2}\\
			&(A,\delta) & \longmapsto &F(A):=K_{*}(\widehat{\mathbb{G}}\underset{\delta, r}{\ltimes} A)
		\end{array}
	$$
	
	The quantum assembly map for $\widehat{\mathbb{G}}$ is given by the natural transformation $\eta^{\widehat{\mathbb{G}}}: \mathbb{L}F\longrightarrow F$. 
	\begin{defi}
		Let $\widehat{\mathbb{G}}$ be a torsion-free discrete quantum group. 
		\begin{itemize}
			\item[-] We say that $\widehat{\mathbb{G}}$ satisfies the quantum Baum-Connes property (with coefficients) if the natural transformation $\eta^{\widehat{\mathbb{G}}}: \mathbb{L}F\longrightarrow F$ is a natural equivalence.
		
			\item[-] We say that $\widehat{\mathbb{G}}$ satisfies the \emph{strong} Baum-Connes property if $\mathscr{K}\mathscr{K}^{\widehat{\mathbb{G}}}=\mathscr{L}_{\widehat{\mathbb{G}}}$.
		\end{itemize}
	\end{defi}
	
	\bigskip
	To the best knowledge of the author it is open to know if the Baum-Connes property is preserved by quantum subgroups in general. However, we can show that it is preserved by \emph{divisible} discrete quantum subgroups. Let $\widehat{\mathbb{H}}<\widehat{\mathbb{G}}$ be any discrete quantum subgroup of $\widehat{\mathbb{G}}$. We have two relevant functors: restriction, which is obvious, and induction, which has been studied by S. Vaes in \cite{VaesInduction} in the framework of quantum groups.
	$$Res^{\widehat{\mathbb{G}}}_{\widehat{\mathbb{H}}}:\mathscr{K}\mathscr{K}^{\widehat{\mathbb{G}}}\longrightarrow \mathscr{K}\mathscr{K}^{\widehat{\mathbb{H}}}\mbox{ and }Ind^{\widehat{\mathbb{G}}}_{\widehat{\mathbb{H}}}:\mathscr{K}\mathscr{K}^{\widehat{\mathbb{H}}}\longrightarrow \mathscr{K}\mathscr{K}^{\widehat{\mathbb{G}}}$$
	
	It is well-known that restriction and induction are \emph{triangulated} functors by virtue of the universal property of the Kasparov category (see \cite{VoigtPoincareDuality} for more details). Moreover, they are adjoint in the sense that $KK^{\widehat{\mathbb{G}}}(Ind^{\widehat{\mathbb{G}}}_{\widehat{\mathbb{H}}}(B), A)\cong KK^{\widehat{\mathbb{H}}}(B, Res^{\widehat{\mathbb{G}}}_{\widehat{\mathbb{H}}}(A))$, for all $\widehat{\mathbb{G}}$-$C^*$-algebra $A$ and all $\widehat{\mathbb{H}}$-$C^*$-algebra $B$ (see \cite{VoigtBaumConnesUnitaryFree} for a proof). Denote by $(L',N')$ the canonical triangulated functors associated to the complementary pair $(\mathscr{L}_{\widehat{\mathbb{H}}}, \mathscr{N}_{\widehat{\mathbb{H}}})$.
	\begin{lem}\label{lem.ResIndDirac}
		Let $\mathbb{G}$, $\mathbb{H}$ be two compact quantum groups. If $\widehat{\mathbb{G}}$ is torsion-free and $\widehat{\mathbb{H}}<\widehat{\mathbb{G}}$ is a divisible torsion-free discrete quantum subgroup, then the following properties hold.
		\begin{enumerate}[i)]
			\item $Res^{\widehat{\mathbb{G}}}_{\widehat{\mathbb{H}}}(\mathscr{L}_{\widehat{\mathbb{G}}})\subset \mathscr{L}_{\widehat{\mathbb{H}}}$ and $Res^{\widehat{\mathbb{G}}}_{\widehat{\mathbb{H}}}(\mathscr{N}_{\widehat{\mathbb{G}}})\subset \mathscr{N}_{\widehat{\mathbb{H}}}$. Hence, we have the following natural isomorphisms $Res^{\widehat{\mathbb{G}}}_{\widehat{\mathbb{H}}}\circ L\cong L'\circ Res^{\widehat{\mathbb{G}}}_{\widehat{\mathbb{H}}}$ and $Res^{\widehat{\mathbb{G}}}_{\widehat{\mathbb{H}}}\circ N\cong N'\circ Res^{\widehat{\mathbb{G}}}_{\widehat{\mathbb{H}}}$.
			\item $Ind^{\widehat{\mathbb{G}}}_{\widehat{\mathbb{H}}}(\mathscr{L}_{\widehat{\mathbb{H}}})\subset \mathscr{L}_{\widehat{\mathbb{G}}}$ and $Ind^{\widehat{\mathbb{G}}}_{\widehat{\mathbb{H}}}(\mathscr{N}_{\widehat{\mathbb{H}}})\subset \mathscr{N}_{\widehat{\mathbb{G}}}$. Hence, we have the following natural isomorphisms $Ind^{\widehat{\mathbb{G}}}_{\widehat{\mathbb{H}}}\circ L'\cong L\circ Ind^{\widehat{\mathbb{G}}}_{\widehat{\mathbb{H}}}$ and $Ind^{\widehat{\mathbb{G}}}_{\widehat{\mathbb{H}}}\circ N'\cong N\circ Ind^{\widehat{\mathbb{G}}}_{\widehat{\mathbb{H}}}$.
		\end{enumerate}
		
		Consequently, $Res^{\widehat{\mathbb{G}}}_{\widehat{\mathbb{H}}}$ transforms the assembly map for $\widehat{\mathbb{G}}$ into the assembly map for $\widehat{\mathbb{H}}$ and $Ind^{\widehat{\mathbb{G}}}_{\widehat{\mathbb{H}}}$ transforms the assembly map for $\widehat{\mathbb{H}}$ into the assembly map for $\widehat{\mathbb{G}}$.
	\end{lem}
	\begin{proof}
		\begin{enumerate}[i)]
			\item Since $\widehat{\mathbb{H}}$ is divisible in $\widehat{\mathbb{G}}$, then $c_0(\widehat{\mathbb{G}})=c_0(\widehat{\mathbb{H}})\otimes c_0(\widehat{\mathbb{H}}\backslash \widehat{\mathbb{G}})$ as $\widehat{\mathbb{H}}$-$C^*$-algebras. Hence, it is clear that $Res^{\widehat{\mathbb{G}}}_{\widehat{\mathbb{H}}}(\mathscr{L}_{\widehat{\mathbb{G}}})\subset \mathscr{L}_{\widehat{\mathbb{H}}}$. Take $N\in\mathscr{N}_{\widehat{\mathbb{G}}}$, then we have that $Res^{\widehat{\mathbb{G}}}_{\mathbb{E}}(N)=0$. Restriction by stages yields that $0=Res^{\widehat{\mathbb{G}}}_{\mathbb{E}}(N)=Res^{\widehat{\mathbb{H}}}_{\mathbb{E}}\big(Res^{\widehat{\mathbb{G}}}_{\widehat{\mathbb{H}}}(N)\big)$, which means that $Res^{\widehat{\mathbb{G}}}_{\widehat{\mathbb{H}}}(N)\in\mathscr{N}_{\widehat{\mathbb{H}}}$.
			
				Given any $A\in Obj.(\mathscr{K}\mathscr{K}^{\widehat{\mathbb{G}}})$, its corresponding $(\mathscr{L}_{\widehat{\mathbb{G}}}, \mathscr{N}_{\widehat{\mathbb{G}}})$-triangle is transformed into a distinguished triangle by restriction because $Res^{\widehat{\mathbb{G}}}_{\widehat{\mathbb{H}}}$ is triangulated. We have just seen that restriction functor preserves the subcategories $\mathscr{L}$ and $\mathscr{N}$. Hence the distinguished triangle given by restriction is actually a $(\mathscr{L}_{\widehat{\mathbb{H}}}, \mathscr{N}_{\widehat{\mathbb{H}}})$-triangle for $Res^{\widehat{\mathbb{G}}}_{\widehat{\mathbb{H}}}(A)$. By uniqueness of these distinguished triangles we get the relations.
			\item Take a generator $Ind^{\widehat{\mathbb{H}}}_{\mathbb{E}}(C)\in\mathscr{L}_{\widehat{\mathbb{H}}}$ with $C\in Obj(\mathscr{K}\mathscr{K})$. Induction by stages (see Proposition $2.7$ in \cite{VoigtPoincareDuality} for a proof) yields that $ Ind^{\widehat{\mathbb{G}}}_{\widehat{\mathbb{H}}}\big(Ind^{\widehat{\mathbb{H}}}_{\mathbb{E}}(C)\big)=Ind^{\widehat{\mathbb{G}}}_{\mathbb{E}}(C)$, which is again a generator in $\mathscr{L}_{\widehat{\mathbb{G}}}$. Hence, we also have $Ind^{\widehat{\mathbb{G}}}_{\widehat{\mathbb{H}}}(\mathscr{L}_{\widehat{\mathbb{H}}})\subset \mathscr{L}_{\widehat{\mathbb{G}}}$. 
			
			Take $N'\in \mathscr{N}_{\widehat{\mathbb{H}}}$. Recall that, since $\mathscr{L}_{\widehat{\mathbb{H}}}$ and $\mathscr{N}_{\widehat{\mathbb{H}}}$ are complementary, then we have $\mathscr{N}_{\widehat{\mathbb{H}}}=\mathscr{L}^{\dashv}_{\widehat{\mathbb{H}}}$. Accordingly, $KK^{\widehat{\mathbb{H}}}(L', N')=(0)$, for all $L'\in \mathscr{L}_{\widehat{\mathbb{H}}}$. By virtue of property $(i)$ above, we can take $L':=Res^{\widehat{\mathbb{G}}}_{\widehat{\mathbb{H}}}(L)$ for any $L\in \mathscr{L}_{\widehat{\mathbb{G}}}$. Hence, the adjointness property between restriction and induction functor yields that $KK^{\widehat{\mathbb{G}}}(L, Ind^{\widehat{\mathbb{G}}}_{\widehat{\mathbb{H}}}(N'))=KK^{\widehat{\mathbb{H}}}(Res^{\widehat{\mathbb{G}}}_{\widehat{\mathbb{H}}}(L), N')=(0)$, for all $L\in \mathscr{L}_{\widehat{\mathbb{G}}}$, which means that $Ind^{\widehat{\mathbb{G}}}_{\widehat{\mathbb{H}}}(N')\in \mathscr{L}^{\dashv}_{\widehat{\mathbb{G}}}=\mathscr{N}_{\widehat{\mathbb{G}}}$.
			
				Given any $B\in Obj.(\mathscr{K}\mathscr{K}^{\widehat{\mathbb{H}}})$, its corresponding $(\mathscr{L}_{\widehat{\mathbb{H}}}, \mathscr{N}_{\widehat{\mathbb{H}}})$-triangle is transformed into a distinguished triangle by induction because $Ind^{\widehat{\mathbb{G}}}_{\widehat{\mathbb{H}}}$ is triangulated. We have just seen that induction functor preserves the subcategories $\mathscr{L}$ and $\mathscr{N}$. Hence the distinguished triangle given by induction is actually a $(\mathscr{L}_{\widehat{\mathbb{G}}}, \mathscr{N}_{\widehat{\mathbb{G}}})$-triangle for $Ind^{\widehat{\mathbb{G}}}_{\widehat{\mathbb{H}}}(B)$. By uniqueness of these distinguished triangles we get the relations.
		\end{enumerate}
	\end{proof}
	\begin{pro}\label{pro.BCDivisibleQuantumSubgroups}
		Let $\mathbb{G}$, $\mathbb{H}$ be two compact quantum groups. Assume that $\widehat{\mathbb{G}}$ is torsion-free. $\widehat{\mathbb{G}}$ satisfies the quantum Baum-Connes property if and only if every divisible torsion-free discrete quantum subgroup $\widehat{\mathbb{H}}<\widehat{\mathbb{G}}$ satisfies the quantum Baum-Connes property.
	\end{pro}
	\begin{proof}
		Assume that $\widehat{\mathbb{G}}$ satisfies the quantum Baum-Connes property and consider a divisible torsion-free discrete quantum subgroup $\widehat{\mathbb{H}}<\widehat{\mathbb{G}}$.
		
		By assumption, $\widehat{\mathbb{G}}$ satisfies the quantum Baum-Connes property with coefficients. In particular, we have a natural isomorphism $\eta^{\widehat{\mathbb{G}}}_{Ind^{\widehat{\mathbb{G}}}_{\widehat{\mathbb{H}}}(B)}: K_*\big(\widehat{\mathbb{G}}\ltimes L(Ind^{\widehat{\mathbb{G}}}_{\widehat{\mathbb{H}}}(B))\big)\longrightarrow K_*\big(\widehat{\mathbb{G}}\ltimes Ind^{\widehat{\mathbb{G}}}_{\widehat{\mathbb{H}}}(B)\big)$, for all $B\in Obj(\mathscr{K}\mathscr{K}^{\widehat{\mathbb{H}}})$.
		
		Thanks to the preceding lemma $Ind^{\widehat{\mathbb{G}}}_{\widehat{\mathbb{H}}}\circ L'\cong L\circ Ind^{\widehat{\mathbb{G}}}_{\widehat{\mathbb{H}}}$, so that we have a natural isomorphism $\eta^{\widehat{\mathbb{G}}}_{Ind^{\widehat{\mathbb{G}}}_{\widehat{\mathbb{H}}}(B)}: K_*\big(\widehat{\mathbb{G}}\ltimes Ind^{\widehat{\mathbb{G}}}_{\widehat{\mathbb{H}}}(L'(B)\big)\longrightarrow K_*\big(\widehat{\mathbb{G}}\ltimes Ind^{\widehat{\mathbb{G}}}_{\widehat{\mathbb{H}}}(B)\big)$, for all $B\in Obj(\mathscr{K}\mathscr{K}^{\widehat{\mathbb{H}}})$.
		
		By virtue of the quantum Green's Imprimitivity theorem (see Theorem $7.3$ in \cite{VaesInduction} for a proof) we have a natural Morita equivalence $\widehat{\mathbb{G}}\ltimes Ind^{\widehat{\mathbb{G}}}_{\widehat{\mathbb{H}}}(B)\underset{M}{\sim} \widehat{\mathbb{H}}\ltimes B$ for all $B\in Obj(\mathscr{K}\mathscr{K}^{\widehat{\mathbb{H}}})$, which yields an isomorphism between $\widehat{\mathbb{G}}\ltimes Ind^{\widehat{\mathbb{G}}}_{\widehat{\mathbb{H}}}(B)$ and $\widehat{\mathbb{H}}\ltimes B$ in $\mathscr{K}\mathscr{K}$.
		
		Moreover, the induction functor transforms the assembly map for $\widehat{\mathbb{H}}$ into the assembly map for $\widehat{\mathbb{G}}$ by the preceding lemma. More precisely, given $B\in Obj(\mathscr{K}\mathscr{K}^{\widehat{\mathbb{H}}})$ if $\Sigma(N'(B))\longrightarrow L'(B)\overset{u'}{{\longrightarrow}} B\longrightarrow N'(B)$ is the $(\mathscr{L}_{\widehat{\mathbb{H}}}, \mathscr{N}_{\widehat{\mathbb{H}}})$-triangle for $B$, then $\Sigma(Ind^{\widehat{\mathbb{G}}}_{\widehat{\mathbb{H}}}\big(N'(B))\big)\longrightarrow Ind^{\widehat{\mathbb{G}}}_{\widehat{\mathbb{H}}}\big(L'(B)\big)\overset{Ind^{\widehat{\mathbb{G}}}_{\widehat{\mathbb{H}}}(u')}{{\longrightarrow}} Ind^{\widehat{\mathbb{G}}}_{\widehat{\mathbb{H}}}\big(B\big)\longrightarrow Ind^{\widehat{\mathbb{G}}}_{\widehat{\mathbb{H}}}\big(N'(B)\big)$ is the $(\mathscr{L}_{\widehat{\mathbb{G}}}, \mathscr{N}_{\widehat{\mathbb{G}}})$-triangle for $Ind^{\widehat{\mathbb{G}}}_{\widehat{\mathbb{H}}}\big(B\big)$.
		
		
		Apply the triangulated functors $\widehat{\mathbb{H}}\underset{r}{\ltimes}\cdot$ and $\widehat{\mathbb{G}}\underset{r}{\ltimes}\cdot$ to these two triangles, respectively so that we get the following distinguished triangles in $\mathscr{K}\mathscr{K}$, $\Sigma(\widehat{\mathbb{H}}\underset{r}{\ltimes}N'(B))\longrightarrow \widehat{\mathbb{H}}\underset{r}{\ltimes}L'(B)\overset{\widehat{\mathbb{H}}\ltimes u'}{{\longrightarrow}} \widehat{\mathbb{H}}\underset{r}{\ltimes}B\longrightarrow \widehat{\mathbb{H}}\underset{r}{\ltimes}N'(B)$ and $\Sigma\big(\widehat{\mathbb{G}}\underset{r}{\ltimes}Ind^{\widehat{\mathbb{G}}}_{\widehat{\mathbb{H}}}\big(N'(B)\big)\big)\longrightarrow \widehat{\mathbb{G}}\underset{r}{\ltimes}Ind^{\widehat{\mathbb{G}}}_{\widehat{\mathbb{H}}}\big(L'(B)\big)\overset{\widehat{\mathbb{G}}\ltimes Ind^{\widehat{\mathbb{G}}}_{\widehat{\mathbb{H}}}(u')}{{\longrightarrow}} \widehat{\mathbb{G}}\underset{r}{\ltimes}Ind^{\widehat{\mathbb{G}}}_{\widehat{\mathbb{H}}}\big(B\big)\longrightarrow \widehat{\mathbb{G}}\underset{r}{\ltimes}Ind^{\widehat{\mathbb{G}}}_{\widehat{\mathbb{H}}}\big(N'(B)\big)$.
		
		Since the isomorphism between $\widehat{\mathbb{G}}\ltimes Ind^{\widehat{\mathbb{G}}}_{\widehat{\mathbb{H}}}(B)$ and $\widehat{\mathbb{H}}\ltimes B$ in $\mathscr{K}\mathscr{K}$ is natural by the quantum Green's Imprimitivity theorem, then we get an isomorphism of distinguished triangles
		$$
			\xymatrix@!C=30mm@R=15mm{
				\mbox{$\Sigma(\widehat{\mathbb{H}}\underset{r}{\ltimes}N'(B))$}\ar[d]^{\mbox{$\wr$}}\ar[r]&\mbox{$\widehat{\mathbb{H}}\underset{r}{\ltimes}L'(B)$}\ar[r]^{\mbox{$\widehat{\mathbb{H}}\ltimes u'$}}\ar[d]^{\mbox{$\wr$}}&\mbox{$\widehat{\mathbb{H}}\underset{r}{\ltimes} B$}\ar[r]\ar[d]^{\mbox{$\wr$}}&\mbox{$\widehat{\mathbb{H}}\underset{r}{\ltimes}N'(B)$}\ar[d]^{\mbox{$\wr$}}\\
				\mbox{$\Sigma\big(\widehat{\mathbb{G}}\underset{r}{\ltimes}Ind^{\widehat{\mathbb{G}}}_{\widehat{\mathbb{H}}}\big(N'(B)\big)\big)$}\ar[r]&\mbox{$\widehat{\mathbb{G}}\underset{r}{\ltimes}Ind^{\widehat{\mathbb{G}}}_{\widehat{\mathbb{H}}}\big(L'(B)\big)$}\ar[r]_{\mbox{$\widehat{\mathbb{G}}\ltimes Ind^{\widehat{\mathbb{G}}}_{\widehat{\mathbb{H}}}(u')$}}&\mbox{$\widehat{\mathbb{G}}\underset{r}{\ltimes}Ind^{\widehat{\mathbb{G}}}_{\widehat{\mathbb{H}}}\big(B\big)$}\ar[r]&\mbox{$\widehat{\mathbb{G}}\underset{r}{\ltimes}Ind^{\widehat{\mathbb{G}}}_{\widehat{\mathbb{H}}}\big(N'(B)\big)$}}
		$$
		which allows to consider the following commutative diagram,
		$$
			\xymatrix@!C=30mm@R=15mm{
				\mbox{$\mathbb{L}F'(B)$}\ar[r]^{\mbox{$\eta^{\widehat{\mathbb{H}}}_{B}$}}\ar[d]^{\mbox{$\wr$}}&\mbox{$F'(B)$}\ar[d]^{\mbox{$\wr$}}\\
				\mbox{$\mathbb{L}F\big(Ind^{\widehat{\mathbb{G}}}_{\widehat{\mathbb{H}}}\big(L'(B)\big)\big)$}\ar[r]_{\mbox{$\eta^{\widehat{\mathbb{G}}}_{Ind^{\widehat{\mathbb{G}}}_{\widehat{\mathbb{H}}}(B)}$}}&\mbox{$F\big(Ind^{\widehat{\mathbb{G}}}_{\widehat{\mathbb{H}}}\big(B\big)\big)$}}
		$$
	
		Since $\eta^{\widehat{\mathbb{G}}}_{Ind^{\widehat{\mathbb{G}}}_{\widehat{\mathbb{H}}}(B)}$ is an isomorphism for all $B\in Obj(\mathscr{K}\mathscr{K}^{\widehat{\mathbb{H}}})$, we conclude that the same is true for $\eta^{\widehat{\mathbb{H}}}_{B}$, that is, $\widehat{\mathbb{H}}$ satisfies the quantum Baum-Connes property with coefficients. 
		The converse is obvious and the proof is complete.
	\end{proof}
	
	\begin{rem}
		For classical groups it is well-known that the Baum-Connes property is preserved by \emph{closed} subgroups. It was showed by J. Chabert and S. Echterhoff in \cite{ChabertPermanence}. To this end they showed that the induction homomorphism $K_*^{top}(H; B)\longrightarrow K_*^{top}(G; Ind^{G}_H(B))$ is \emph{always bijective} (see Theorem 2.2 in \cite{ChabertPermanence} for a proof). In our case, this result is encoded in the identification $K_*(\widehat{\mathbb{H}}\ltimes L'(B))\cong K_*(\widehat{\mathbb{G}}\ltimes L(Ind^{\widehat{\mathbb{G}}}_{\widehat{\mathbb{H}}}(B)))$, obtained by the property $Ind^{\widehat{\mathbb{G}}}_{\widehat{\mathbb{H}}}\circ L'\cong L\circ Ind^{\widehat{\mathbb{G}}}_{\widehat{\mathbb{H}}}$ plus the quantum Green's Imprimitivity theorem.
	\end{rem}
	\begin{rem}\label{rem.BCDivisibleSubgroups}
		To the best knowledge of the author it is open to know if the \emph{strong} Baum-Connes property is preserved by quantum subgroups in general. However, it is well-known that it is preserved by \emph{divisible} discrete quantum subgroups (see Lemma $6.7$ in \cite{VoigtBaumConnesUnitaryFree} for a proof).
	\end{rem}

\section{\textsc{Quantum semi-direct product}}\label{sec.Quantumsemi-directProduct}		
	Let $\mathbb{G}=(C(\mathbb{G}),\Delta)$ be a compact quantum group and $\Gamma$ be a discrete group so that $\Gamma$ is acting on $\mathbb{G}$ by quantum automorphisms with action $\alpha$. In this situation, we can construct the \emph{quantum semi-direct product of $\mathbb{G}$ by $\Gamma$} and it is denoted by $\mathbb{F}:=\Gamma \underset{\alpha}{\ltimes}\mathbb{G}$, where $C(\mathbb{F})=\Gamma\underset{\alpha, m}{\ltimes} C_m(\mathbb{G})$ (see \cite{WangSemidirect} for more details). By definition of the crossed product by a discrete group we have a unital faithful $*$-homomorphism $\pi: C_{m}(\mathbb{G})\longrightarrow C(\mathbb{F})$ and a group homomorphism $u: \Gamma\longrightarrow \mathcal{U}(C(\mathbb{F}))$ defined by $u_{\gamma}:=\lambda_{\gamma}\otimes id_{C_m(\mathbb{G})}$, for all $\gamma\in\Gamma$ such that $C(\mathbb{F})\equiv\Gamma\ltimes_{\alpha, m}C_{m}(\mathbb{G})=C^*\langle \pi(a)u_{\gamma}: a\in C_{m}(\mathbb{G}), \gamma\in\Gamma \rangle$. The co-multiplication $\Theta$ of $\mathbb{F}$ is such that $\Theta\circ \pi=(\pi\otimes\pi)\circ\Delta\mbox{ and } \Theta(u_{\gamma})=u_{\gamma}\otimes u_{\gamma}$, for all $\gamma\in\Gamma$. The Haar state on $\mathbb{F}$ is given by $h_{\mathbb{F}}:=h_{\mathbb{G}}\circ E\circ\kappa$, where $h_{\mathbb{G}}$ is the Haar state of $\mathbb{G}$, $\kappa: \Gamma\underset{\alpha,m}{\ltimes} C_m(\mathbb{G}) \twoheadrightarrow \Gamma\underset{\alpha,r}{\ltimes} C_r(\mathbb{G})$ is the canonical surjection and $E:\Gamma\underset{\alpha,r}{\ltimes} C_r(\mathbb{G})\rightarrow C_r(\mathbb{G})$ is the canonical conditional expectation.
	\begin{noteSec}
		It should be noticed that the notation $\mathbb{F}:=\Gamma \underset{\alpha}{\ltimes}\mathbb{G}$ to refer the semi-direct product of a compact quantum group $\mathbb{G}$ by a discrete group $\Gamma$ is not well behaved with the classical case. Indeed, if $G$ is a classical compact group, the semi-direct product $\Gamma\ltimes G$ is non-discrete in general. In order to cover the compatibility with the classical case, one possible replacement for this notation could be $\mathbb{F}:=\Gamma \underset{\alpha}{\widehat{\ltimes}}\mathbb{G}$ meaning $\widehat{\mathbb{F}}:=\Gamma \underset{\alpha}{\rtimes} \widehat{\mathbb{G}}$. Nevertheless, the above notation will be kept for the rest of the paper.
	\end{noteSec}

	We have $Irr(\mathbb{F})=\Gamma\bigotop Irr(\mathbb{G})$, which means precisely that if $y\in Irr(\mathbb{F})$, then there exist unique $\gamma\in\Gamma$ and $x\in Irr(\mathbb{G})$ such that $w^y:= w^{(\gamma,x)}=v^{\gamma}\otop v^x=\big[v^{\gamma}\big]_{13}\big[v^x\big]_{23}\in\mathcal{B}(\mathbb{C}\otimes H_x)\otimes C(\mathbb{F})$, where $v^{\gamma}:=1_{\mathbb{C}}\otimes u_{\gamma}\in\mathbb{C}\otimes C(\mathbb{F})$ and $v^x:=(id\otimes \pi)(w^x)\in\mathcal{B}(H_x)\otimes C(\mathbb{F})$. 
	
	The representation theory of a quantum semi-direct product $\mathbb{F}$ described above allows to give some explicit expressions which are useful for subsequent computations. For instance, it is advisable to give an explicit description of $\widehat{\mathbb{F}}$ in terms of $\Gamma$ and $\widehat{\mathbb{G}}$.
			
	First of all, since $\alpha$ is an action of $\Gamma$ on $\mathbb{G}$ by quantum automorphisms, then for every $\gamma\in\Gamma$, we have that $(id\otimes \alpha_\gamma)(w^x)$ is an irreducible unitary finite dimensional representation of $\mathbb{G}$ on $H_x$ whenever $x\in Irr(\mathbb{G})$. Hence there exists a unique class $\alpha_\gamma(x)\in Irr(\mathbb{G})$ such that $(id\otimes \alpha_\gamma)(w^x)\cong w^{\alpha_\gamma(x)}$. Since $dim(\alpha_\gamma(x))=dim(x)$ we can assume that $w^{\alpha_\gamma(x)}\in\mathcal{B}(H_x)\otimes C(\mathbb{G})$, for all $\gamma\in\Gamma$ (if this is not the case, we might change the representative of $\alpha_\gamma(x)$ by an appropriate one in the orbit of $x$).
				
	Hence, there exists a unique, up to a multiplicative factor in $S^1$, unitary operator $V_{\gamma, x}\in \mathcal{U}(H_x)$ such that $(id\otimes\alpha_\gamma)(w^x)=(V_{\gamma, x}\otimes id)w^{\alpha_\gamma(x)}(V^*_{\gamma, x}\otimes id)$. Notice that it is clear that $\alpha_e(x)=x$, for all $x\in Irr(\mathbb{G})$ and that $\alpha_\gamma(\epsilon)=\epsilon$, for all $\gamma\in\Gamma$. Therefore, we can choose the multiplicative factor defining $V_{\gamma, x}$ such that $V_{e, x}=id_{H_x}$, for all $x\in Irr(\mathbb{G})$ and $V_{\gamma, \epsilon}=1_\mathbb{C}$, for all $\gamma\in\Gamma$. We keep this choice for the sequel.
	
	Given $\gamma,\gamma'\in \Gamma$ and $x,x'\in Irr(\mathbb{G})$, consider the corresponding irreducible representations of $\mathbb{F}$, say $y:=(\gamma,x), y':=(\gamma', x')\in Irr(\mathbb{F})$, which means that $w^y=v^\gamma\otop v^x$ and $w^{y'}=v^{\gamma'}\otop v^{x'}$. A straightforward computation yields the following $w^{y\otop y'}:=w^{y}\otop w^{y'}=v^{\gamma\gamma'}\otop\big((V_{\gamma'^{-1}}\otimes id)v^{\alpha_{\gamma'^{-1}}(x)}(V^*_{\gamma'^{-1}}\otimes id)\otop v^{x'}\big)$, where $v^{\alpha_{\gamma'^{-1}}(x)}:=(id\otimes \pi\circ\alpha_{\gamma'^{-1}})(w^x)\in \mathcal{B}(H_x)\otimes C(\mathbb{F})$. 
	Consequently, the decomposition of $y\otop y'$ into direct sum of irreducible representations depends only on the corresponding decomposition of $\alpha_{\gamma'^{-1}}(x)\otop x'$. More precisely, if $\{x_k\}_{k=1,\ldots, r}$ is such a decomposition for $\alpha_{\gamma'^{-1}}(x)\otop x'$, then the formula above implies that the corresponding decomposition for $y\otop y'$ is given by $\{(\gamma\gamma', x_k)\}_{k=1,\dots, r}$.
	
	The following lemma provides explicit and useful formulas for the sequel.
	\begin{lemSec}\label{lem.FusionRulesSemiDirect}
		Let $\mathbb{G}=(C(\mathbb{G}),\Delta)$ be a compact quantum group and $\Gamma$ be a discrete group acting on $\mathbb{G}$ by quantum automorphisms with action $\alpha$. Let $\mathbb{F}:=\Gamma\underset{\alpha}{\ltimes} \mathbb{G}$ be the corresponding quantum semi-direct product.
		\begin{enumerate}[i)]
			\item For all $\gamma, g, h\in \Gamma$ and all $x, y, z\in Irr(\mathbb{G})$, we have
				$$h_{\mathbb{F}}\Big(\chi_{\mathbb{F}}(\gamma, x)^*\chi_{\mathbb{F}}\big((g,y)\otop (h, z)\big)\Big)=\left\{ \begin{array}{l}
					h_{\mathbb{G}}\Big(\chi_{\mathbb{G}}(x)^*\chi_{\mathbb{G}}\big(\alpha_{h^{-1}}(y)\otop z\big)\Big)\mbox{, if $\gamma=gh$} \\
					0 \mbox{, otherwise} \\
				\end{array} \right.$$
			\item For all $\gamma, g, h\in \Gamma$ and all $x, y, z\in Irr(\mathbb{G})$, we have
				$$Mor\Big((\gamma, x), (g,y)\otop (h, z)\Big)\cong \left\{ \begin{array}{l}
					Mor\Big(x, \alpha_{h^{-1}}(y)\otop z\Big)\mbox{, if $\gamma=gh$} \\
					0 \mbox{, otherwise} \\
				\end{array} \right.$$
			\item The dual discrete quantum group $\widehat{\mathbb{F}}=(c_0(\widehat{\mathbb{F}}), \widehat{\Theta})$ is given precisely by
				$$c_0(\widehat{\mathbb{F}})\cong c_0(\Gamma)\otimes c_0(\widehat{\mathbb{G}})$$
				and $\widehat{\Theta}: c_0(\widehat{\mathbb{F}})\longrightarrow M(c_0(\widehat{\mathbb{F}})\otimes c_0(\widehat{\mathbb{F}}))$ such that
				$$\widehat{\Theta}(\delta_\gamma\otimes a)\big(p_{(g,y)}\otimes p_{(h,z)}\big)=\delta_{\gamma, gh}\ (\delta_g\otimes p_y\otimes \delta_h\otimes p_z)\Big((V_{h^{-1}, y}\otimes p_z)\widehat{\Delta}(a)(p_y\otimes p_z)(V^*_{h^{-1}, y} \otimes p_z)\Big)_{24}\mbox{,}$$
				for all $\gamma, g, h\in \Gamma$, all $a\in c_0(\widehat{\mathbb{G}})$ and all $y,z\in Irr(\mathbb{G})$.
		\end{enumerate}
	\end{lemSec}
	\begin{proof}
		\begin{enumerate}[i)]
			\item This follows from elementary properties of the character together with the definition of the Haar state of $\mathbb{F}$.
			\item On the one hand, given $\gamma, g, h\in \Gamma$ and all $x, y, z\in Irr(\mathbb{G})$, the computation of $Mor\Big((\gamma, x), (g,y)\otop (h, z)\Big)$ reduces to case when $\gamma =gh$ thanks to the formula of $(i)$. On the other hand, as observed above, for all $\gamma\in \Gamma$ and all $x\in Irr(\mathbb{G})$ there exists a unique (up to a multiplicative factor) unitary operator $V_{\gamma, x}\in \mathcal{U}(H_x)$ such that $(id\otimes \alpha_{\gamma})(w^x)=(V_{\gamma, x}\otimes id_{C(\mathbb{G})})w^{\alpha_{\gamma}(x)}(V^*_{\gamma, x}\otimes id_{C(\mathbb{G})})$.
			
			Hence for all $\gamma, g, h\in \Gamma$ and all $x, y, z\in Irr(\mathbb{G})$ such that $\gamma=gh$, we define
			$$
			\begin{array}{rccl}
				\psi :&Mor\Big(x, \alpha_{h^{-1}}(y)\otop z\Big)& \longrightarrow & Mor\Big((\gamma, x), (g,y)\otop (h, z)\Big)\\
				&\Phi & \longmapsto &\psi(\Phi):= (V_{h^{-1}, y}\otimes id_{H_z})\circ \Phi\mbox{,}
			\end{array}
			$$
			which is a linear isomorphism with inverse $\psi^{-1}(\widetilde{\Phi})=(V^*_{h^{-1}, y}\otimes id_{H_z})\circ \widetilde{\Phi}$, for all $\widetilde{\Phi}\in Mor\Big((\gamma, x), (g,y)\otop (h, z)\Big)$. Routine computations show that $\psi$ is well-defined.
			\item It is clear that $c_0(\widehat{\mathbb{F}})\cong c_0(\Gamma)\otimes c_0(\widehat{\mathbb{G}})$. For the formula of the statement we have just to use the fact that the co-multiplication $\widehat{\Theta}$ is characterized by the relation $\widehat{\Theta}(S)\circ \widetilde{\Phi}=\widetilde{\Phi}\circ S$, for all $S\in \mathcal{B}(H_{(\gamma, x)})\cong \mathcal{B}(H_x)$ with $(\gamma, x)\in Irr(\mathbb{F})$, $\widetilde{\Phi}\in Mor\Big((\gamma, x), (g,y)\otop (h, z)\Big)$ and all $(g,y), (h,z)\in Irr(\mathbb{F})$ together with the isomorphism $\psi$ of $(ii)$.
		\end{enumerate}
	\end{proof}

	\begin{remsSec}\label{rem.QuantumSubgroups}
		\begin{enumerate}
			\item It is important to observe that $\Gamma$ and $\widehat{\mathbb{G}}$ are quantum subgroups of $\widehat{\mathbb{F}}$ with canonical surjections given respectively by $\rho_{\widehat{\mathbb{G}}}:= \varepsilon_{\Gamma}\otimes id_{c_0(\widehat{\mathbb{G}})}\mbox{ and } \rho_{\Gamma}:= id_{c_0(\Gamma)}\otimes \varepsilon_{\widehat{\mathbb{G}}}$, where $\varepsilon_{\Gamma}$ denote de co-unit of $\Gamma$ and $\varepsilon_{\widehat{\mathbb{G}}}$ the co-unit of $\widehat{\mathbb{G}}$, which can be checked by using the formulae of the preceding lemma.
		
		In other words, we have the following canonical injections
		$$\iota^r_\Gamma: C^*_r(\Gamma)\hookrightarrow C_r(\mathbb{F})\mbox{ and } \iota^m_\Gamma: C^*_m(\Gamma)\hookrightarrow C_m(\mathbb{F})$$
		$$\iota^r_{\mathbb{G}}: C_r(\mathbb{G})\hookrightarrow C_r(\mathbb{F})\mbox{ and } \iota^m_{\mathbb{G}}: C_m(\mathbb{G})\hookrightarrow C_m(\mathbb{F})$$
		which intertwine the corresponding co-multiplications.
		
		Let us describe this injections more precisely. The co-unit map $\varepsilon_{\mathbb{G}}: Pol(\mathbb{G})\longrightarrow\mathbb{C}$ extends to a ($\alpha$-invariant) character on $C_m(\mathbb{G})$, which we always denote by $\varepsilon_{\mathbb{G}}: C_m(\mathbb{G})\longrightarrow \mathbb{C}$. Recall that $C_m(\mathbb{F})=\Gamma\underset{\alpha, m}{\ltimes}C_m(\mathbb{G})=C^*\langle \pi(a)u_{\gamma}: a\in C_m(\mathbb{G}), \gamma\in\Gamma\rangle $. So, with the help of the $\alpha$-invariant character above, we can identify $C^*_m(\Gamma)$ with the subalgebra of $C_m(\mathbb{F})$ generated by $\{u_{\gamma} : \gamma\in\Gamma\}$ by universal property (see Remark 3.6 in \cite{FimaBiproduit} for more details). Likewise, recall that $C_r(\mathbb{F})=\Gamma\underset{\alpha, r}{\ltimes}C_r(\mathbb{G})=C^*\langle \pi(a)u_{\gamma}: a\in C_r(\mathbb{G}), \gamma\in\Gamma\rangle $ is equipped with a GNS-faithful conditional expectation $E:\Gamma\ltimes_{\alpha, r}C_r(\mathbb{G})\longrightarrow C_r(\mathbb{G})$, which restricted to the subalgebra generated by $\{u_{\gamma}:\gamma\in\Gamma\}$ is just $E(u_{\gamma})=\delta_{\gamma,e}\in\mathbb{C}$. Remember as well that $u_{\gamma}=\lambda_{\gamma} \otimes id_{C_r(\mathbb{G})} \cong \big[\lambda_{\gamma}\big]_1$ in $\Gamma\underset{\alpha, r}{\ltimes}C_r(\mathbb{G})\subset \mathcal{L}_{C_r(\mathbb{G})}(l^2(\Gamma)\otimes C_r(\mathbb{G}))$; so that this subalgebra is identified canonically to $C^*_r(\Gamma)=\Gamma\underset{tr,r}{\ltimes}\mathbb{C}$ by universal property (here $tr$ denotes the trivial action). 
				
		Observe that, by construction, we have the following relations
		$$\tau_{\mathbb{F}}\circ \iota^m_{\Gamma}=\iota^r_{\Gamma}\circ \tau_{\Gamma}\mbox{, } \varepsilon_{\mathbb{F}}\circ \iota^m_{\Gamma}=\varepsilon_{\Gamma}\mbox{, }\tau_{\mathbb{F}}\circ \iota^m_{\mathbb{G}}=\iota^r_{\mathbb{G}}\circ \tau_{\mathbb{G}}\mbox{ and } \varepsilon_{\mathbb{F}}\circ \iota^m_{\mathbb{G}}=\varepsilon_{\mathbb{G}}\mbox{,}$$
		
		where $\tau_{\mathbb{F}}: C_m(\mathbb{F})\twoheadrightarrow C_r(\mathbb{F})$, $\tau_{\Gamma}: C^*_m(\Gamma)\twoheadrightarrow C^*_r(\Gamma)$, $\tau_{\mathbb{G}}: C_m(\mathbb{G})\twoheadrightarrow C_r(\mathbb{G})$ are the canonical surjections and $\varepsilon_{\mathbb{F}}:Pol(\mathbb{F})\longrightarrow \mathbb{C}$, $\varepsilon_{\Gamma}:\Gamma\longrightarrow \mathbb{C}$, $\varepsilon_{\mathbb{G}}:Pol(\mathbb{G})\longrightarrow \mathbb{C}$ are the co-unit of $\mathbb{F}$, $\Gamma$ and $\mathbb{G}$, respectively whose extension to  $C_m(\mathbb{F})$, $C^*_m(\Gamma)$ and $C_m(\mathbb{G})$ are still denoted by $\varepsilon_{\mathbb{F}}$, $\varepsilon_{\Gamma}$ and $\varepsilon_{\mathbb{G}}$, respectively.
			\item Accordingly, if $(A, \delta)$ is a left $\widehat{\mathbb{F}}$-$C^*$-algebra, then $(A, \delta_{\widehat{\mathbb{G}}})$ is a left $\widehat{\mathbb{G}}$-$C^*$-algebra with $\delta_{\widehat{\mathbb{G}}}:=(\rho_{\widehat{\mathbb{G}}}\otimes id_A)\circ \delta$ and $(A, \delta_{\Gamma})$ is a left $\Gamma$-$C^*$-algebra with $\delta_{\Gamma}:=(\rho_{\Gamma}\otimes id_A)\circ \delta$.
			
			Here it is important to notice the following. Since $\Gamma$ is a classical group, a $\Gamma$-$C^*$-algebra is equivalent to a $C^*$-algebra equipped with a co-action of $c_0(\Gamma)$. This correspondence explains the abuse of language used above. Indeed, the non-degenerate $*$-homomorphism $\delta_\Gamma:= (\rho_\Gamma\otimes id_A)\circ \delta : A\longrightarrow M(c_0(\Gamma)\otimes A)$ defined above is a \emph{co-action} of $(c_0(\Gamma), \widehat{\Delta}_\Gamma)$. The latter is equivalent to give a family of $*$-homomorphisms $\delta^\gamma_{\Gamma}: A\longrightarrow A$, for all $\gamma\in \Gamma$ satisfying $\delta^e_{\Gamma}=id_A$ and $\delta^{\gamma\gamma'}_{\Gamma}=\delta^{\gamma'}_\Gamma\circ \delta^{\gamma}_\Gamma$ for all $\gamma, \gamma'\in\Gamma$, among other properties. Hence, the map
		$$
			\begin{array}{rccl}
				&\Gamma& \longrightarrow & Aut(A)\\
				&\gamma & \longmapsto &\big(\delta_{\Gamma}\big)_{\gamma}\mbox{, $\big(\delta_{\Gamma}\big)_{\gamma}(a):=\delta^{\gamma^{-1}}_\Gamma(a)$}
			\end{array}
		$$
		defines an action of $\Gamma$ on $A$. By abuse of notation, we denote this action by $\delta_\Gamma$ and the difference between the action and the co-action will be clear by the context. 
			\item Moreover, the representation theory of $\mathbb{F}$ yields that $\Gamma$ and $\widehat{\mathbb{G}}$ are divisible in $\widehat{\mathbb{F}}$. Namely, given an irreducible representation $y:=(\gamma, x)\in Irr(\mathbb{F})$ with $\gamma\in \Gamma$ and $x\in Irr(\mathbb{G})$, then $x=(e, x), \gamma=(\gamma,\epsilon)\in [y]$ in $\sim\backslash Irr(\mathbb{F})$. For all $s\in \Gamma$ we have that $s\otop (e,x)=(s,\epsilon)\otop (e,x)=(s,x)\in Irr(\mathbb{F})$, which shows that $\Gamma$ is divisible in $\widehat{\mathbb{F}}$. For all $s\in Irr(G)$ we have that $(\gamma, \epsilon)\otop s=(\gamma, \epsilon)\otop (e,s)=(\gamma, s)$, which shows that $\widehat{\mathbb{G}}$ is divisible in $\widehat{\mathbb{F}}$.
		\end{enumerate}
	\end{remsSec}
	
	Using the formulae of Lemma \ref{lem.FusionRulesSemiDirect} and the previous remarks we get the following formulas.
	\begin{corSec}\label{cor.ActionsComultiplicationsF}
		The following properties hold
		\begin{enumerate}[i)]
			\item The $C^*$-algebra $c_0(\widehat{\mathbb{F}})=c_0(\Gamma)\otimes c_0(\widehat{\mathbb{G}})$ is a $\widehat{\mathbb{G}}$-$C^*$-algebra with action $\widehat{\Theta}_{\widehat{\mathbb{G}}}:=(\rho_{\widehat{\mathbb{G}}}\otimes id_{c_0(\widehat{\mathbb{F}})})\circ \widehat{\Theta}$ such that
			$$\widehat{\Theta}_{\widehat{\mathbb{G}}}(\delta_\gamma\otimes a)\big(p_y\otimes\delta_h\otimes p_z\big)=\delta_{\gamma, h}\ (p_y\otimes\delta_h\otimes p_z) \Big((V_{h^{-1}, y}\otimes p_z)\widehat{\Delta}(a)(p_y\otimes p_z)(V^*_{h^{-1}, y} \otimes p_z)\Big)_{13}\mbox{,}$$
			for all $\gamma, h\in\Gamma$, all $a\in c_0(\widehat{\mathbb{G}})$ and all $y,z\in Irr(\mathbb{G})$.
			
			\item The $C^*$-algebra $c_0(\widehat{\mathbb{F}})=c_0(\Gamma)\otimes c_0(\widehat{\mathbb{G}})$ is a $\Gamma$-$C^*$-algebra with action $\widehat{\Theta}_{\Gamma}:=(\rho_{\Gamma}\otimes id_{c_0(\widehat{\mathbb{F}})})\circ \widehat{\Theta}$ such that
			$$\widehat{\Theta}_{\Gamma}(\delta_\gamma\otimes a)\big(\delta_g\otimes\delta_h\otimes p_z\big)=\delta_{\gamma, gh}\ (\delta_g\otimes \delta_h\otimes p_z) (id\otimes id\otimes a) \mbox{,}$$
			for all $\gamma, g, h\in\Gamma$, all $a\in c_0(\widehat{\mathbb{G}})$ and all $z\in Irr(\mathbb{G})$.
			\item If $\eta: c_0(\Gamma)\otimes c_0(\widehat{\mathbb{G}})\longrightarrow \widetilde{M}(c_0(\widehat{\mathbb{G}})\otimes c_0(\Gamma)\otimes c_0(\widehat{\mathbb{G}}))$ denotes the action of $\widehat{\mathbb{G}}$ on $c_0(\widehat{\mathbb{F}})=c_0(\Gamma)\otimes c_0(\widehat{\mathbb{G}})$ given by the composition $(\Sigma_{12}\otimes id_{c_0(\widehat{\mathbb{G}})})\circ(id_{c_0(\Gamma)}\otimes \widehat{\Delta})$, then $$\widehat{\Theta}_{\widehat{\mathbb{G}}}(\cdot)=(\mathscr{U}\otimes id_{c_0(\widehat{\mathbb{G}})})\eta(\cdot)(\mathscr{U}^*\otimes id_{c_0(\widehat{\mathbb{G}})})\mbox{,}$$
			where $\mathscr{U}\in \mathcal{U}\big(M(c_0(\widehat{\mathbb{G}})\otimes c_0(\Gamma))\big)$ is the unitary such that $\mathscr{U}(p_x\otimes \delta_\gamma)=V_{\gamma^{-1}, x}\otimes \delta_\gamma$, for every $x\in Irr(\mathbb{G})$ and every $\gamma\in \Gamma$.
			
		\end{enumerate}
	\end{corSec}
	\begin{remSec}
		In accordance with the previous remarks, let us give the expression of $\widehat{\Theta}_\Gamma$ as a true action of $\Gamma$ on $c_0(\Gamma)\otimes c_0(\widehat{\mathbb{G}})$ and not as a co-action of $(c_0(\Gamma), \widehat{\Delta}_\Gamma)$ as done in the previous corollary. By applying the formula obtained in the previous corollary, for every $\gamma, r\in\Gamma$ and $a\in c_0(\widehat{\mathbb{G}})$ we write
		\begin{equation*}
		\begin{split}
			\widehat{\Theta}^\gamma_\Gamma(\delta_r\otimes a)&=\widehat{\Theta}_\Gamma(\delta_r\otimes a)(\delta_\gamma\otimes id_{c_0(\widehat{\mathbb{F}})})=\underset{r=st}{\underset{s,t}{\sum}}(\delta_s\otimes \delta_t\otimes id_{c_0(\widehat{\mathbb{G}})}) (id_{c_0(\Gamma)}\otimes id_{c_0(\Gamma)}\otimes a)(\delta_\gamma\otimes id_{c_0(\widehat{\mathbb{F}})})=\delta_{\gamma^{-1}r}\otimes a
		\end{split}
		\end{equation*}
		
		Hence, the corresponding action of $\Gamma$ on $c_0(\Gamma)\otimes c_0(\widehat{\mathbb{G}})$, still denoted by $\widehat{\Theta}_\Gamma$, is given by $\big(\widehat{\Theta}_\Gamma\big)_\gamma(\delta_r\otimes a)=\widehat{\Theta}^{\gamma^{-1}}_\Gamma(\delta_r\otimes a)=\delta_{\gamma r}\otimes a$, for all $\gamma, r\in\Gamma$ and $a\in c_0(\widehat{\mathbb{G}})$.
	\end{remSec}

	\bigskip
	Let us set some notations for the sequel. The canonical triple (in the sense of Theorem \ref{theo.QuantumReducedCrossedProduct}) associated to the reduced crossed product $\Gamma \underset{\delta_{\Gamma},r}{\ltimes} A\subset \mathcal{L}_{A}(l^2(\Gamma)\otimes A)$ is denoted by $(\sigma, \nu, E)$, the one associated to the reduced crossed product $\widehat{\mathbb{F}}\underset{\delta, r}{\ltimes} A\subset \mathcal{L}_{A}(L^2(\mathbb{F})\otimes A)$ is denoted by $(\pi_{\delta}, V, E_{\delta})$ and the one associated to the reduced crossed product $\widehat{\mathbb{G}}\underset{\delta_{\widehat{\mathbb{G}}}, r}{\ltimes} A\subset \mathcal{L}_{A}(L^2(\mathbb{G})\otimes A)$ is denoted by $(\pi_{\delta_{\widehat{\mathbb{G}}}}, U, E_{\delta_{\widehat{\mathbb{G}}}})$.
	
	\bigskip
	\begin{remSec}\label{rem.AlternativeDescription}
		Using the universal property of $\widehat{\mathbb{G}}\underset{\delta_{\widehat{\mathbb{G}}}, r}{\ltimes} A$ and the notations above, it is straightforward to see that if $(A,\delta)$ is a left $\widehat{\mathbb{F}}$-$C^*$-algebra, there exists a canonical $*$-isomorphism $$\widehat{\mathbb{G}}\underset{\delta_{\widehat{\mathbb{G}}}, r}{\ltimes} A\cong C^*\langle \pi_\delta(a)V^{(e,x)}_{i,j}: a\in A, x\in Irr(\mathbb{G}), i,j=1,\ldots,dim(x) \rangle=:\mathscr{C}$$
		
		For this we have just to restrict the canonical triple $(\pi_{\delta}, V, E_{\delta})$ to $\mathscr{C}$.
	\end{remSec}

	\bigskip
	Finally, let $(A,\delta)$ be a left $\widehat{\mathbb{F}}$-$C^*$-algebra and consider the reduced crossed product $\widehat{\mathbb{G}}\underset{\delta_{\widehat{\mathbb{G}}}, r}{\ltimes} A$. We use systematically the canonical identifications $\pi_{\delta_{\widehat{\mathbb{G}}}}(a)\cong \pi_{\delta}(a)$ and $U^{x}_{i,j}\cong V^{(e,x)}_{i,j}$, for all $a\in A$, all $x\in Irr(\mathbb{G})$ and all $i,j=1,\ldots, dim(x)$ given by the preceding remark. 
	
	
	Let us define an action $\partial:\Gamma\rightarrow Aut(\widehat{\mathbb{G}}\ltimes_{\delta_{\widehat{\mathbb{G}}}, r} A)$. Given $\gamma\in\Gamma$ we define the automorphism $\partial_\gamma : \widehat{\mathbb{G}}\underset{\delta_{\widehat{\mathbb{G}}}, r}{\ltimes} A\longrightarrow \widehat{\mathbb{G}}\underset{\delta_{\widehat{\mathbb{G}}}, r}{\ltimes} A$ by $\partial_\gamma := Ad_{\phi_V(u_\gamma)}$, where $Ad_{(\cdot)}$ denotes the adjoint map. This defines clearly an invertible map for each $\gamma\in\Gamma$ so it remains to show that the space $\widehat{\mathbb{G}}\underset{\delta_{\widehat{\mathbb{G}}}, r}{\ltimes} A$ is preserved.
		
		On the one hand, we have $\partial_\gamma\big(\phi_V(w^x_{i,j})\big)=\phi_V(u_\gamma)\phi_V(w^x_{i,j})\phi^*_V(u_\gamma)=\phi_V\big(u_\gamma w^x_{i,j}u^*_\gamma\big)=\phi_V\big(\alpha_\gamma(w^x_{i,j})\big)\in \widehat{\mathbb{G}}\underset{\delta_{\widehat{\mathbb{G}}}, r}{\ltimes} A$,
		for all $\gamma\in\Gamma$, $x\in Irr(\mathbb{G})$, $i,j=1,\ldots, n_x$. On the other hand, the relations of the reduced crossed product $\widehat{\mathbb{F}}\underset{\delta, r}{\ltimes} A$ following Theorem \ref{theo.QuantumReducedCrossedProduct} are precisely $\pi_\delta(a)\phi_V(u_\gamma v^x_{i,j})=\overset{n_x}{\underset{k=1}{\sum}}\phi_V(u_\gamma v^x_{i,k})\pi_\delta(\delta^y_{k,j}(a))$, for all $y:=(\gamma, x)\in Irr(\mathbb{F})$, $a\in A$ and all $i,j=1,\ldots, n_x$. In particular, if we take $x:=\epsilon$ this formula becomes $\pi_{\delta_{\widehat{\mathbb{G}}}}(a)\phi_V(u_\gamma)=\phi_V(u_\gamma)\pi_{\delta_{\widehat{\mathbb{G}}}}\big(\delta^{\gamma}_{\Gamma}(a)\big)$, which means $\big(\partial_\gamma(\pi_{\delta_{\widehat{\mathbb{G}}}}(a))\big)^*=\phi^*_V(u_\gamma)\pi_{\delta_{\widehat{\mathbb{G}}}}(a)\phi_V(u_\gamma)=\pi_{\delta_{\widehat{\mathbb{G}}}}\big(\delta^{\gamma}_{\Gamma}(a)\big)\in \widehat{\mathbb{G}}\underset{\delta_{\widehat{\mathbb{G}}}, r}{\ltimes} A$,
		for all $\gamma\in\Gamma$, $a\in A$. In particular, we have $\phi_V(u_\gamma)\pi_{\delta_{\widehat{\mathbb{G}}}}(a)\phi^*_V(u_\gamma)=\pi_{\delta_{\widehat{\mathbb{G}}}}\big(\delta^{\gamma^{-1}}_{\Gamma}(a)\big)$. Hence, $\partial$ is a well-defined action of $\Gamma$ on $\widehat{\mathbb{G}}\ltimes_{\delta_{\widehat{\mathbb{G}}}, r} A$ such that $\partial_{\gamma}\big(\pi_{\delta_{\widehat{\mathbb{G}}}}(a)U^x_{i,j}\big)=\pi_{\delta}\big(\big(\delta_{\Gamma}\big)_{\gamma}(a)\big)\phi_U\big(\alpha_{\gamma}(w^x_{i,j})\big)$, for all $\gamma\in \Gamma$, all $a\in A$, all $x\in Irr(\mathbb{G})$ and all $i,j=1,\ldots, n_x$

	\bigskip
	The preceding results are true for any $\widehat{\mathbb{F}}$-$C^*$-algebra. We can apply them to the case of the $\widehat{\mathbb{F}}$-$C^*$-algebra $(c_0(\widehat{\mathbb{F}}), \widehat{\Theta})$, which is particularly interesting for our purpose. Recall from Lemma \ref{lem.FusionRulesSemiDirect} that we have the identification $c_0(\widehat{\mathbb{F}})=c_0(\Gamma)\otimes c_0(\widehat{\mathbb{G}})$ and that the dual co-multiplication $\widehat{\Theta}$ has been explicitly described in terms of this identification. Next, we want to describe explicitly the action $\partial$ of $\Gamma$ on $\widehat{\mathbb{G}}\underset{\widehat{\Theta}_{\widehat{\mathbb{G}}}}{\ltimes} \big(c_0(\Gamma)\otimes c_0(\widehat{\mathbb{G}})\big)$ constructed above. Let us set some notations. We denote by $(\pi_{\widehat{\Theta}}, U_{\widehat{\Theta}}, E_{\widehat{\Theta}})$, $(\pi_{\widehat{\Delta}}, U_{\widehat{\Delta}}, E_{\widehat{\Delta}})$, $(\pi_{\eta}, U_{\eta}, E_\eta)$ the canonical triples (following Theorem \ref{theo.QuantumReducedCrossedProduct}) associated to the reduced crossed products $\widehat{\mathbb{G}}\underset{\widehat{\Theta}_{\widehat{\mathbb{G}}}}{\ltimes} \big(c_0(\Gamma)\otimes c_0(\widehat{\mathbb{G}})\big)$, $\widehat{\mathbb{G}}\underset{\widehat{\Delta}}{\ltimes} c_0(\widehat{\mathbb{G}})$, $\widehat{\mathbb{G}}\underset{\eta}{\ltimes} \big(c_0(\Gamma)\otimes c_0(\widehat{\mathbb{G}})\big)$, respectively. 
	We denote by $\mathscr{U}\in \mathcal{U}\big(M(c_0(\widehat{\mathbb{G}})\otimes c_0(\Gamma))\big)$ the unitary introduced in Corollary \ref{cor.ActionsComultiplicationsF} such that $\mathscr{U}(p_x\otimes \delta_\gamma)=V_{\gamma^{-1}, x}\otimes \delta_\gamma$, for every $x\in Irr(\mathbb{G})$ and every $\gamma\in \Gamma$.
	\begin{lemSec}\label{lem.MoritaEquivalenceLemmaBC}
		The following properties hold.
		\begin{enumerate}[i)]
			\item There exists a canonical $*$-isomorphism
				$$\widehat{\mathbb{G}}\underset{\widehat{\Theta}_{\widehat{\mathbb{G}}}}{\ltimes} \big(c_0(\Gamma)\otimes c_0(\widehat{\mathbb{G}})\big)\cong c_0(\Gamma)\otimes \widehat{\mathbb{G}}\underset{\widehat{\Delta}}{\ltimes} c_0(\widehat{\mathbb{G}})$$
				which is $\Gamma$-equivariant, where $c_0(\Gamma)\otimes \widehat{\mathbb{G}}\underset{\widehat{\Delta}}{\ltimes} c_0(\widehat{\mathbb{G}})$ is equipped with the action $\mu$ of $\Gamma$ such that
				$$\mu_\gamma\big((\delta_r\otimes \pi_{\widehat{\Delta}}(a))W^x_{i,j}\big)=(\delta_{\gamma r}\otimes \pi_{\widehat{\Delta}}(a))\phi_{W}\big(\alpha_\gamma(w^x_{i,j})\big)\mbox{,}$$
				for all $a\in c_0(\widehat{\mathbb{G}})$, $\gamma\in\Gamma$, $x\in Irr(\mathbb{G})$, $i,j=1,\ldots, n_x$, where 
				$$W:=\big(U_{\widehat{\Delta}}\big)_{13}\big(id_{c_0(\widehat{\mathbb{G}})}\otimes (id_{c_0(\Gamma)}\otimes \pi_{\widehat{\Delta}})\big)(\mathscr{U}^*\otimes id_{c_0(\widehat{\mathbb{G}})})\in M\big(c_0(\widehat{\mathbb{G}})\otimes c_0(\Gamma)\otimes \widehat{\mathbb{G}}\underset{\widehat{\Delta}}{\ltimes} c_0(\widehat{\mathbb{G}})\big)$$
			\item There exists a canonical $\Gamma$-equivariant $*$-isomorphism
				$$c_0(\Gamma)\otimes \widehat{\mathbb{G}}\underset{\widehat{\Delta}}{\ltimes} c_0(\widehat{\mathbb{G}})\cong c_0(\Gamma)\otimes \mathcal{K}(L^2(\mathbb{G}))\mbox{,}$$
				where $c_0(\Gamma)\otimes \widehat{\mathbb{G}}\underset{\widehat{\Delta}}{\ltimes} c_0(\widehat{\mathbb{G}})$ is equipped with the action $\mu$ defined in $(i)$, $c_0(\Gamma)$ is equipped with the action induced by co-multiplication $\widehat{\Delta}_\Gamma$ and $\mathcal{K}(L^2(\mathbb{G}))$ is equipped with the adjoint action of $\Gamma$ with respect to the unitary representation of $\Gamma$ on $L^2(\mathbb{G})$ induced by $\alpha$.
				
				In particular, there exists a canonical $\Gamma$-equivariant Morita equivalence
				$$c_0(\Gamma)\otimes \widehat{\mathbb{G}}\underset{\widehat{\Delta}}{\ltimes} c_0(\widehat{\mathbb{G}})\underset{\Gamma-M}{\sim} c_0(\Gamma)$$
				
				
		\end{enumerate}
	\end{lemSec}
	\begin{proof}
		\begin{enumerate}[i)]

			\item Recall that $\eta=(\Sigma_{12}\otimes id_{c_0(\widehat{\mathbb{G}})})\circ (id_{c_0(\Gamma)}\otimes \widehat{\Delta})$. By Proposition \ref{pro.CrossedProductTensorProduct} we know that there exists a canonical $*$-isomorphism $\widehat{\mathbb{G}}\underset{\eta}{\ltimes} \big(c_0(\Gamma)\otimes c_0(\widehat{\mathbb{G}})\big)\cong c_0(\Gamma)\otimes \widehat{\mathbb{G}}\underset{\widehat{\Delta}}{\ltimes} c_0(\widehat{\mathbb{G}})$, where the latter is equipped with the triple $\big(id_{c_0(\Gamma)}\otimes \pi_{\widehat{\Delta}}, \big(U_{\widehat{\Delta}}\big)_{13}, id_{c_0(\Gamma)}\otimes E_{\widehat{\Delta}}\big)$. Next, if we replace $\big(U_{\widehat{\Delta}}\big)_{13}$ by the unitary $W$ of the statement, the corresponding triple is again associated to $c_0(\Gamma)\otimes \widehat{\mathbb{G}}\underset{\widehat{\Delta}}{\ltimes} c_0(\widehat{\mathbb{G}})$ in the sense of Theorem \ref{theo.QuantumReducedCrossedProduct}. 
			
			Observe that both triple give rise to isomorphic underlying $C^*$-algebras by means of the unitary $\mathscr{U}$. We claim that the $C^*$-algebra described in terms of the triple $\big(id_{c_0(\Gamma)}\otimes \pi_{\widehat{\Delta}}, W, id_{c_0(\Gamma)}\otimes E_{\widehat{\Delta}}\big)$ is identified to the reduced crossed product $\widehat{\mathbb{G}}\underset{\widehat{\Theta}_{\widehat{\mathbb{G}}}}{\ltimes} \big(c_0(\Gamma)\otimes c_0(\widehat{\mathbb{G}})\big)$. Let us check its universal property.
			
			On the one hand, for all $a\in c_0(\Gamma)\otimes c_0(\widehat{\mathbb{G}})$ we write
			\begin{equation*}
		\begin{split}
			W^*&\big(id_{c_0(\widehat{\mathbb{G}})}\otimes (id_{c_0(\Gamma)}\otimes\pi_{\widehat{\Delta}})(a)\big)W\\
			&=\big(id_{c_0(\widehat{\mathbb{G}})}\otimes (id_{c_0(\Gamma)}\otimes \pi_{\widehat{\Delta}})\big)(\mathscr{U}\otimes id_{c_0(\widehat{\mathbb{G}})})\\
			&\big(U^*_{\widehat{\Delta}}\big)_{13}\big(id_{c_0(\widehat{\mathbb{G}})}\otimes (id_{c_0(\Gamma)}\otimes\pi_{\widehat{\Delta}})(a)\big)\big(U_{\widehat{\Delta}}\big)_{13}\\
			&\big(id_{c_0(\widehat{\mathbb{G}})}\otimes (id_{c_0(\Gamma)}\otimes \pi_{\widehat{\Delta}})\big)(\mathscr{U}^*\otimes id_{c_0(\widehat{\mathbb{G}})})\\
			&\overset{(1)}{=}\big(id_{c_0(\widehat{\mathbb{G}})}\otimes (id_{c_0(\Gamma)}\otimes \pi_{\widehat{\Delta}})\big)(\mathscr{U}\otimes id_{c_0(\widehat{\mathbb{G}})})\\
			&\big(id_{c_0(\widehat{\mathbb{G}})}\otimes(id_{c_0(\Gamma)}\otimes \pi_{\widehat{\Delta}}\big)(\Sigma_{12}\otimes id_{c_0(\widehat{\mathbb{G}})})(id_{c_0(\Gamma)}\otimes \widehat{\Delta})(a)\\
			&\big(id_{c_0(\widehat{\mathbb{G}})}\otimes (id_{c_0(\Gamma)}\otimes \pi_{\widehat{\Delta}})\big)(\mathscr{U}^*\otimes id_{c_0(\widehat{\mathbb{G}})})\\
			&=\big(id_{c_0(\widehat{\mathbb{G}})}\otimes (id_{c_0(\Gamma)}\otimes \pi_{\widehat{\Delta}})\big)\Big((\mathscr{U}\otimes id_{c_0(\widehat{\mathbb{G}})})\eta(a)(\mathscr{U}^*\otimes id_{c_0(\widehat{\mathbb{G}})})\Big)\\
			&\overset{(2)}{=}\big(id_{c_0(\widehat{\mathbb{G}})}\otimes (id_{c_0(\Gamma)}\otimes \pi_{\widehat{\Delta}})\big)\widehat{\Theta}_{\widehat{\mathbb{G}}}(a)\mbox{,}
			\end{split}
		\end{equation*}
		where in $(1)$ we have used Proposition \ref{pro.CrossedProductTensorProduct} and in $(2)$ we have used that  $\widehat{\Theta}_{\widehat{\mathbb{G}}}$ is conjugate of $\eta$ by $\mathscr{U}$ thanks to Corollary \ref{cor.ActionsComultiplicationsF}. On the other hand, a routine computation yields the following expression for all $x\in Irr(\mathbb{G})$, $i,j=1,\ldots, n_x$ $W^x_{i,j}=\underset{\gamma\in \Gamma}{\sum} \delta_\gamma\otimes \big(U^x_{\widehat{\Delta}}\big)_{i,j}\big(V^*_{\gamma, x}\big)_{i,j}\in M\big(c_0(\Gamma)\otimes \widehat{\mathbb{G}}\underset{\widehat{\Delta}}{\ltimes} c_0(\widehat{\mathbb{G}})\big)$, so that for all $r\in\Gamma$, $a\in c_0(\widehat{\mathbb{G}})$, $x\in Irr(\mathbb{G})$, $i,j=1,\ldots, n_x$ we have
		\begin{equation*}
		\begin{split}
			(id_{c_0(\Gamma)}\otimes E_{\widehat{\Delta}})&\Big((id_{c_0(\Gamma)}\otimes \pi_{\widehat{\Delta}})(\delta_r\otimes a)W^x_{i,j}\Big)\\
			&=(id_{c_0(\Gamma)}\otimes E_{\widehat{\Delta}})\Big(\delta_r\otimes \pi_{\widehat{\Delta}}(a)\big(U^x_{\widehat{\Delta}}\big)_{i,j}\big(V^*_{\gamma, x}\big)_{i,j}\Big)\\
			&=\delta_r\otimes a\ \delta_{x,\epsilon}\big(V^*_{\gamma, x}\big)_{i,j}=\delta_r\otimes a\ \delta_{x,\epsilon}=E_{\widehat{\Theta}}(\pi_{\widehat{\Theta}}(\delta_r\otimes a)\big(U^x_{\widehat{\Theta}}\big)_{i,j})
			\end{split}
		\end{equation*}
		
		Hence, by universal property of $\widehat{\mathbb{G}}\underset{\widehat{\Theta}_{\widehat{\mathbb{G}}}}{\ltimes} \big(c_0(\Gamma)\otimes c_0(\widehat{\mathbb{G}})\big)$ there exists a canonical $*$-isomorphism
			$$\psi: \widehat{\mathbb{G}}\underset{\widehat{\Theta}_{\widehat{\mathbb{G}}}}{\ltimes} \big(c_0(\Gamma)\otimes c_0(\widehat{\mathbb{G}})\big)\overset{\sim}{\longrightarrow} c_0(\Gamma)\otimes \widehat{\mathbb{G}}\underset{\widehat{\Delta}}{\ltimes} c_0(\widehat{\mathbb{G}})$$
			such that $$\psi\big(\pi_{\widehat{\Theta}}(\delta_\gamma\otimes a)\big(U_{\widehat{\Theta}}^x\big)_{i,j}\big)=(\delta_\gamma\otimes \pi_{\widehat{\Delta}}(a))W^x_{i,j}\mbox{,}$$
			for all $\gamma\in\Gamma$, $a\in c_0(\widehat{\mathbb{G}})$, $x\in Irr(\mathbb{G})$, $i,j=1,\ldots, n_x$. Moreover, for all $\gamma, r\in\Gamma$, $a\in c_0(\widehat{\mathbb{G}})$, $x\in Irr(\mathbb{G})$, $i,j=1,\ldots, n_x$ we write
			$$\pi_{\widehat{\Theta}}(\delta_r\otimes a)\big(U_{\widehat{\Theta}}^x\big)_{i,j}\overset{\partial_\gamma}{\mapsto} \pi_{\widehat{\Theta}}\big((\widehat{\Theta}_\Gamma)_\gamma(\delta_r\otimes a)\big)\phi_{U_{\widehat{\Theta}}}(\alpha_\gamma(w^x_{i,j}))\overset{\psi}{\mapsto} (\delta_{\gamma r}\otimes \pi_{\widehat{\Delta}}(a))\phi_{W}(\alpha_\gamma(w^x_{i,j}))$$
			$$\pi_{\widehat{\Theta}}(\delta_r\otimes a)\big(U_{\widehat{\Theta}}^x\big)_{i,j}\overset{\psi}{\mapsto} (\delta_r\otimes \pi_{\widehat{\Delta}}(a))W^x_{i,j}\overset{\mu_\gamma}{\mapsto} (\delta_{\gamma r}\otimes \pi_{\widehat{\Delta}}(a))\phi_{W}(\alpha_\gamma(w^x_{i,j}))\mbox{,}$$
			which yields that $\mu$ is a well-defined action of $\Gamma$ on $c_0(\Gamma)\otimes \widehat{\mathbb{G}}\underset{\widehat{\Delta}}{\ltimes} c_0(\widehat{\mathbb{G}})$ and thus the $\Gamma$-equivariance of the statement.

			\item 
			
			We are going to establish a canonical $*$-isomorphism $\psi: c_0(\Gamma)\otimes \widehat{\mathbb{G}}\underset{\widehat{\Delta}}{\ltimes} c_0(\widehat{\mathbb{G}})\overset{\sim}{\longrightarrow} c_0(\Gamma)\otimes \mathcal{K}(L^2(\mathbb{G}))$.
			For this, remark firstly that the canonical triple $(\pi_{\widehat{\Delta}}, U_{\widehat{\Delta}}, E_{\widehat{\Delta}})$ associated to $\widehat{\mathbb{G}}\underset{\widehat{\Delta}}{\ltimes} c_0(\widehat{\mathbb{G}})$ following Theorem \ref{theo.QuantumReducedCrossedProduct} is exactly $(\widehat{\lambda}, \widehat{W}_{\mathbb{G}}\otimes  id_{c_0(\widehat{\mathbb{G}})}, \Omega \otimes id_{c_0(\widehat{\mathbb{G}})})$. 
			
			Moreover, it is well-known that $\widehat{\mathbb{G}}\underset{\widehat{\Delta}}{\ltimes} c_0(\widehat{\mathbb{G}})\cong \mathcal{K}(L^2(\mathbb{G}))$ and for the latter we have the triple $(\widehat{\lambda}, \widehat{W}_{\mathbb{G}}, \widehat{E})$ where $\widehat{E}$ is defined as in the proof of Theorem \ref{theo.QuantumReducedCrossedProduct}. In particular, we have $\mathcal{K}(L^2(\mathbb{G}))=C^*\langle\widehat{\lambda}(a)\lambda(w^x_{i,j})\ |\ a\in c_0(\widehat{\mathbb{G}}), x\in Irr(\mathbb{G}), i.j=1,\ldots, n_x\rangle$.
			
			Since $\Gamma$ acts on $\mathbb{G}$ by quantum automorphisms with action $\alpha$, then $L^2(\mathbb{G})$ is equipped with the action of $\Gamma$ such that $\gamma\cdot \lambda(w^{x'}_{k,l})\Omega = \lambda\big(\alpha_\gamma(w^{x'}_{k,l})\big)\Omega$, for all $\gamma\in\Gamma$, $x'\in Irr(\mathbb{G})$, $k,l=1,\ldots, n_{x'}$. Therefore, the corresponding action on $\mathcal{K}(L^2(\mathbb{G}))$ is such that $\gamma \cdot \widehat{\lambda}(a)\lambda(w^x_{i,j})= \widehat{\lambda}(a)\lambda\big(\alpha_\gamma(w^x_{i,j})\big)$, for all $a\in c_0(\widehat{\mathbb{G}})$, $x\in Irr(\mathbb{G})$, $i,j=1,\ldots, n_x$, which is a straightforward computation.
			
			Since $c_0(\Gamma)\otimes \widehat{\mathbb{G}}\underset{\widehat{\Delta}}{\ltimes} c_0(\widehat{\mathbb{G}})$ has been described in $(ii)$ with the help of the unitary $\mathscr{U}$, then we consider now the triple $(id_{c_0(\Gamma)}\otimes \widehat{\lambda}, \widetilde{W}, id_{c_0(\Gamma)}\otimes \widehat{E})$ associated to $c_0(\Gamma)\otimes \mathcal{K}(L^2(\mathbb{G}))$ where 
			$$\widetilde{W}:=\big(\widehat{W}_{\mathbb{G}}\big)_{13}\big(id_{c_0(\widehat{\mathbb{G}})}\otimes (id_{c_0(\Gamma)}\otimes \widehat{\lambda})\big)(\mathscr{U}^*\otimes id_{c_0(\widehat{\mathbb{G}})})\in M\big(c_0(\widehat{\mathbb{G}})\otimes c_0(\Gamma)\otimes \mathcal{K}(L^2(\mathbb{G})\big)$$
			
			Observe that both triple $(id_{c_0(\Gamma)}\otimes \widehat{\lambda}, \big(\widehat{W}_{\mathbb{G}}\big)_{13}, id_{c_0(\Gamma)}\otimes \widehat{E})$ and $(id_{c_0(\Gamma)}\otimes \widehat{\lambda}, \widetilde{W}, id_{c_0(\Gamma)}\otimes \widehat{E})$ give rise to isomorphic underlying $C^*$-algebras by means of the unitary $\mathscr{U}$. We claim that the $C^*$-algebra $c_0(\Gamma)\otimes \mathcal{K}(L^2(\mathbb{G}))$ described in terms of the triple $\big(id_{c_0(\Gamma)}\otimes \pi_{\widehat{\Delta}}, \widetilde{W}, id_{c_0(\Gamma)}\otimes E_{\widehat{\Delta}}\big)$ is identified to the reduced crossed product $\widehat{\mathbb{G}}\underset{\widehat{\Theta}_{\widehat{\mathbb{G}}}}{\ltimes} \big(c_0(\Gamma)\otimes c_0(\widehat{\mathbb{G}})\big)$. We check its universal property in an analogous way as in $(i)$ by observing that here we have $\widetilde{W}^x_{i,j}=\underset{\gamma\in \Gamma}{\sum} \delta_\gamma\otimes \big(\widehat{W}_{\mathbb{G}}^x\big)_{i,j}\big(V^*_{\gamma, x}\big)_{i,j}\in M\big(c_0(\Gamma)\otimes \mathcal{K}(L^2(\mathbb{G}))\big)$, for all $x\in Irr(\mathbb{G})$, $i,j=1,\ldots, n_x$.

		Hence, by property $(i)$ and universal property of $\widehat{\mathbb{G}}\underset{\widehat{\Theta}_{\widehat{\mathbb{G}}}}{\ltimes} \big(c_0(\Gamma)\otimes c_0(\widehat{\mathbb{G}})\big)$ there exists a canonical $*$-isomorphism
			$$\psi: c_0(\Gamma)\otimes \widehat{\mathbb{G}}\underset{\widehat{\Delta}}{\ltimes} c_0(\widehat{\mathbb{G}})\overset{\sim}{\longrightarrow} c_0(\Gamma)\otimes \mathcal{K}(L^2(\mathbb{G}))$$
			such that $$\psi\big((\delta_\gamma\otimes \pi_{\widehat{\Delta}}(a))W^x_{i,j}\big)=(\delta_\gamma\otimes \widehat{\lambda}(a))\widetilde{W}^x_{i,j}\mbox{,}$$
			for all $\gamma\in\Gamma$, $a\in c_0(\widehat{\mathbb{G}})$, $x\in Irr(\mathbb{G})$, $i,j=1,\ldots, n_x$.
			
			Finally, let us study the $\Gamma$-equivariance condition. 
			By definition, the action of $\Gamma$ on $c_0(\Gamma)\otimes \mathcal{K}(L^2(\mathbb{G}))$ is such that for all $\gamma, r\in\Gamma$, $a\in c_0(\widehat{\mathbb{G}})$, $i,j=1,\ldots, n_x$ $\gamma\cdot \big(\delta_r\otimes \widehat{\lambda}(a)\lambda(w^x_{i,j})\big)=\delta_{\gamma r}\otimes \widehat{\lambda}(a)\lambda\big(\alpha_\gamma(w^x_{i,j})\big)$, which allows to show the $\Gamma$-equivariance of the above $*$-isomorphism because for all $\gamma, r\in\Gamma$, $a\in c_0(\widehat{\mathbb{G}})$, $x\in Irr(\mathbb{G})$, $i,j=1,\ldots, n_x$ we write
			$$(\delta_r\otimes \pi_{\widehat{\Delta}}(a))W^x_{i,j}\overset{\mu_\gamma}{\mapsto} (\delta_{\gamma r}\otimes \pi_{\widehat{\Delta}}(a))\phi_W\big(\alpha_\gamma(w^x_{i,j})\big)\overset{\psi}{\mapsto} (\delta_{\gamma r}\otimes \widehat{\lambda}(a))\phi_{\widetilde{W}}\big(\alpha_\gamma(w^x_{i,j})\big)$$
			$$(\delta_r\otimes \pi_{\widehat{\Delta}}(a))W^x_{i,j}\overset{\psi}{\mapsto} (\delta_r\otimes \widehat{\lambda}(a))\widetilde{W}^x_{i,j}\overset{\gamma\cdot }{\mapsto} (\delta_{\gamma r}\otimes \widehat{\lambda}(a))\phi_{\widetilde{W}}\big(\alpha_\gamma(w^x_{i,j})\big)$$

		\end{enumerate}
	\end{proof}

\subsection{Torsion phenomena}
	Let us study the torsion in the sense of Meyer-Nest. For the following result it is advisable to keep in mind the spectral theory for compact quantum groups recalled in Section \ref{sec.Notations}.
	
	\begin{theo}\label{theo.TorsionQuantumsemi-direct}
		Let $\mathbb{F}=\Gamma\underset{\alpha}{\ltimes}\mathbb{G}$ be the quantum semi-direct product of $\mathbb{G}$ by $\Gamma$. If $\Gamma$ and $\widehat{\mathbb{G}}$ are torsion-free, then $\widehat{\mathbb{F}}$ is torsion-free.
	\end{theo}
	\begin{proof}
		
		Let $A$ be a finite dimensional $C^*$-algebra equipped with a right torsion action of $\mathbb{F}$, say $\delta:A\longrightarrow A\otimes C(\mathbb{F})$. 
		
		Let us define $\Lambda:=\{\gamma\in\Gamma\ |\ \exists\ x\in Irr(\mathbb{G}) \mbox{ such that } K_{(\gamma,x)}\neq 0\}$, where $K_{(\gamma,x)}$ denotes the spectral subspace associated to the representation $(\gamma,x) =: y \in Irr(\mathbb{F})$. We claim that $\Lambda$ is a finite subgroup of $\Gamma$. Indeed, $\Lambda$ is a subgroup of $\Gamma$ because given $g,h\in \Lambda$, let $X_g\in K_{(g,x_g)}$ and $X_h\in K_{(h,x_h)}$ be some non-zero elements for the corresponding irreducible representations $x_g, x_h\in Irr(\mathbb{G})$. Put $y_g:=(g,x_g), y_h:=(h, x_h)\in Irr(\mathbb{F})$. By virtue of Lemma \ref{lem.NonZeroProducts}, there exist an irreducible representation $z:=(\gamma,x)\in Irr(\mathbb{F})$ and an intertwiner $\Phi\in Mor(z,y_g\otop y_h)$ such that $X_g\underset{\Phi}{\otimes} X_h\neq 0$.
		
		Besides, the proof of Lemma \ref{lem.NonZeroProducts} shows that $z$ is an irreducible representation of the decomposition of $y_g\otop y_h$ in direct sum of irreducible representations. Thanks to the fusion rules of a quantum semi-direct product we have that $w^{y_g}\otop w^{y_h}=v^{gh}\otop (v^{\alpha_{h^{-1}}(x_g)}\otop v^{x_h})$. Next, consider the decomposition in direct sum of irreducible representations of the tensor product $\alpha_{h^{-1}}(x_g)\otop x_h$, say $\{x_k\}_{k=1,\ldots, r}$ for some $r\in\mathbb{N}$. Hence we write $w^{y_g\otop y_h}=\overset{r}{\underset{k=1}{\bigoplus}}\ v^{gh}\otop x_k$. As a result, the irreducible representation $z=(\gamma, x)\in Irr(\mathbb{F})$ found above must be of the form $(gh, x_k)$ for some $k=1,\ldots, r$. Recall that $X_g\underset{\Phi}{\otimes} X_h\in K_z$ by definition. This shows that $gh=\gamma\in \Lambda$ as required. Moreover, $\Lambda$ is finite because $A$ is finite dimensional.
		
		Thanks to the torsion-freeness of $\Gamma$, $\Lambda$ is just the trivial subgroup $\{e\}$. Hence, for every $y\in Irr(\mathbb{F})$, $K_y\neq 0$ implies $y=(e,x)$ for some $x\in Irr(\mathbb{G})$. Consequently, the spectral decomposition for $A=\mathcal{A}_{\mathbb{F}}$ becomes $A=\underset{x\in Irr(\mathbb{G})}{\bigoplus}\mathcal{A}_{(e,x)}=\mathcal{A}_{\mathbb{G}}$ and the action $\delta$ takes its values on $A\otimes \pi (C_m(\mathbb{G}))$ so that $\delta$ is actually an action of $\mathbb{G}$ on $A$. Since $\widehat{\mathbb{G}}$ is torsion-free by assumption, we achieve the conclusion.
	\end{proof}
	\begin{note}
		The converse of the preceding statement would be true whenever the torsion-freeness is preserved under divisible discrete quantum subgroups as conjectured in Section \ref{sec.Torsion}.
	\end{note}

\subsection{The Baum-Connes property}\label{sec.BCQuantumsemi-directProduct}
	Let us adapt the notations from Section \ref{pro.ConesTensorProduct} for a quantum semi-direct product $\mathbb{F}=\Gamma\underset{\alpha}{\ltimes}\mathbb{G}$. In order to formulate the quantum Baum-Connes property we assume that $\widehat{\mathbb{F}}$, $\Gamma$ and $\widehat{G}$ are all torsion-free.
	
	Consider the equivariant Kasparov categories associated to $\widehat{\mathbb{F}}$ and $\Gamma$, say $\mathscr{K}\mathscr{K}^{\widehat{\mathbb{F}}}$ and $\mathscr{K}\mathscr{K}^{\Gamma}$, respectively; with canonical suspension functors denoted by $\Sigma$. 
	Consider the usual complementary pair of localizing subcategories in $\mathscr{K}\mathscr{K}^{\widehat{\mathbb{F}}}$ and $\mathscr{K}\mathscr{K}^{\Gamma}$, say $(\mathscr{L}_{\widehat{\mathbb{F}}}, \mathscr{N}_{\widehat{\mathbb{F}}})$ and $(\mathscr{L}_{\Gamma}, \mathscr{N}_{\Gamma})$, respectively. The canonical triangulated functors associated to these complementary pairs will be denoted by $(L,N)$ and $(L', N')$, respectively. Next, consider the homological functors defining the \emph{quantum} Baum-Connes assembly maps for $\widehat{\mathbb{F}}$ and $\Gamma$. Namely,
	$$
		\begin{array}{rccllrccl}
			F:&\mathscr{K}\mathscr{K}^{\widehat{\mathbb{F}}}& \longrightarrow &\mathscr{A}b^{\mathbb{Z}/2}&&F':&\mathscr{K}\mathscr{K}^{\Gamma}& \longrightarrow & \mathscr{A}b^{\mathbb{Z}/2}\\
			&(A,\delta) & \longmapsto &F(A):=K_{*}(\widehat{\mathbb{F}}\underset{\delta, r}{\ltimes} A)&&&(B,\beta) & \longmapsto &F'(B):= K_{*}(\Gamma \underset{\beta,r}{\ltimes} B)
		\end{array}
	$$
	
	The quantum assembly maps for $\widehat{\mathbb{F}}$ and for $\Gamma$ are given by the following natural transformations $\eta^{\widehat{\mathbb{F}}}: \mathbb{L}F\longrightarrow F \mbox{ and } \eta^{\Gamma}: \mathbb{L}F'\longrightarrow F'$ (where that $\mathbb{L}F=F\circ L$ and $\mathbb{L}F'= F'\circ L'$). 

	Given a $\widehat{\mathbb{F}}$-$C^*$-algebra $(A,\delta)\in Obj.(\mathscr{K}\mathscr{K}^{\widehat{\mathbb{F}}})$, we regard it as an object in $\mathscr{K}\mathscr{K}^{\widehat{\mathbb{G}}}$ by restricting the action as explained in Remark \ref{rem.QuantumSubgroups}, that is, we consider \\$(A, \delta_{\widehat{\mathbb{G}}})\in Obj.(\mathscr{K}\mathscr{K}^{\widehat{\mathbb{G}}})$. In this way, it is licit to consider the crossed product $\widehat{\mathbb{G}}\underset{\delta_{\widehat{\mathbb{G}}}, r}{\ltimes} A$. Observe by the way that we have shown previously that $\widehat{\mathbb{G}}\underset{\delta_{\widehat{\mathbb{G}}}, r}{\ltimes} A$ is naturally a $\Gamma$-$C^*$-algebra with action $\partial$. Consider now $(B,\nu)\in Obj.(\mathscr{K}\mathscr{K}^{\widehat{\mathbb{F}}})$ an other $\widehat{\mathbb{F}}$-$C^*$-algebra, let $(B,\nu_{\widehat{\mathbb{G}}})\in Obj.(\mathscr{K}\mathscr{K}^{\widehat{\mathbb{G}}})$ be the corresponding restriction and $\widehat{\mathbb{G}}\underset{\nu_{\widehat{\mathbb{G}}}, r}{\ltimes} B$ the corresponding crossed product, which is again a $\Gamma$-$C^*$-algebra with action $\partial'$. If $\mathcal{X}\in KK^{\widehat{\mathbb{F}}}(A,B)$ is a homomorphism in $\mathscr{K}\mathscr{K}^{\widehat{\mathbb{F}}}$ between $A$ and $B$, we regard it as a homomorphism between $A$ and $B$ in $\mathscr{K}\mathscr{K}^{\widehat{\mathbb{G}}}$, that is, $\mathcal{X}$ is also a $\widehat{\mathbb{G}}$-equivariant triple. Then the functoriality of the crossed product assures that there exists a Kasparov triple (via the descendent homomorphism) $\widehat{\mathbb{G}}\underset{r}{\ltimes} \mathcal{X}\in KK(\widehat{\mathbb{G}}\underset{\delta_{\widehat{\mathbb{G}}}, r}{\ltimes} A,\widehat{\mathbb{G}}\underset{\nu_{\widehat{\mathbb{G}}}, r}{\ltimes} B)$ which is a homomorphism between $\widehat{\mathbb{G}}\underset{\delta_{\widehat{\mathbb{G}}}, r}{\ltimes} A$ and $\widehat{\mathbb{G}}\underset{\nu_{\widehat{\mathbb{G}}}, r}{\ltimes} B$ in $\mathscr{K}\mathscr{K}$. But $\widehat{\mathbb{G}}\underset{\delta_{\widehat{\mathbb{G}}}, r}{\ltimes} A$ and $\widehat{\mathbb{G}}\underset{\nu_{\widehat{\mathbb{G}}}, r}{\ltimes} B$ are actually $\Gamma$-$C^*$-algebras and we can show that the descent homomorphism yields an \emph{equivariant} Kasparov triple, that is, $\widehat{\mathbb{G}}\underset{r}{\ltimes} \mathcal{X}\in KK^{\Gamma}(\widehat{\mathbb{G}}\underset{\delta_{\widehat{\mathbb{G}}}, r}{\ltimes} A,\widehat{\mathbb{G}}\underset{\nu_{\widehat{\mathbb{G}}}, r}{\ltimes} B)$. Namely, if $\mathcal{X}=((H,\delta_H),\pi,F)$, the descent Kasparov triple is given by \\$\widehat{\mathbb{G}}\underset{r}{\ltimes} \mathcal{X}:=(H\otimes_{\pi_{\nu}}\widehat{\mathbb{G}}\underset{\nu_{\widehat{\mathbb{G}}}, r}{\ltimes} B, id\ltimes \pi, F\otimes id)$. Thanks again to Remark \ref{rem.QuantumSubgroups}, $A$ and $B$ are also $\Gamma$-$C^*$-algebras with actions $\delta_\Gamma$ and $\nu_\Gamma$, respectively and so we can regard $\mathcal{X}$ as a homomorphism between $A$ and $B$ in $\mathscr{K}\mathscr{K}^{\Gamma}$, that is, $\mathcal{X}$ is also a $\Gamma$-equivariant triple with action $\delta_{H, \Gamma}:=(id\otimes \rho_{\Gamma})\circ \delta_{H}$. Hence  $H\otimes_{\pi_{\nu}}\widehat{\mathbb{G}}\underset{\nu_{\widehat{\mathbb{G}}}, r}{\ltimes} B$ is a $\Gamma$-equivariant Hilbert $\widehat{\mathbb{G}}\underset{\nu_{\widehat{\mathbb{G}}}, r}{\ltimes} B$-module with the diagonal action $\tau:=(\delta_{H,\Gamma})\otimes \partial'$. Routine computations show that $\widehat{\mathbb{G}}\underset{r}{\ltimes} \mathcal{X}$ is a $\Gamma$-equivariant triple. For more details about these functorial constructions we refer to \cite{VergniouxThesis} or \cite{SkandalisUnitaries}. 
	
	In other words, it is licit to consider the following functor:
	$$
		\begin{array}{rccl}
			\mathcal{Z}:&\mathscr{K}\mathscr{K}^{\widehat{\mathbb{F}}}& \longrightarrow &\mathscr{K}\mathscr{K}^{\Gamma}\\
			&(A,\delta) & \longmapsto &\mathcal{Z}(A):= \widehat{\mathbb{G}}\underset{\delta_{\widehat{\mathbb{G}}}, r}{\ltimes} A
		\end{array}
	$$
	
	\begin{rem}
		Notice that the functor above is well defined at the level of equivariant Kasparov groups. Indeed, let $\mathcal{X}:=((H,\delta_H),\pi,F), \mathcal{X}':=((H',\delta_{H'}),\pi',F')\in \mathbb{E}^{\widehat{\mathbb{F}}}(A, B)$ two $\widehat{\mathbb{F}}$-equivariant Kasparov triple which are homotopic by means of $\mathcal{E}:=((\mathcal{E},\delta_{\mathcal{E}}),\rho,L)\in\mathbb{E}^{\widehat{\mathbb{F}}}(A, C([0,1])\otimes B)$. Remark that $\mathcal{X}$ and $\mathcal{X}'$ will be also homotopic with respect to the restriction actions to $\widehat{\mathbb{G}}$ and $\Gamma$. By the well-known descent homomorphism, $\mathcal{Z}(\mathcal{E})\in\mathbb{E}\big(\widehat{\mathbb{G}}\underset{\delta_{\widehat{\mathbb{G}}}, r}{\ltimes} A,C([0,1])\otimes \widehat{\mathbb{G}}\underset{\nu_{\widehat{\mathbb{G}}}, r}{\ltimes} B\big)$ yields a homotopy between $\mathcal{Z}(\mathcal{X})$ and $\mathcal{Z}(\mathcal{X}')$. If we equipped $\mathcal{Z}(\mathcal{X})$, $\mathcal{Z}(\mathcal{X}')$ and $\mathcal{Z}(\mathcal{E})$ with the diagonal actions $\tau:=(\delta_{H,\Gamma})\otimes \partial'$, $\tau':=(\delta_{H',\Gamma})\otimes \partial'$ and $\widetilde{\tau}:=(\delta_{\mathcal{E},\Gamma})\otimes \partial'$, then a straightforward computation yields that $\mathcal{Z}(\mathcal{X})$ and $\mathcal{Z}(\mathcal{X}')$ are \emph{equivariantly} homotopic by means of $\mathcal{Z}(\mathcal{E})\in \mathbb{E}^{\Gamma}\big(\widehat{\mathbb{G}}\underset{\delta_{\widehat{\mathbb{G}}}, r}{\ltimes} A,C([0,1])\otimes \widehat{\mathbb{G}}\ltimes_{\nu_{\widehat{\mathbb{G}}}, r} B\big)$.
	\end{rem}
	\pagebreak
	\begin{theo}\label{theo.AssociativityandFunctorZ}
		Let $\mathbb{F}=\Gamma\underset{\alpha}{\ltimes}\mathbb{G}$ be a quantum semi-direct product.
		\begin{enumerate}[i)]
			\item (Associativity for the quantum semi-direct product) If $(A,\delta)$ is a $\widehat{\mathbb{F}}$-$C^*$-algebra, then there exists a canonical $*$-isomorphism $\widehat{\mathbb{F}}\underset{\delta, r}{\ltimes} A\cong \Gamma\underset{\partial, r}{\ltimes}\Big(\widehat{\mathbb{G}}\underset{\delta_{\widehat{\mathbb{G}}}, r}{\ltimes} A\Big)$.
			\item The functor $\mathcal{Z}$
					is triangulated such that $\mathcal{Z}(\mathscr{L}_{\widehat{\mathbb{F}}})\subset \mathscr{L}_{\Gamma}$.
		\end{enumerate}
	\end{theo}
	\begin{proof}
		\begin{enumerate}[i)]
			\item In order to prove the isomorphism $\widehat{\mathbb{F}}\underset{\delta, r}{\ltimes} A\cong \Gamma\underset{\partial, r}{\ltimes}\mathscr{C}$, we are going to apply the universal property of the reduced crossed product $\widehat{\mathbb{F}}\underset{\delta, r}{\ltimes} A$. Thus we have to define a triple $(\overline{\rho}, \overline{V}, \overline{E})$ associated to $\Gamma\underset{\partial, r}{\ltimes}\mathscr{C}$ (in the sense of Theorem \ref{theo.QuantumReducedCrossedProduct}). Namely, let us put $\overline{\rho}: A\longrightarrow \Gamma\underset{\partial, r}{\ltimes}\mathscr{C}$ as the composition $\pi_\delta\circ\varrho$; $\overline{V}\in M(c_0(\widehat{\mathbb{F}})\otimes \Gamma\underset{\partial, r}{\ltimes}\mathscr{C})$ as the unitary $V$ itself and $\overline{E}:\Gamma\underset{\partial, r}{\ltimes}\mathscr{C}\longrightarrow A$ as the composition $\mathscr{E}\circ E_{\delta |}$.
					
					
					Routine computations show that the triple $(\overline{\rho}, \overline{V}, \overline{E})$ constructed in this way satisfies the appropriated universal property.
			\item First of all, using Corollary \ref{pro.CrossedProductTensorProduct} and Proposition \ref{pro.ConesCrossedProduct} it is straightforward to see that $\mathcal{Z}$ is triangulated and stable with respect to the canonical suspension functors of the corresponding Kasparov categories.In addition, Proposition \ref{pro.ConesCrossedProduct} guarantees the functor $\mathcal{Z}$ transforms mapping cone triangles into mapping cone triangles and thus it is triangulated. 

					Let us show that $\mathcal{Z}(\mathscr{L}_{\widehat{\mathbb{F}}})\subset \mathscr{L}_{\Gamma}$. Namely,  since all our discrete quantum groups are supposed to be torsion-free, then we know that $\mathscr{L}_{\widehat{\mathbb{F}}}$ is the localizing subcategory of $\mathscr{K}\mathscr{K}^{\widehat{\mathbb{F}}}$ generated by the objects of the form $c_0(\widehat{\mathbb{F}})\otimes C$ with $C$ any $C^*$-algebra in the Kasparov category $\mathscr{K}\mathscr{K}$. Likewise, $\mathscr{L}_{\Gamma}$ is by definition the localizing subcategory of $\mathscr{K}\mathscr{K}^{\Gamma}$ generated by the objects of the form $Ind_{\{e\}}^{\Gamma}(B)$ with $B$ any $C^*$-algebra in the Kasparov category $\mathscr{K}\mathscr{K}$. Recall as well that $c_0(\widehat{\mathbb{F}})\cong c_0(\Gamma)\otimes c_0(\widehat{\mathbb{G}})$ by virtue of the representation theory of $\mathbb{F}=\Gamma\underset{\alpha}{\ltimes}\mathbb{G}$. Hence we write
		\begin{equation*}\label{eq.PreservingL}
			\begin{split}
				\mathcal{Z}(c_0(\widehat{\mathbb{F}})\otimes C)&=\widehat{\mathbb{G}}\underset{\widehat{\Theta}_{\widehat{\mathbb{G}}}, r}{\ltimes} \big(c_0(\widehat{\mathbb{F}})\otimes C\big)\cong \widehat{\mathbb{G}}\underset{\widehat{\Theta}_{\widehat{\mathbb{G}}}, r}{\ltimes} \big(c_0(\Gamma)\otimes c_0(\widehat{\mathbb{G}})\otimes C\big)\\
				&\overset{(1)}{\cong} \widehat{\mathbb{G}}\underset{\widehat{\Theta}_{\widehat{\mathbb{G}}}, r}{\ltimes} \big(c_0(\Gamma)\otimes c_0(\widehat{\mathbb{G}})\big)\otimes C\overset{(2)}{\cong} c_0(\Gamma)\otimes C
			\end{split}
		\end{equation*}
		where in $(1)$ we use Proposition \ref{pro.CrossedProductTensorProduct} and in $(2)$ we use the $\Gamma$-equivariant Morita equivalence given by Lemma \ref{lem.MoritaEquivalenceLemmaBC}. In other words, $\mathcal{Z}(c_0(\widehat{\mathbb{F}})\otimes C)$ is a $\Gamma$-$C^*$-algebra in $\mathscr{K}\mathscr{K}^{\Gamma}$ induced by the trivial subgroup $\{e\}<\Gamma$, which yields the claim because $\mathcal{Z}$ is triangulated and compatible with countable direct sums.
		\end{enumerate}
	\end{proof}

	\begin{rem}\label{rem.AssociativityFunctor}
		Consider the following functors: $\mathscr{K}\mathscr{K}^{\widehat{\mathbb{F}}}\overset{j_{\widehat{\mathbb{F}}}}{\longrightarrow} \mathscr{K}\mathscr{K}$ and $\mathscr{K}\mathscr{K}^{\widehat{\mathbb{F}}}\overset{\mathcal{Z}}{\longrightarrow} \mathscr{K}\mathscr{K}^{\Gamma}\overset{j_{\Gamma}}{\longrightarrow}\mathscr{K}\mathscr{K}$, where $j_{\widehat{\mathbb{F}}}$ is the descent functor with respect to $\widehat{\mathbb{F}}$ and $j_{\Gamma}$ is the descent functor with respect to $\Gamma$.
		
		The theorem above yields that for every $\widehat{\mathbb{F}}$-$C^*$-algebra $(A,\delta)\in Obj.(\mathscr{K}\mathscr{K}^{\widehat{\mathbb{F}}})$ there exists an isomorphism $\eta_{A}:\widehat{\mathbb{F}}\underset{\delta, r}{\ltimes} A\overset{\sim}{\longrightarrow} \Gamma\underset{\partial, r}{\ltimes}\Big(\widehat{\mathbb{G}}\underset{\delta_{\widehat{\mathbb{G}}}, r}{\ltimes} A\Big)$ in $\mathscr{K}\mathscr{K}$. Actually, we get a natural equivalence between the functors above. For this, we have to show that given two $\widehat{\mathbb{F}}$-$C^*$-algebra $(A,\delta), (B, \nu)\in Obj.(\mathscr{K}\mathscr{K}^{\widehat{\mathbb{F}}})$ and a Kasparov triple $\mathcal{X}\in KK^{\widehat{\mathbb{F}}}(A,B)$, the following diagram in $\mathscr{K}\mathscr{K}$ is commutative
		\begin{equation*}
			\xymatrix@!C=70mm@R=15mm{
				\mbox{$\widehat{\mathbb{F}}\underset{\delta, r}{\ltimes} A$}\ar[d]_{\mbox{$\eta_A$}}\ar[r]^{\mbox{$\widehat{\mathbb{F}}\underset{r}{\ltimes}\mathcal{X}$}}&\mbox{$\widehat{\mathbb{F}}\underset{\nu, r}{\ltimes} B$}\ar[d]^{\mbox{$\eta_B$}}\\
				\mbox{$\Gamma\underset{\partial, r}{\ltimes}\Big(\widehat{\mathbb{G}}\underset{\delta_{\widehat{\mathbb{G}}}, r}{\ltimes} A\Big)$}\ar[r]_{\mbox{$\Gamma\underset{r}{\ltimes} \Big(\widehat{\mathbb{G}}\underset{r}{\ltimes}\mathcal{X}\Big)$}}&\mbox{$\Gamma\underset{\partial', r}{\ltimes}\Big(\widehat{\mathbb{G}}\underset{\nu_{\widehat{\mathbb{G}}}, r}{\ltimes} B\Big)$}}
		\end{equation*}
		which is a routine computation. Hence, we have canonically $F=F'\circ \mathcal{Z}$.
	\end{rem}
	
	\begin{lem}\label{lem.InvertibleElement}
		Let $(\mathcal{T}, \Sigma)$, $(\mathcal{T}', \Sigma')$ be two triangulated categories. Let $(\mathscr{L}_{\mathcal{T}}, \mathscr{N}_{\mathcal{T}})$ and $(\mathscr{L}_{\mathcal{T}'}, \mathscr{N}_{\mathcal{T}'})$ be two complementary pairs of localizing subcategories in $\mathcal{T}$ and $\mathcal{T}'$, respectively. Denote by $(L,N)$ and by $(L', N')$ respectively, the canonical triangulated functors associated the these complementary pairs. Let $F:\mathcal{T}\longrightarrow \mathcal{A}b$ and $F':\mathcal{T}'\longrightarrow \mathcal{A}b$ be two homological functors.
		
		If $\mathcal{Z}:\mathcal{T}\longrightarrow \mathcal{T}'$ is a triangulated functor such that $F=F'\circ \mathcal{Z}$ and $\mathcal{Z}(\mathscr{L}_{\mathcal{T}})\subset \mathscr{L}_{\mathcal{T}'}$, then for all object $X\in Obj.(\mathcal{T})$ there exists a homomorphism
		$$\psi \in Hom_{\mathcal{T}'}(\mathcal{Z}(L(X)), L'(\mathcal{Z}(X)))$$
		such that the following diagram is commutative
		\begin{equation*}
		\begin{gathered}
			\xymatrix@C=20mm@!R=15mm{
				\mbox{$\mathbb{L}F(X)$}\ar[d]_{\mbox{$\eta_X$}}\ar[r]^{\mbox{$\Psi$}}&\mbox{$\mathbb{L}F'(\mathcal{Z}(X))$}\ar[d]^{\mbox{$\eta'_{\mathcal{Z}(X)}$}}\\
				\mbox{$F(X)$}\ar[r]^{\mbox{$\cong$}}&\mbox{$F'(\mathcal{Z}(X))$}}
		\end{gathered}
		\end{equation*}
		where $\Psi=F'(\psi)$. If moreover $\mathcal{Z}(\mathscr{N}_{\mathcal{T}})\subset \mathscr{N}_{\mathcal{T}'}$, then $\psi$ is an isomorphism.
	\end{lem}
	\begin{proof}
		Given an object $X\in Obj.(\mathcal{T})$, consider the corresponding distinguished triangle with respect to the complementary pair $(\mathscr{L}_{\mathcal{T}}, \mathscr{N}_{\mathcal{T}})$, say $\Sigma(N(X))\longrightarrow L(X)\overset{u}{\longrightarrow} X\longrightarrow N(X)$. Consider the distinguished $(\mathscr{L}_{\mathcal{T}'}, \mathscr{N}_{\mathcal{T}'})$-triangle associated to the object $\mathcal{Z}(X)\in Obj.(\mathcal{T}')$ say
		\begin{equation}\label{eq.Triangle}
			\begin{split}
				\Sigma'(N'(\mathcal{Z}(X)))\longrightarrow L'(\mathcal{Z}(X))\overset{u'}{\longrightarrow} \mathcal{Z}(X)\longrightarrow N'(\mathcal{Z}(X))	
			\end{split}
		\end{equation}
						
		Let us fix the object $\mathcal{Z}(L(X))=:T\in Obj.(\mathcal{T}')$ and take the long exact sequence associated to the above triangle with respect to the object $T$. Namely,
		\begin{equation*}
			\begin{split}
				\ldots&\rightarrow Hom_{\mathcal{T}'}(T,\Sigma'(N'(\mathcal{Z}(X))))\rightarrow Hom_{\mathcal{T}'}(T,L'(\mathcal{Z}(X)))\overset{(u')_{*}}{\rightarrow}\\
				&\rightarrow Hom_{\mathcal{T}'}(T,\mathcal{Z}(X))\rightarrow Hom_{\mathcal{T}'}(T,N'(\mathcal{Z}(X)))\rightarrow \ldots				
			\end{split}
		\end{equation*}
						
		Since $L(X)\in \mathscr{L}_{\mathcal{T}}$ and we have $\mathcal{Z}(\mathscr{L}_{\mathcal{T}})\subset \mathscr{L}_{\mathcal{T}'}$ by assumption, then $T\in \mathscr{L}_{\mathcal{T}'}$. But, by definition of complementary pair, we have $\mathscr{L}_{\mathcal{T}'}\subset \mathscr{N}^{\vdash}_{\mathcal{T}'}$. In particular, we obtain $Hom_{\mathcal{T}'}(T,\Sigma'(N'(\mathcal{Z}(X))))=(0)=Hom_{\mathcal{T}'}(T,N'(\mathcal{Z}(X)))$. Hence the above long exact sequence yields the isomorphism $Hom_{\mathcal{T}'}(T,L'(\mathcal{Z}(X)))\overset{(u')_{*}}{\cong} Hom_{\mathcal{T}'}(T,\mathcal{Z}(X))$. Hence, just take $\psi:=(u')^{-1}_{*}(\mathcal{Z}(u))$.
		
		Next, put $\Psi:=F'(\psi): F'\big(\mathcal{Z}(L(X))\big)\longrightarrow F'\big(L'(\mathcal{Z}(X))\big)$. The functoriality of constructions and the definition of the element $\psi$ above yields straightforwardly the diagram of the statement.
		
		If moreover we have $\mathcal{Z}(\mathscr{N}_{\mathcal{T}})\subset \mathscr{N}_{\mathcal{T}'}$, then the functor $\mathcal{Z}$ transforms a $(\mathscr{L}_{\mathcal{T}}, \mathscr{N}_{\mathcal{T}})$-triangle for $X$ into a $(\mathscr{L}_{\mathcal{T}'}, \mathscr{N}_{\mathcal{T}'})$-triangle for $\mathcal{Z}(X)$. Since the distinguished triangles associated to a complementary pair are unique up to a isomorphism, we have an isomorphism of distinguished triangles between (\ref{eq.Triangle}) and the image of $\Sigma(N(X))\longrightarrow L(X)\overset{u}{\longrightarrow} X\longrightarrow N(X)$ by $\mathcal{Z}$,
		 $$
			\xymatrix@!C=30mm@R=15mm{
				\mbox{$\Sigma'(\mathcal{Z}(N(X)))$}\ar[d]^{\mbox{$\wr$}}\ar[r]&\mbox{$\mathcal{Z}(L(X))$}\ar[r]^{\mbox{$\mathcal{Z}(u)$}}\ar[d]^{\mbox{$\wr\ \psi$}}&\mbox{$\mathcal{Z}(X)$}\ar[r]\ar@{=}[d]^{\mbox{$id$}}&\mbox{$ \mathcal{Z}(N(X))$}\ar[d]^{\mbox{$\wr$}}\\
				\mbox{$\Sigma'(N'(\mathcal{Z}(X)))$}\ar[r]&\mbox{$L'(\mathcal{Z}(X))$}\ar[r]_{\mbox{$u'$}}&\mbox{$\mathcal{Z}(X)$}\ar[r]&\mbox{$N'(\mathcal{Z}(X)))$}}
		$$
	\end{proof}

	If we apply the preceding lemma to our particular situation, we get that for every $\widehat{\mathbb{F}}$-$C^*$-algebra $(A,\delta)$ there exists an element $\psi\in KK^{\Gamma}(\widehat{\mathbb{G}}\underset{r}{\ltimes} L(A), L'(\widehat{\mathbb{G}}\underset{\delta_{\widehat{\mathbb{G}}}, r}{\ltimes} A))$ such that the following diagram is commutative
		\begin{equation}\label{eq.CommutativeDiagram}
		\begin{gathered}
			\xymatrix@C=20mm@!R=15mm{
				\mbox{$\mathbb{L}F(A)$}\ar[d]_{\mbox{$\eta^{\widehat{\mathbb{F}}}_{A}$}}\ar[r]^-{\mbox{$\Psi$}}&\mbox{$\mathbb{L}F'(\widehat{\mathbb{G}}\underset{\delta_{\widehat{\mathbb{G}}}, r}{\ltimes}A)$}\ar[d]^{\mbox{$\eta^{\Gamma}_{\widehat{\mathbb{G}}\underset{r}{\ltimes} A}$}}\\
				\mbox{$F(A)$}\ar[r]^{\mbox{$\cong$}}&\mbox{$F'(\widehat{\mathbb{G}}\underset{\delta_{\widehat{\mathbb{G}}}, r}{\ltimes} A)$}}
		\end{gathered}
		\end{equation}
				
		where $\Psi:=F'(\psi)$, $\eta^{\widehat{\mathbb{F}}}_{A}$ is the assembly map for $\widehat{\mathbb{F}}$ with coefficients in $A$ and $\eta^{\Gamma}_{\widehat{\mathbb{G}}\ltimes_{r} A}$ is the assembly map for $\Gamma$ with coefficients in $\widehat{\mathbb{G}}\underset{\delta_{\widehat{\mathbb{G}}}, r}{\ltimes} A$. Precisely, if $\Sigma(N(A))\longrightarrow L(A)\overset{u}{\longrightarrow} A\longrightarrow N(A)$ is a $(\mathscr{L}_{\widehat{\mathbb{F}}}, \mathscr{N}_{\widehat{\mathbb{F}}})$-triangle associated to $A$ and $\Sigma(N'(\widehat{\mathbb{G}}\underset{\delta_{\widehat{\mathbb{G}}}, r}{\ltimes} A))\longrightarrow L'(\widehat{\mathbb{G}}\underset{\delta_{\widehat{\mathbb{G}}}, r}{\ltimes} A)\overset{u'}{\longrightarrow} \widehat{\mathbb{G}}\underset{\delta_{\widehat{\mathbb{G}}}, r}{\ltimes} A\longrightarrow N'(\widehat{\mathbb{G}}\underset{\delta_{\widehat{\mathbb{G}}}, r}{\ltimes} A)$ is a $(\mathscr{L}_{\Gamma}, \mathscr{N}_{\Gamma})$-triangle associated to $\widehat{\mathbb{G}}\ltimes_{ \delta_{\widehat{\mathbb{G}}}, r} A$, then $\psi:=(u')^{-1}_{*}(\mathcal{Z}(u))$.
	
	\bigskip
	We can now conclude our study with the following theorem, generalizing the result \cite{Chabert} of J. Chabert as we have discussed in the introduction of this article.
	\begin{theo}\label{theo.FinalTheorem}
		Let $\mathbb{F}=\Gamma\underset{\alpha}{\ltimes}\mathbb{G}$ be a quantum semi-direct product. Assume that $\widehat{\mathbb{F}}$, $\Gamma$ and $\widehat{\mathbb{G}}$ are torsion-free discrete quantum groups. Let $(A,\delta)$ be a left $\widehat{\mathbb{F}}$-$C^*$-algebra. $\widehat{\mathbb{F}}$ satisfies the quantum Baum-Connes property with coefficients in $A$ if and only if $\Gamma$ satisfies the Baum-Connes property with coefficients in $\widehat{\mathbb{G}}\underset{\delta_{\widehat{\mathbb{G}}}, r}{\ltimes}A$ and $\widehat{\mathbb{G}}$ satisfies the quantum Baum-Connes property with coefficients in $A$.
	\end{theo}
	\begin{proof}
		Assume that $\widehat{\mathbb{F}}$ satisfies the quantum Baum-Connes property. Since $\Gamma$ and $\widehat{\mathbb{G}}$ are divisible torsion-free discrete quantum subgroups of $\widehat{\mathbb{F}}$ thanks to Remarks \ref{rem.QuantumSubgroups}, then by Proposition \ref{pro.BCDivisibleQuantumSubgroups} they satisfy the Baum-Connes property.
		
		Conversely, assume that $\widehat{\mathbb{G}}$ satisfies the quantum Baum-Connes property with coefficients in $A$ and $\Gamma$ satisfies the Baum-Connes property with coefficients in $\widehat{\mathbb{G}}\underset{\delta_{\widehat{\mathbb{G}}}, r}{\ltimes}A$. Since $\widehat{\mathbb{F}}$ is supposed to be torsion-free, then Theorem \ref{theo.TorsionQuantumsemi-direct} assures that $\Gamma$ and $\widehat{\mathbb{G}}$ are torsion-free. Consequently, the \emph{only} finite subgroup of $\Gamma$ is the trivial one, $\{e\}<\Gamma$. It is obvious that the trivial group $\{e\}$ satisfies the Baum-Connes property.
		
		Denote by $(\mathscr{L}_{\{e\}}, \mathscr{N}_{\{e\}})$ the complementary pair of localizing subcategories in $\mathscr{K}\mathscr{K}^{\{e\}}$ and by $(L'', N'')$ the associated functors. Consider the element $\psi\in KK^{\Gamma}(\widehat{\mathbb{G}}\underset{r}{\ltimes} L(A), L'(\widehat{\mathbb{G}}\underset{\delta_{\widehat{\mathbb{G}}}, r}{\ltimes} A))$ constructed from Lemma \ref{lem.InvertibleElement} and denote by $\psi_{\{e\}}\in KK(\widehat{\mathbb{G}}\underset{r}{\ltimes} L(A), L'(\widehat{\mathbb{G}}\underset{\delta_{\widehat{\mathbb{G}}}, r}{\ltimes} A))$ the same element when the action of $\Gamma$ is restricted to the trivial subgroup. Precisely, the latter is defined in the following way. Given a $\widehat{\mathbb{F}}$-$C^*$-algebra $A$, let $\Sigma(N(A))\longrightarrow L(A)\overset{u}{\longrightarrow} A\longrightarrow N(A)$ be a $(\mathscr{L}_{\widehat{\mathbb{F}}}, \mathscr{N}_{\widehat{\mathbb{F}}})$-triangle associated to $A$ and let \\$\Sigma(N''\big(Res^{\Gamma}_{\{e\}}\big(\widehat{\mathbb{G}}\ltimes_{ \delta_{\widehat{\mathbb{G}}}, r} A\big)\big)\longrightarrow L''\big(Res^{\Gamma}_{\{e\}}\big(\widehat{\mathbb{G}}\ltimes_{ \delta_{\widehat{\mathbb{G}}}, r} A\big)\big)\overset{u''}{\longrightarrow} Res^{\Gamma}_{\{e\}}\big(\widehat{\mathbb{G}}\ltimes_{ \delta_{\widehat{\mathbb{G}}}, r} A\big)\longrightarrow N''\big(Res^{\Gamma}_{\{e\}}\big(\widehat{\mathbb{G}}\ltimes_{ \delta_{\widehat{\mathbb{G}}}, r} A\big)\big)$ be a $(\mathscr{L}_{\{e\}}, \mathscr{N}_{\{e\}})$-triangle associated to $Res^{\Gamma}_{\{e\}}\big(\widehat{\mathbb{G}}\ltimes_{ \delta_{\widehat{\mathbb{G}}}, r} A\big)$. Then $\psi_{\{e\}}:=(u'')^{-1}_{*}\big(Res^{\Gamma}_{\{e\}}\big(\mathcal{Z}(u)\big)\big)$, which is well defined following the analogue proof of Lemma \ref{lem.InvertibleElement} together with Lemma \ref{lem.ResIndDirac}. Observe that $\psi_{\{e\}}\in KK\Big(Res^{\Gamma}_{\{e\}}\big(\widehat{\mathbb{G}}\underset{r}{\ltimes} L(A)\big), L''\big(Res^{\Gamma}_{\{e\}}\big(\widehat{\mathbb{G}}\underset{r}{\ltimes} A\big)\big)\Big)$, but by Lemma \ref{lem.ResIndDirac} we know that there exists a natural isomorphism $Res^{\Gamma}_{\{e\}}\circ L'\cong L''\circ Res^{\Gamma}_{\{e\}}$ so that we identify $\psi_{\{e\}}$ with $Res^{\Gamma}_{\{e\}}(\psi)\in KK\Big(Res^{\Gamma}_{\{e\}}\big(\widehat{\mathbb{G}}\underset{r}{\ltimes} L(A)\big), Res^{\Gamma}_{\{e\}}\big(L'\big(\widehat{\mathbb{G}}\underset{r}{\ltimes} A\big)\big)\Big)$. Thanks again to Lemma \ref{lem.ResIndDirac} we know that the restriction functor transforms the assembly map for $\Gamma$ into the assembly map for $\{e\}$. In this way, the restriction of the action of $\Gamma$ to the trivial subgroup transforms the commutative diagram (\ref{eq.CommutativeDiagram}) into the following commutative diagram (where the notation $Res^{\Gamma}_{\{e\}}(\cdot)$ has been removed for simplicity),
		\begin{equation*}
		\begin{gathered}
			\xymatrix@C=15mm@!R=15mm{
				\mbox{$\mathbb{L}F_{\{e\}}(A)$}\ar[d]_{\mbox{$\eta^{\widehat{\mathbb{F}}_{\{e\}}}_{A}$}}\ar[r]^-{\mbox{$\Psi_{\{e\}}$}}&\mbox{$\mathbb{L}F'_{\{e\}}(\widehat{\mathbb{G}}\underset{\delta_{\widehat{\mathbb{G}}}, r}{\ltimes}A)$}\ar[d]^{\mbox{$\eta^{\{e\}}_{\widehat{\mathbb{G}}\underset{r}{\ltimes} A}$}}\\
				\mbox{$F_{\{e\}}(A)$}\ar@{=}[r]&\mbox{$F'_{\{e\}}(\widehat{\mathbb{G}}\underset{\delta_{\widehat{\mathbb{G}}}, r}{\ltimes} A)$}}
		\end{gathered}
		\end{equation*} 
	where $F_{\{e\}}$ et $F'_{\{e\}}$ are the analogous functors to $F$ and $F'$ defined with respect to $\mathbb{F}_{\{e\}}:=\{e\}\underset{\alpha_{|}}{\ltimes}\mathbb{G}=\mathbb{G}$ and $\{e\}$, respectively and $\Psi_{\{e\}}:=F'_{\{e\}}(\psi_{\{e\}})$. In this situation we have $\eta^{\widehat{\mathbb{F}}_{\{e\}}}_{A}=\eta^{\widehat{\mathbb{G}}}_{A}$. Since $\widehat{\mathbb{G}}$ satisfies the quantum Baum-Connes property by assumption, then $\eta^{\widehat{\mathbb{F}}_{\{e\}}}_{A}=\eta^{\widehat{\mathbb{G}}}_{A}$ is a natural isomorphism. Hence, $\Psi_{\{e\}}$ is a natural isomorphism. This means in particular that the element $\psi$ of Lemma \ref{lem.InvertibleElement} induces, by restriction, a $K$-equivalence between $\widehat{\mathbb{G}}\underset{r}{\ltimes} L(A)$ and $L'(\widehat{\mathbb{G}}\underset{\delta_{\widehat{\mathbb{G}}}, r}{\ltimes} A)$. Observe that $\{e\}\underset{r}{\ltimes}\Big(\widehat{\mathbb{G}}\underset{r}{\ltimes} L(A)\Big)\cong \widehat{\mathbb{G}}\underset{r}{\ltimes} L(A)$ and that $\{e\}\underset{r}{\ltimes}\Big(L'(\widehat{\mathbb{G}}\underset{\delta_{\widehat{\mathbb{G}}}, r}{\ltimes} A)\Big)\cong L'(\widehat{\mathbb{G}}\underset{\delta_{\widehat{\mathbb{G}}}, r}{\ltimes} A)$. Since $\{e\}$ is the only finite subgroup of $\Gamma$ by our assumptions, Theorem $9.3$ in \cite{MeyerNest} yields therefore that the same element induces a $K$-equivalence between $\Gamma\underset{r}{\ltimes} L'\Big(\widehat{\mathbb{G}}\underset{r}{\ltimes} L(A)\Big)$ and $\Gamma\underset{r}{\ltimes} L'(\widehat{\mathbb{G}}\underset{\delta_{\widehat{\mathbb{G}}}, r}{\ltimes} A)$.

	Observe that $L(A)\in\mathscr{L}_{\widehat{\mathbb{F}}}$, so $\widehat{\mathbb{G}}\underset{r}{\ltimes} L(A)=\mathcal{Z}(L(A))\in\mathscr{L}_{\Gamma}$ thanks to Theorem \ref{theo.AssociativityandFunctorZ} $(ii)$. Hence we have $L'\Big(\widehat{\mathbb{G}}\underset{r}{\ltimes} L(A)\Big)\cong \widehat{\mathbb{G}}\underset{r}{\ltimes} L(A)$ in $\mathscr{K}\mathscr{K}^{\Gamma}$. In other words, $\Gamma\underset{r}{\ltimes} \Big(\widehat{\mathbb{G}}\underset{r}{\ltimes} L(A)\Big)$ is $K$-equivalent to $\Gamma\underset{r}{\ltimes} L'(\widehat{\mathbb{G}}\underset{\delta_{\widehat{\mathbb{G}}}, r}{\ltimes} A)$ via the element $\psi$. That is, $\Psi=F'(\psi)$ is an isomorphism.
	
	To conclude, we use the commutative diagramme (\ref{eq.CommutativeDiagram}). Namely, since $\Gamma$ satisfies the Baum-Connes property with coefficients in $\widehat{\mathbb{G}}\underset{\delta_{\widehat{\mathbb{G}}}, r}{\ltimes}A$ by assumption, then $K_{*}(\Gamma\underset{r}{\ltimes}L'(\widehat{\mathbb{G}}\underset{\delta_{\widehat{\mathbb{G}}}, r}{\ltimes} A))\cong K_{*}(\Gamma\underset{\partial, r}{\ltimes}(\widehat{\mathbb{G}}\underset{\delta_{\widehat{\mathbb{G}}}, r}{\ltimes}A))$ via $\eta^{\Gamma}_{\widehat{\mathbb{G}}\underset{r}{\ltimes} A}$. By using the associativity for quantum semi-direct products from Theorem \ref{theo.AssociativityandFunctorZ} $(i)$ we get $K_{*}(\widehat{\mathbb{F}}\underset{r}{\ltimes}L(A))\overset{\sim}{\longrightarrow} K_{*}(\widehat{\mathbb{F}}\underset{\delta, r}{\ltimes}A)$ via $\eta^{\Gamma}_{\widehat{\mathbb{G}}\underset{r}{\ltimes} A}\circ\Psi$. So $\mathbb{L}F(A)\cong F(A)$ through $\eta^{\widehat{\mathbb{F}}}_{A}$ thanks to the commutativity of diagram (\ref{eq.CommutativeDiagram}). That is, $\widehat{\mathbb{F}}$ satisfies the quantum Baum-Connes property with coefficients in $A$.
	\end{proof}
	\begin{rem}
		The argument of the preceding theorem can be applied when $\Gamma$ has more finite subgroups than the trivial one. Indeed, we could do the argument with the quantum semi-direct products given by $\mathbb{F}_{\Lambda}:=\Lambda\underset{\alpha_{|}}{\ltimes}\mathbb{G}$ for every finite subgroup $\Lambda<\Gamma$. In that case, the claim ``\emph{$\Psi=F'(\psi)$ is an isomorphism}'', which is used in order to conclude using the commutative diagram (\ref{eq.CommutativeDiagram}), can be achieved by applying Theorem $9.3$ in \cite{MeyerNest}. The problem with this case is that the finite subgroups of $\Gamma$ provide torsion of $\widehat{\mathbb{F}}$ by virtue of Theorem \ref{theo.TorsionQuantumsemi-direct}. Hence the theoretical framework for the quantum Baum-Connes property fails. It is reasonable to expect that the same stabilization property holds for any quantum semi-direct product (not necessarily torsion-free) once the Baum-Connes property can be formulated properly without the torsion-freeness assumption.
	\end{rem}
	
	\begin{theo}\label{theo.FinalTheoremStrong}
		Let $\mathbb{F}=\Gamma\underset{\alpha}{\ltimes}\mathbb{G}$ be a quantum semi-direct product such that $\widehat{\mathbb{F}}$, $\Gamma$ and $\widehat{\mathbb{G}}$ is a torsion-free discrete quantum groups. 
		
		If $\widehat{\mathbb{F}}$ satisfies the \emph{strong} quantum Baum-Connes property, then $\Gamma$ satisfies the \emph{strong} Baum-Connes property and $\widehat{\mathbb{G}}$ satisfies the \emph{strong} quantum Baum-Connes property.
	\end{theo}
	\begin{proof}
		Assume that $\widehat{\mathbb{F}}$ satisfies the \emph{strong} quantum Baum-Connes property. Since $\Gamma$ and $\widehat{\mathbb{G}}$ are \emph{divisible} torsion-free discrete quantum subgroups of $\widehat{\mathbb{F}}$ thanks to Remark \ref{rem.QuantumSubgroups}, then they satisfy the \emph{strong} Baum-Connes property by virtue of Remark \ref{rem.BCDivisibleSubgroups}.
	\end{proof}
		
\subsection{$K$-amenability property}
	To finish, we study an other property of own interest: the \emph{$K$-amenability}. Namely, we get the following
	\begin{theo}\label{theo.KAmenability}
		Let $\mathbb{F}=\Gamma\underset{\alpha}{\ltimes}\mathbb{G}$ be a quantum semi-direct product. Then $\mathbb{F}$ is co-$K$-amenable if and only if $\Gamma$ is $K$-amenable and $\mathbb{G}$ is co-$K$-amenable.
	\end{theo}
	\begin{proof}
		Assume that $\mathbb{F}$ is co-$K$-amenable. This means that there exists an element $\alpha_{\mathbb{F}}\in KK(C_r(\mathbb{F}), \mathbb{C})$ such that $[\tau_{\mathbb{F}}]\underset{C_r(\mathbb{F})}{\otimes} \alpha=[\varepsilon_{\mathbb{F}}]\in KK(C_m(\mathbb{F}),\mathbb{C})$,
			where $\tau_{\mathbb{F}}: C_m(\mathbb{F})\twoheadrightarrow C_r(\mathbb{F})$ is the canonical surjection and $\varepsilon_{\mathbb{F}}:Pol(\mathbb{F})\longrightarrow \mathbb{C}$ is the co-unit of $\mathbb{F}$ whose extension to $C_m(\mathbb{F})$ is still denoted by $\varepsilon_{\mathbb{F}}$.
			
			By virtue of Remark \ref{rem.QuantumSubgroups} we know that $\Gamma$ and $\widehat{\mathbb{G}}$ are discrete quantum subgroups of $\widehat{\mathbb{F}}$ via the canonical injections $\iota^r_{\Gamma}: C^*_r(\Gamma)\hookrightarrow C_r(\mathbb{F})\mbox{ and }\iota^r_{\mathbb{G}}: C_r(\mathbb{G})\hookrightarrow C_r(\mathbb{F})$. In this situation it is straightforward to show that $\Gamma$ and $\widehat{\mathbb{G}}$ are $K$-amenable with elements $\alpha_{\Gamma}:=[\iota^r_{\Gamma}]\underset{C_r(\mathbb{F})}{\otimes}\alpha_{\mathbb{F}}\in KK(C^*_r(\Gamma), \mathbb{C})\mbox{ and }\alpha_{\mathbb{G}}:=[\iota^r_{\mathbb{G}}]\underset{C_r(\mathbb{F})}{\otimes}\alpha_{\mathbb{F}}\in KK(C_r(\mathbb{G}), \mathbb{C})$.
			
		 Conversely, assume that $\Gamma$ is $K$-amenable and that $\mathbb{G}$ is co-$K$-amenable. By virtue of the $K$-amenability characterization of J. Cuntz (see Theorem 2.1 in \cite{CuntzKmoyenabilite}), the surjection $\Gamma\underset{m}{\ltimes} A\twoheadrightarrow \Gamma\underset{r}{\ltimes} A$ induces a $KK$-equivalence for every $\Gamma$-$C^*$-algebra $A$.	In particular, $\Gamma\underset{m}{\ltimes} C_m(\mathbb{G})\twoheadrightarrow \Gamma\underset{r}{\ltimes} C_m(\mathbb{G})$ induces a $KK$-equivalence. Since $\mathbb{G}$ is co-$K$-amenable, then the canonical surjection $\tau_{\mathbb{G}}: C_m(\mathbb{G})\twoheadrightarrow C_r(\mathbb{G})$, which is $\Gamma$-equivariant, induces a $\Gamma$-equivariant $KK$-equivalence. If $j_{\Gamma}$ denotes the descent homomorphism with respect to $\Gamma$, which is compatible with the Kasparov product, then it is clear that $[id\ltimes \tau_{\mathbb{G}}]=j_{\Gamma}([\tau_{\mathbb{G}}])\in KK(\Gamma\underset{r}{\ltimes} C_m(\mathbb{G}), \Gamma\underset{r}{\ltimes} C_r(\mathbb{G}))$ is an invertible element. In other words, the composition $\Gamma\underset{m}{\ltimes} C_m(\mathbb{G})\twoheadrightarrow \Gamma\underset{r}{\ltimes} C_m(\mathbb{G})\overset{id\ltimes \tau_{\mathbb{G}}}{\rightarrow} \Gamma\underset{r}{\ltimes} C_r(\mathbb{G}))$, which is precisely $\tau_{\mathbb{F}}$, induces a $KK$-equivalence. Hence $\mathbb{F}$ is co-$K$-amenable.
	\end{proof}
	
\section{\textsc{Compact Bicrossed Product}}
In this section we observe that all preceding results can be established for a compact bricrossed product in the sense of \cite{FimaBiproduit} under torsion-freeness assumption. Let $G$ be a (classical) compact  group and $\Gamma$ be a discrete group so that $(\Gamma, G)$ is a matched pair (see \cite{SkandalisUnitaries} or \cite{FimaBiproduit} for a precise definition). Then there exists a continuous left action of $\Gamma$ on the topological space $G$, $\alpha:\Gamma\times G\longrightarrow G$, and a continuous right action of $G$ on the topological space $\Gamma$, $\beta:G\times \Gamma\longrightarrow \Gamma$. Both actions $\alpha$ and $\beta$ are related in the following way: for every $\gamma\in\Gamma$ and every $g\in G$, we have $\gamma g=\alpha_\gamma(g)\beta_g(\gamma)$. In particular, if $e\in\Gamma$ denotes the identity element of $\Gamma$, then $\beta_g(e)=e$, for all $g\in G$. Hence $\#[e]=1$, where $[e]\in\Gamma/G$ is the corresponding class in the quotient space. Observe that $\alpha_{e}=id_G$ and $\beta_{e}=id_{\Gamma}$, where $e$ denotes either the identity element in $\Gamma$ or in $G$, respectively. Notice that, since $\beta$ is continuous and $G$ is compact, then every orbit $[\gamma]$ in $\Gamma/G$ with $\gamma\in \Gamma$ is finite.
		
		For every class $[\gamma]\in \Gamma/G$, we define the following \emph{clopen} subsets of $G$ (see \cite{FimaBiproduit} for more details) $A_{r,s}:=\{g\in G : \beta_g(r)=s\}$, for every $r,s\in [\gamma]$. Consider as well its characteristic function, say $\mathbbold{1}_{A_{r,s}}=: \mathbbold{1}_{r,s}$, for all $r,s\in[\gamma]$. We can show that $\Big(\mathbbold{1}_{r,s}\Big)_{r,s\in[\gamma]}\in \mathcal{M}_{\#[\gamma]}(\mathbb{C})\otimes C(G)$ is a magic unitary and a unitary representation of $G$ (see \cite{FimaBiproduit} for more details).
		
		In this situation, we can construct the compact bicrossed product of the matched pair $(\Gamma, G)$ and it is denoted by $\mathbb{F}=\Gamma {_\alpha}\bowtie_{\beta}G$, where $C(\mathbb{F})=\Gamma\underset{\alpha, m}{\ltimes} C(G)$ (see \cite{FimaBiproduit} for more details). By definition of the crossed product by a discrete group we have a unital faithful $*$-homomorphism $\pi: C(G)\longrightarrow C(\mathbb{F})$ and a group homomorphism $u: \Gamma\longrightarrow \mathcal{U}(C(\mathbb{F}))$ defined by $u_{\gamma}:=\lambda_{\gamma}\otimes id_{C(G)}$, for all $\gamma\in\Gamma$ such that $C(\mathbb{F})\equiv\Gamma\underset{\alpha, m}{\ltimes}C(G)=C^*\langle \pi(f)u_{\gamma}: f\in C(G), \gamma\in\Gamma \rangle$. The co-multiplication $\Theta$ of $\mathbb{F}$ is such that $\Theta\circ \pi=(\pi\otimes\pi)\circ\Delta_G$ and $\Theta(u_{\gamma})=\underset{r\in [\gamma]}{\sum} u_{\gamma}\alpha(\mathbbm{1}_{\gamma,r})\otimes u_{r}$ for all $\gamma\in\Gamma$.


	\begin{remSec}\label{rem.QuantumSubgroupsBicrossed}
		As in Remarks \ref{rem.QuantumSubgroups} we observe that $\hat{G}$ is a quantum subgroup of $\hat{\mathbb{F}}$ with canonical surjection given by $\rho_{\hat{G}}:=\varepsilon_{\Gamma}\otimes id_{C^*_r(G)}$, where $\varepsilon_{\Gamma}$ is the co-unit of $\Gamma$. As a result, if $(A, \delta)$ is any $\hat{\mathbb{F}}$-$C^*$-algebra, then $(A, \delta_{\hat{G}})$ is a $\hat{G}$-$C^*$-algebra with $\delta_{\hat{G}}=(\rho_{\hat{G}}\otimes id_A)\circ \delta$. Notice that $\Gamma$ is not in general a quantum subgroup of $\widehat{\mathbb{F}}$. Indeed, the canonical injection $\iota^m_\Gamma:C^*_m(\Gamma)\hookrightarrow C_m(\mathbb{F})$ does not intertwine the corresponding co-multiplications.
	\end{remSec}
	
	\bigskip
	In order to legitimate the Baum-Connes property formulation for the dual of a compact bicrossed product $\mathbb{F}=\Gamma {_\alpha}\bowtie_{\beta}G$, we need $\widehat{\mathbb{F}}$ to be torsion-free. In this way, we do the following crucial observation.
	\begin{proSec}\label{pro.BicrossedBetaTrivial}
		Let $\mathbb{F}=\Gamma {_\alpha}\bowtie_{\beta}G$ be a compact bicrossed product of the matched pair $(\Gamma, G)$. If $\widehat{\mathbb{F}}$ is torsion-free, then the action $\beta$ is trivial. Consequently, $\mathbb{F}=\Gamma \underset{\alpha}{\ltimes}\mathbb{G}$ is a quantum semi-direct product with $\mathbb{G}:=(C(G), \Delta)$.
	\end{proSec}
	\begin{proof}
		Let $G_0$ be the connected component of the identity element $e$, which is always a closed normal subgroup of $G$. Consequently, $G/G_{0}$ is a finite group because $G$ is supposed to be compact. Its dual is therefore a finite discrete quantum subgroup of $\widehat{G}$. The latter is a discrete quantum subgroup of $\widehat{\mathbb{F}}$ as explained in Remark \ref{rem.QuantumSubgroupsBicrossed}. Since $\widehat{\mathbb{F}}$ is torsion-free by assumption, then $G/G_0$ must be the trivial group (recall Remark \ref{rem.TorsionFiniteSubgroup}). Hence $G$ must be connected, which forces $\beta$ to be the trivial action.
	\end{proof}
	
	As we have pointed out several times, the torsion-freeness assumption is needed whenever we work with the (current) quantum Baum-Connes property for discrete quantum groups. So, this hypothesis forces the compact bicrossed product case to became a quantum semi-direct product. Therefore, Theorem \ref{theo.FinalTheorem} and Theorem \ref{theo.FinalTheoremStrong} still hold. In this sense, the torsion case is the interesting one. The analogous strategy used in Section \ref{sec.BCQuantumsemi-directProduct} could be applied for a compact bicrossed product because its representation theory can be described explicitly as in the quantum semi-direct product case. However, the presence of a non trivial action $\beta$ makes this task more involved. It is reasonable to expect that the same stabilization property holds for any compact bicrossed product (not necessarily torsion-free) once the Baum-Connes property can be formulated without the torsion-freeness assumption.
	
	The $K$-amenability property can be established independently of this assumption. Notice here that $G$ is automatically amenable (so $K$-amenable) because it is a classical compact group.
	\begin{theoSec}\label{theo.KAmenabilityBicrossed}
		Let $\mathbb{F}=\Gamma{_\alpha}\bowtie_{\beta} G$ be a compact bicrossed product. Then $\mathbb{F}$ is co-$K$-amenable if and only if $\Gamma$ is $K$-amenable.
	\end{theoSec}
	\begin{proof}	
		Assume that $\mathbb{F}$ is co-$K$-amenable. This means that the canonical surjection $\tau_{\mathbb{F}}: C_m(\mathbb{F})\twoheadrightarrow C_r(\mathbb{F})$ induces a $KK$-equivalence, that is, the corresponding Kasparov triple $[\tau_{\mathbb{F}}]\in KK(C_m(\mathbb{F}), C_r(\mathbb{F}))$ is invertible. Let $\mathcal{X}\in KK(C_r(\mathbb{F}), C_m(\mathbb{F}))$ be its inverse, so that we have $[\tau_{\mathbb{F}}]\underset{C_r(\mathbb{F})}{\otimes} \mathcal{X}=1_{C_m(\mathbb{F})}$ and $\mathcal{X}\underset{C_m(\mathbb{F})}{\otimes} [\tau_{\mathbb{F}}]=1_{C_r(\mathbb{F})}$. 
		
		Consider the canonical $\alpha$-invariant character $\varepsilon_G: C(G) \longrightarrow \mathbb{C}$, $f\mapsto \varepsilon(f):=f(e)$. Consider then the unital $*$-homomorphisms $\varepsilon^m_{\Gamma}:=\Gamma\underset{\alpha, m}{\ltimes}\varepsilon_G: C_m(\mathbb{F})\longrightarrow C^*_m(\Gamma)$ and  $\varepsilon^r_{\Gamma}:=\Gamma\underset{\alpha, r}{\ltimes}\varepsilon_G: C_r(\mathbb{F})\longrightarrow C^*_r(\Gamma)$, which are such that $\tau_{\Gamma}\circ\varepsilon^m_{\Gamma}=\varepsilon^r_{\Gamma}\circ\tau_{\mathbb{F}}$, where $\tau_\Gamma:C^*_m(\Gamma)\longrightarrow C^*_r(\Gamma)$ denotes the canonical surjection.
		
		Recall that $C_m(\mathbb{F})=\Gamma\underset{\alpha, m}{\ltimes}C(G)=C^*\langle \pi(f)u_{\gamma}: f\in C(G), \gamma\in\Gamma\rangle $. So, with the help of the $\alpha$-invariant character above, we identify $C^*_m(\Gamma)$ with the subalgebra of $C_m(\mathbb{F})$ generated by $\{u_{\gamma} : \gamma\in\Gamma\}$ by universal property (see Remark 3.6 in \cite{FimaBiproduit} for more details). Hence, we consider the canonical injection $\iota_m: C^*_m(\Gamma)\hookrightarrow C_m(\mathbb{F})$, which is such that $\varepsilon^m_{\Gamma}\circ \iota_m=id_{C^*_m(\Gamma)}$.
		
		Likewise, recall that $C_r(\mathbb{F})=\Gamma\underset{\alpha, r}{\ltimes}C(G)=C^*\langle \pi(f)u_{\gamma}: f\in C(G), \gamma\in\Gamma\rangle $ is equipped with a conditional expectation $E:\Gamma\underset{\alpha, r}{\ltimes}C(G)\longrightarrow C(G)$, which restricted to the subalgebra generated by $\{u_{\gamma}\ :\ \gamma\in\Gamma\}$ is just $E(u_{\gamma})=\delta_{\gamma,e}\in\mathbb{C}$. Remember as well that $u_{\gamma}=\lambda_{\gamma}\otimes id_{C(G)}$ in $\Gamma\ltimes_{\alpha, r}C(G)\subset \mathcal{L}_{C(G)}(l^2(\Gamma)\otimes C(G))$. Hence this subalgebra is identified canonically to $C^*_r(\Gamma)=\Gamma\underset{tr, r}{\ltimes}\mathbb{C}$ by universal property. Hence, we consider the canonical injection $\iota_r: C^*_r(\Gamma)\hookrightarrow C_r(\mathbb{F})$, which is such that $\varepsilon^r_{\Gamma}\circ \iota_r=id_{C^*_r(\Gamma)}$
		
		By construction we observe that $\tau_{\mathbb{F}}\circ \iota_m=\iota_r\circ\tau_{\Gamma}$. Define the element $\mathcal{Y}:=[\iota_r]\underset{C_r(\mathbb{F})}{\otimes} \mathcal{X}\underset{C_m(\mathbb{F})}{\otimes} [\varepsilon^m_{\Gamma}]\in KK(C^*_r(\Gamma), C^*_m(\Gamma))$. A straightforward computation, using all the preceding relations, yields that $\mathcal{Y}$ is the inverse of the induced element $[\tau_{\Gamma}]\in KK(C^*_m(\Gamma), C^*_r(\Gamma))$. Hence $\Gamma$ is $K$-amenable. 

		Conversely, assume that $\Gamma$ is $K$-amenable. By virtue of the $K$-amenability characterization of J. Cuntz (see Theorem 2.1 in \cite{CuntzKmoyenabilite}), the surjection $\tau_{\mathbb{F}}: \Gamma\underset{\alpha, m}{\ltimes} C(G)\twoheadrightarrow \Gamma\underset{\alpha, r}{\ltimes} C(G)$ induces a $KK$-equivalence. Hence $\widehat{\mathbb{F}}$ is co-$K$-amenabble.
	\end{proof}
	\begin{remSec}
		It is important to notice that the preceding proof can be simplified by using the same argument as in Theorem \ref{theo.KAmenability} together with Remark \ref{rem.QuantumSubgroupsBicrossed}.
	\end{remSec}

\section{\textsc{Quantum Direct Product}}\label{sec.QuantumDirectProduct}
		Let $\mathbb{G}=(C(\mathbb{G}),\Delta_{\mathbb{G}})$ and $\mathbb{H}=(C(\mathbb{H}),\Delta_{\mathbb{H}})$ be two compact quantum groups. We construct the \emph{quantum direct product of $\mathbb{G}$ and $\mathbb{H}$} and it is denoted by $\mathbb{F}:=\mathbb{G} \times \mathbb{H}$, where $C(\mathbb{F})=C_m(\mathbb{G})\underset{\max}{\otimes} C_m(\mathbb{H})$ (see \cite{WangSemidirect} for more details). The co-multiplication $\Theta$ of $\mathbb{F}$ is such that $\Theta(a\otimes b)=\Delta_{\mathbb{G}}(a)_{(1)}\otimes \Delta_{\mathbb{H}}(b)_{(1)}\otimes \Delta_{\mathbb{G}}(a)_{(2)}\otimes \Delta_{\mathbb{H}}(b)_{(2)}$, for all $a\in C_m(\mathbb{G})$ and all $b\in C_m(\mathbb{H})$.
		
	We have $Irr(\mathbb{F})=\big[Irr(\mathbb{G})\big]_{13}\ \big[Irr(\mathbb{H})\big]_{24}$, which means precisely that if $y\in Irr(\mathbb{F})$, then there exist unique $x\in Irr(\mathbb{G})$ and $z\in Irr(\mathbb{H})$ such that $w^y:= w^{(x,z)}=\big[w^x\big]_{13}\big[w^z\big]_{24}\in\mathcal{B}(H_x\otimes H_z)\otimes C(\mathbb{F})$, where $\big[w^x\big]_{13}$ and $\big[w^z\big]_{24}$ are the corresponding legs of $w^x$ and $w^z$, respectively inside $\mathcal{B}(H_x)\otimes\mathcal{B}(H_z)\otimes C_m(\mathbb{G})\underset{\max}{\otimes}C_m(\mathbb{H})$. As a result, we obtain the following decomposition $c_0(\widehat{\mathbb{F}})\cong c_0(\Gamma)\otimes c_0(\widehat{\mathbb{G}})$.
	
	The fusion rules for a quantum direct product can be easily established. More precisely, let $x,x'\in Irr(\mathbb{G})$ and $z,z'\in Irr(\mathbb{H})$ be irreducible representations of $\mathbb{G}$ and $\mathbb{H}$ and consider the corresponding irreducible representations of $\mathbb{F}$, say $y:=(x,z), y':=(x', z')\in Irr(\mathbb{F})$. Thanks to the description of $Irr(\mathbb{F})$ we know that $w^y= \big[w^x\big]_{13}\big[w^z\big]_{24}\mbox{ and } w^{y'}= \big[w^{x'}\big]_{13}\big[w^{z'}\big]_{24}$, where the legs are considered inside $\mathcal{B}(H_x)\otimes\mathcal{B}(H_z)\otimes C_m(\mathbb{G})\underset{\max}{\otimes}C_m(\mathbb{H})$ and $\mathcal{B}(H_{x'})\otimes\mathcal{B}(H_{z'})\otimes C_m(\mathbb{G})\underset{\max}{\otimes}C_m(\mathbb{H})$, respectively. The flip map $H_z\otimes H_{x'}\longrightarrow H_{x'}\otimes H_z$ yields the following obvious identification $w^{y\otop y'}:=w^{y}\otop w^{y'}=\big[w^{x}\otop w^{x'}\big]_{13}\big[w^z\otop w^{z'}\big]_{24}$.
	
	\begin{remsSec}\label{rem.QuantumSubgroupsDirectProduct}
		\begin{enumerate}
			\item It is important to observe that $\widehat{\mathbb{G}}$ and $\widehat{\mathbb{H}}$ are quantum subgroups of $\widehat{\mathbb{F}}$. Indeed, the canonical injections $\iota^r_{\mathbb{G}}: C_r(\mathbb{G})\hookrightarrow C_r(\mathbb{F})\mbox{ and } \iota^r_{\mathbb{H}}: C_r(\mathbb{H})\hookrightarrow C_r(\mathbb{F})$ intertwine the corresponding co-multiplications by construction.
			\item Moreover, the representation theory of $\mathbb{F}$ yields that $\widehat{\mathbb{G}}$ and $\widehat{\mathbb{H}}$ are divisible in $\widehat{\mathbb{F}}$. Namely, given an irreducible representation $y:=(x, z)\in Irr(\mathbb{F})$ with $x\in Irr(\mathbb{G})$ and $z\in Irr(\mathbb{H})$, then $x=(x, \epsilon), z=(\epsilon,z)\in [y]$ in $\sim\backslash Irr(\mathbb{F})$. For all $s\in Irr(\mathbb{G})$ we have that $s\otop (\epsilon,z)=(s,\epsilon)\otop (\epsilon,z)=(s,z)\in Irr(\mathbb{F})$, which shows that $\widehat{\mathbb{G}}$ is divisible in $\widehat{\mathbb{F}}$. For all $s\in Irr(H)$ we have that $(x, \epsilon)\otop s=(x, \epsilon)\otop (\epsilon,s)=(x, s)$, which shows that $\widehat{\mathbb{H}}$ is divisible in $\widehat{\mathbb{F}}$.
		\end{enumerate}
	\end{remsSec}

\subsection{Torsion phenomena}
	The description of the irreducible representations of $\mathbb{F}:=\mathbb{G} \times \mathbb{H}$ allows to obtain the following decomposition of its fusion ring: $Fus(\widehat{\mathbb{F}})=Fus(\widehat{\mathbb{G}})\otimes Fus(\widehat{\mathbb{H}})$. Namely, we have $Irr(\mathbb{F})=\big[Irr(\mathbb{G})\big]_{13} \big[Irr(\mathbb{H})\big]_{24}$, so that $Irr(\mathbb{F})$ can be regarded as the tensor product of $Irr(\mathbb{G})$ and $Irr(\mathbb{H})$ as \emph{based rings} (recall Section \ref{sec.Torsion}). Indeed, given $y\in Irr(\mathbb{F})$, take $x\in Irr(\mathbb{G})$ and $z\in Irr(\mathbb{H})$ such that $y=\big[x\big]_{13} \big[z\big]_{24}$. If $w^{\overline{y}}$, $w^{\overline{x}}$ and $w^{\overline{z}}$ are representatives of $\overline{y}$, $\overline{x}$ and $\overline{z}$, respectively; then we have $\epsilon_{\mathbb{F}}=\epsilon_{\mathbb{G}}\otop \epsilon_{\mathbb{H}}$, $w^{\overline{y}}=\overline{\big[w^{x}\big]_{13}\big[w^{z}\big]_{24}}=\big[w^{\overline{x}}\big]_{13} \big[w^{\overline{z}}\big]_{24}$ and $d(w^{y}= \big[w^{x}\big]_{13}\big[w^{z}\big]_{24})=d(w^{x})d(w^z)$. In conclusion, the decomposition $Fus(\widehat{\mathbb{F}})=Fus(\widehat{\mathbb{G}})\otimes Fus(\widehat{\mathbb{H}})$ holds. Accordingly, we have the following result.
	\begin{pro}\label{pro.StrongTorsionQuantumDirect}
		Let $\mathbb{F}:=\mathbb{G} \times \mathbb{H}$ be the quantum direct product of $\mathbb{G}$ and $\mathbb{H}$. $\widehat{\mathbb{F}}$ is strong torsion-free if and only if $\widehat{\mathbb{G}}$ and $\widehat{\mathbb{H}}$ are strong torsion-free.
	\end{pro}
	\begin{proof}
		Assume that $\widehat{\mathbb{G}}$ and $\widehat{\mathbb{H}}$ are strong torsion-free. Then they are in particular torsion-free. By Remark \ref{rem.TorsionFiniteSubgroup}, $\widehat{\mathbb{G}}$ and $\widehat{\mathbb{H}}$ can \emph{not} contain finite discrete quantum subgroups. Hence, $Fus(\widehat{\mathbb{G}})$ and $Fus(\widehat{\mathbb{H}})$ can \emph{not} contain finite fusion subrings. Theorem \ref{theo.TensorProductFusionRings} assures thus that $Fus(\widehat{\mathbb{G}})\otimes Fus(\widehat{\mathbb{H}})=Fus(\widehat{\mathbb{F}})$ is torsion-free. In other words, $\widehat{\mathbb{F}}$ is strong torsion-free. The converse is a consequence of \cite{YukiKenny} as explained in Remark \ref{rem.StrongTorsionFreeSubroups} because $\widehat{\mathbb{G}}$ and $\widehat{\mathbb{H}}$ are divisible in $\widehat{\mathbb{F}}$.
	\end{proof}
	
	It is interesting to study directly the torsion-freeness in the sense of Meyer-Nest, which has already been done in \cite{YukiKenny}.
	\begin{theo}[Y. Arano and K. De Commer, \cite{YukiKenny}]\label{theo.TorsionQuantumDirect}
		Let $\mathbb{F}:=\mathbb{G} \times \mathbb{H}$ be the quantum direct product. If $\widehat{\mathbb{G}}$ and $\widehat{\mathbb{H}}$ are torsion-free, then $\widehat{\mathbb{F}}$ is torsion-free.
	\end{theo}
	
	\begin{note}
		The converse of the preceding statement would be true whenever the torsion-freeness is preserved under divisible discrete quantum subgroups as conjectured in Section \ref{sec.Torsion}.
	\end{note}
	
\subsection{The Baum-Connes property}
	Let us adapt the notations from Section \ref{pro.ConesTensorProduct} for a quantum direct product $\mathbb{F}:=\mathbb{G} \times \mathbb{H}$. In order to formulate the quantum Baum-Connes property we assume that $\widehat{\mathbb{F}}$, $\widehat{\mathbb{G}}$ and $\widehat{\mathbb{H}}$ are all torsion-free.

	Consider the equivariant Kasparov categories associated to $\widehat{\mathbb{F}}$, $\widehat{\mathbb{G}}$ and $\widehat{\mathbb{H}}$, say $\mathscr{K}\mathscr{K}^{\widehat{\mathbb{F}}}$, $\mathscr{K}\mathscr{K}^{\widehat{\mathbb{G}}}$ and $\mathscr{K}\mathscr{K}^{\widehat{\mathbb{H}}}$, respectively; with canonical suspension functors denoted by $\Sigma$. Consider the usual complementary pair of localizing subcategories in $\mathscr{K}\mathscr{K}^{\widehat{\mathbb{F}}}$, $\mathscr{K}\mathscr{K}^{\widehat{\mathbb{G}}}$ and $\mathscr{K}\mathscr{K}^{\widehat{\mathbb{H}}}$, say $(\mathscr{L}_{\widehat{\mathbb{F}}}, \mathscr{N}_{\widehat{\mathbb{F}}})$, $(\mathscr{L}_{\widehat{\mathbb{G}}}, \mathscr{N}_{\widehat{\mathbb{G}}})$ and $(\mathscr{L}_{\widehat{\mathbb{H}}}, \mathscr{N}_{\widehat{\mathbb{H}}})$, respectively. The canonical triangulated functors associated to these complementary pairs will be denoted by $(L,N)$, $(L', N')$ and $(L'', N'')$, respectively. Next, consider the homological functors defining the \emph{quantum} Baum-Connes assembly maps for $\widehat{\mathbb{F}}$, $\widehat{\mathbb{G}}$ and $\widehat{\mathbb{H}}$. Namely,
	 $$
		\begin{array}{rccl}
			F:&\mathscr{K}\mathscr{K}^{\widehat{\mathbb{F}}}& \longrightarrow &\mathscr{A}b^{\mathbb{Z}/2}\\
			&(C,\delta) & \longmapsto &F(C):=K_{*}(\widehat{\mathbb{F}}\ltimes_{\delta, r} C)
		\end{array}
	$$
	$$
		\begin{array}{rccllrccl}
			F':&\mathscr{K}\mathscr{K}^{\widehat{\mathbb{G}}}& \longrightarrow &\mathscr{A}b^{\mathbb{Z}/2}&&F'':&\mathscr{K}\mathscr{K}^{\widehat{\mathbb{H}}}& \longrightarrow & \mathscr{A}b^{\mathbb{Z}/2}\\
			&(A,\alpha) & \longmapsto &F'(A):=K_{*}(\widehat{\mathbb{G}}\underset{\alpha, r}{\ltimes} A)&&&(B,\beta) & \longmapsto &F''(B):= K_{*}(\widehat{\mathbb{H}} \underset{\beta,r}{\ltimes} B)
		\end{array}
	$$
	
	The quantum assembly maps for $\widehat{\mathbb{F}}$, $\widehat{\mathbb{G}}$ and $\widehat{\mathbb{H}}$ are given by the following natural transformations $\eta^{\widehat{\mathbb{F}}}: \mathbb{L}F\longrightarrow F$, $\eta^{\widehat{\mathbb{G}}}: \mathbb{L}F'\longrightarrow F'$ and $\eta^{\widehat{\mathbb{H}}}: \mathbb{L}F''\longrightarrow F''$ (where that $\mathbb{L}F=F\circ L$, $\mathbb{L}F'= F'\circ L'$ and $\mathbb{L}F''= F''\circ L''$).
	
	Consider an object in $\mathscr{K}\mathscr{K}^{\widehat{\mathbb{G}}}\times \mathscr{K}\mathscr{K}^{\widehat{\mathbb{H}}}$, say $(A,\alpha)\times (B,\beta)\in Obj.(\mathscr{K}\mathscr{K}^{\widehat{\mathbb{G}}})\times Obj.(\mathscr{K}\mathscr{K}^{\widehat{\mathbb{H}}})$. The tensor product $A\otimes B$ is a $\widehat{\mathbb{F}}$-$C^*$-algebra with action $\alpha\otimes\beta$. Let $(A',\alpha')\times (B',\beta')\in Obj.(\mathscr{K}\mathscr{K}^{\widehat{\mathbb{G}}})\times Obj.(\mathscr{K}\mathscr{K}^{\widehat{\mathbb{H}}})$ an other object in $\mathscr{K}\mathscr{K}^{\widehat{\mathbb{G}}}\times \mathscr{K}\mathscr{K}^{\widehat{\mathbb{H}}}$. Let $\mathcal{X}:=((H,\delta_H),\pi,F)\in KK^{\widehat{\mathbb{G}}}(A,A')$ and $\mathcal{Y}:=((H',\delta_{H'}), \pi', F')\in KK^{\widehat{\mathbb{H}}}(B,B')$ be two Kasparov triples so that $\mathcal{X}\times \mathcal{Y}\in KK^{\widehat{\mathbb{G}}}(A,A')\times KK^{\widehat{\mathbb{H}}}(B,B')$ is a homomorphism between $(A,B)$ and $(A',B')$ in $\mathscr{K}\mathscr{K}^{\widehat{\mathbb{G}}}\times \mathscr{K}\mathscr{K}^{\widehat{\mathbb{H}}}$. Define the tensor product of these Kasparov triples in the following way
	$$\mathcal{Z}(\mathcal{X},\mathcal{Y}):=\tau^{B}_R(\mathcal{X})\otimes_{A'\otimes B}\tau^{A'}_L(\mathcal{Y})=:\mathcal{X}\otimes \mathcal{Y}\in KK(A\otimes B, A'\otimes B')\mbox{,}$$
	where $\tau^{B}_R(\mathcal{X}):=\mathcal{X}\otimes B$ and $\tau^{A'}_L(\mathcal{Y}):=A'\otimes \mathcal{Y}$ are the corresponding exterior tensor product Kasparov triples (see \cite{Blackadar} for more details). Moreover, since $\tau^B_R(\mathcal{X})\in KK^{\widehat{\mathbb{F}}}(A\otimes B,A'\otimes B)$ and $\tau^{A'}_L(\mathcal{Y})\in KK^{\widehat{\mathbb{F}}}(A'\otimes B,A'\otimes B')$, then $\mathcal{X}\otimes \mathcal{Y}$ is also equipped with an action of $\widehat{\mathbb{F}}$ given by Kasparov product.
	
	In other words, it is licit to consider the following functor:
	$$
		\begin{array}{rccl}
			\mathcal{Z}:&\mathscr{K}\mathscr{K}^{\widehat{\mathbb{G}}}\times \mathscr{K}\mathscr{K}^{\widehat{\mathbb{H}}}& \longrightarrow &\mathscr{K}\mathscr{K}^{\widehat{\mathbb{F}}}\\
			&(A,\alpha)\times (B,\beta) & \longmapsto &\mathcal{Z}(A, B):= (C:=A\otimes B, \delta:=\alpha\otimes \beta)
		\end{array}
	$$
	
	\begin{lem}\label{lem.FunctorPreservingLDirectProduct}
		The functor $\mathcal{Z}$ is such that $\mathcal{Z}(\mathscr{L}_{\widehat{\mathbb{G}}}\times \mathscr{L}_{\widehat{\mathbb{H}}})\subset \mathscr{L}_{\widehat{\mathbb{F}}}$ and $\mathcal{Z}(\mathscr{N}_{\widehat{\mathbb{G}}}\times \mathscr{N}_{\widehat{\mathbb{H}}})\subset \mathscr{N}_{\widehat{\mathbb{F}}}$.
		
		If $(A_0, \alpha_0)\in Obj(\mathscr{K}\mathscr{K}^{\widehat{\mathbb{G}}})$ is a given $\widehat{\mathbb{G}}$-$C^*$-algebra, the functor
		$$
		\begin{array}{rccl}
			{_{A_0}}\mathcal{Z}:&\mathscr{K}\mathscr{K}^{\widehat{\mathbb{H}}}& \longrightarrow &\mathscr{K}\mathscr{K}^{\widehat{\mathbb{F}}}\\
			&(B,\beta) & \longmapsto &{_{A_0}}\mathcal{Z}(B):= \mathcal{Z}(A_0, B)
		\end{array}
		$$
		is triangulated such that ${_{A_0}}\mathcal{Z}(\mathscr{N}_{\widehat{\mathbb{H}}})\subset \mathscr{N}_{\widehat{\mathbb{F}}}$.
		
	\end{lem}
	\begin{proof}
		Firstly, let us show that $\mathcal{Z}(\mathscr{L}_{\widehat{\mathbb{G}}}\times \mathscr{L}_{\widehat{\mathbb{H}}})\subset \mathscr{L}_{\widehat{\mathbb{F}}}$. Namely, since all our discrete quantum groups are supposed to be torsion-free, then we know that $\mathscr{L}_{\widehat{\mathbb{G}}}$, $\mathscr{L}_{\widehat{\mathbb{H}}}$ and $\mathscr{L}_{\widehat{\mathbb{F}}}$ are the localizing subcategories generated by the objects of the form $c_0(\widehat{\mathbb{G}})\otimes C_1$, $c_0(\widehat{\mathbb{H}})\otimes C_2$ and $c_0(\widehat{\mathbb{F}})\otimes C_3$ in $\mathscr{K}\mathscr{K}^{\widehat{\mathbb{G}}}$, $\mathscr{K}\mathscr{K}^{\widehat{\mathbb{H}}}$ and $\mathscr{K}\mathscr{K}^{\widehat{\mathbb{F}}}$, respectively where $C_1, C_2, C_3\in Obj(\mathscr{K}\mathscr{K})$. Recall as well that $c_0(\widehat{\mathbb{F}})=c_0(\widehat{\mathbb{G}})\otimes c_0(\widehat{\mathbb{H}})$ by virtue of the representation theory of $\mathbb{F}=\mathbb{G}\times\mathbb{H}$ (see Section \ref{sec.QuantumDirectProduct}). Hence we write
		\begin{equation*}
			\begin{split}
				\mathcal{Z}\big(c_0(\widehat{\mathbb{G}})\otimes C_1, c_0(\widehat{\mathbb{H}})\otimes C_2\big)&=c_0(\widehat{\mathbb{G}})\otimes C_1\otimes c_0(\widehat{\mathbb{H}})\otimes C_2\\
				&\cong c_0(\widehat{\mathbb{G}})\otimes c_0(\widehat{\mathbb{H}})\otimes C_1\otimes C_2= c_0(\widehat{\mathbb{F}})\otimes C_3\in Obj(\mathscr{L}_{\widehat{\mathbb{F}}})\mbox{,}
			\end{split}
		\end{equation*}
		where $C_3:=C_1\otimes C_2\in Obj(\mathscr{K}\mathscr{K})$. This shows that $\mathcal{Z}$ sends generators of $\mathscr{L}_{\widehat{\mathbb{G}}}\times \mathscr{L}_{\widehat{\mathbb{H}}}$ to generators of $\mathscr{L}_{\widehat{\mathbb{F}}}$, so $\mathcal{Z}$ sends generators of $\mathscr{L}_{\widehat{\mathbb{G}}}\times \mathscr{L}_{\widehat{\mathbb{H}}}$ to $\mathscr{L}_{\widehat{\mathbb{F}}}$. Denote by $\mathcal{S}$ the class of objects in $\mathscr{K}\mathscr{K}^{\widehat{\mathbb{F}}}$ of the form $\mathcal{Z}(\mbox{``generator of $\mathscr{L}_{\widehat{\mathbb{G}}}\times \mathscr{L}_{\widehat{\mathbb{H}}}$''})$. It is clear that $\mathcal{S}\subset \mathcal{Z}(\mathscr{L}_{\widehat{\mathbb{G}}}\times \mathscr{L}_{\widehat{\mathbb{H}}})$. Consider $\langle \mathcal{S} \rangle$ the localising subcategory generated by $\mathcal{S}$, that is, the smallest triangulated subcategory containing the objects of $\mathcal{S}$ and stable with respect to countable direct sums. $\mathscr{L}_{\widehat{\mathbb{F}}}$ is a triangulated subcategory containing the objects of $\mathcal{S}$ by the discussion above and stable with respect to countable direct sums by definition. Hence, by minimality, we have that $\langle \mathcal{S} \rangle \subset \mathscr{L}_{\widehat{\mathbb{F}}}$. Finally, since $\mathcal{Z}$ sends generators of $\mathscr{L}_{\widehat{\mathbb{G}}}\times \mathscr{L}_{\widehat{\mathbb{H}}}$ to generators of $\mathscr{L}_{\widehat{\mathbb{F}}}$, $\mathcal{Z}$ is compatible with countable direct sums and $\mathcal{Z}(\mathscr{L}_{\widehat{\mathbb{G}}}\times \mathscr{L}_{\widehat{\mathbb{H}}})$ is a subcategory in $\mathscr{K}\mathscr{K}^{\widehat{\mathbb{F}}}$ containing $\mathcal{S}$ but not triangulated (because $\mathcal{Z}$ is not a triangulated functor), it must be $\mathcal{Z}(\mathscr{L}_{\widehat{\mathbb{G}}}\times \mathscr{L}_{\widehat{\mathbb{H}}})\subset \langle \mathcal{S} \rangle$, which yields the claim.
		
		
		Secondly, let us show that $\mathcal{Z}(\mathscr{N}_{\widehat{\mathbb{G}}}\times \mathscr{N}_{\widehat{\mathbb{H}}})\subset \mathscr{N}_{\widehat{\mathbb{F}}}$. For this we have to notice that the restriction functor is obviously compatible with the tensor functor $\mathcal{Z}$. Given $A\in Obj(\mathscr{N}_{\widehat{\mathbb{G}}})$ and $B\in Obj(\mathscr{N}_{\widehat{\mathbb{H}}})$, we can write $Res^{\widehat{\mathbb{F}}}_{\mathbb{E}}\big(\mathcal{Z}(A,B)\big)=Res^{\widehat{\mathbb{F}}}_{\mathbb{E}}(A\otimes B)=Res^{\widehat{\mathbb{G}}}_{\mathbb{E}}(A)\otimes Res^{\widehat{\mathbb{H}}}_{\mathbb{E}}(B)\cong 0$, so that $\mathcal{Z}(A,B)\in Obj(\mathscr{N}_{\widehat{\mathbb{F}}})$.
		
		Next, fix a $\widehat{\mathbb{G}}$-$C^*$-algebra $(A_0, \alpha_0)\in Obj(\mathscr{K}\mathscr{K}^{\widehat{\mathbb{G}}})$ and consider the functor ${_{A_0}}\mathcal{Z}$ of the statement (which is well defined on homomorphisms in an analogous way as $\mathcal{Z}$ by using the exterior tensor product of Kasparov triples). In order to show that ${_{A_0}}\mathcal{Z}$ is triangulated, we are going to show that ${_{A_0}}\mathcal{Z}$ is compatible with the suspension functors of the corresponding Kasparov categories and that ${_{A_0}}\mathcal{Z}$ preserves mapping cone triangles.
		
		For the first claim, given $(B,\beta)\in Obj(\mathscr{K}\mathscr{K}^{\widehat{\mathbb{H}}})$ we have
		\begin{equation*}
			\begin{split}
				{_{A_0}}\mathcal{Z}(\Sigma(B))&={_{A_0}}\mathcal{Z}\big(C_0(\mathbb{R})\otimes B\big)=A_0\otimes C_0(\mathbb{R})\otimes B\\
				&\cong C_0(\mathbb{R})\otimes A_0\otimes B\overset{(1)}{\cong}\Sigma(A_0\otimes B)=\Sigma({_{A_0}}\mathcal{Z}(B))\mbox{,}
			\end{split}
		\end{equation*}
		where the identification $(1)$ is simply induced by the canonical identification $A_0\otimes C_0\big([0,1], B\big)\cong C_0\big([0,1], A_0\otimes B\big)$. Let us show that ${_{A_0}}\mathcal{Z}$ preserves mapping cone triangles. Consider a mapping cone triangle in $\mathscr{K}\mathscr{K}^{\widehat{\mathbb{H}}}$, say $\Sigma(B')\longrightarrow C_{\phi}\longrightarrow B\overset{\phi}{\longrightarrow} B'$, where $\phi: B\longrightarrow B'$ is a $\widehat{\mathbb{H}}$-equivariant $*$-homomorphism. Apply the functor ${_{A_0}}\mathcal{Z}$ so that we obtain the following diagram $\Sigma(A_0\otimes B')\longrightarrow A_0\otimes C_{\phi}\longrightarrow A_0\otimes B\overset{id\otimes\phi}{\longrightarrow} A_0\otimes B'$, where $A_0\otimes C_{\phi}\cong C_{id\otimes \phi}$ by virtue of Proposition \ref{pro.ConesTensorProduct}. Hence the above diagram is again a mapping cone triangle in $\mathscr{K}\mathscr{K}^{\widehat{\mathbb{F}}}$. Moreover, if now $B\in Obj(\mathscr{N}_{\widehat{\mathbb{H}}})$, then $Res^{\widehat{\mathbb{F}}}_{\mathbb{E}}\big({_{A_0}}\mathcal{Z}(B)\big)=Res^{\widehat{\mathbb{G}}}_{\mathbb{E}}(A_0)\otimes Res^{\widehat{\mathbb{H}}}_{\mathbb{E}}(B)\cong 0$, which implies that ${_{A_0}}\mathcal{Z}(B)\in \mathscr{N}_{\widehat{\mathbb{H}}}$. 
	\end{proof}
	
	\begin{lem}\label{lem.InvertibleElementDirectProduct}
		Let $\mathbb{F}=\mathbb{G}\times \mathbb{H}$ be a quantum direct product of compact quantum groups such that $\widehat{\mathbb{F}}$, $\widehat{\mathbb{G}}$ and $\widehat{\mathbb{H}}$ are a torsion-free discrete quantum groups.
		
		\begin{enumerate}[i)]
			\item For all $\widehat{\mathbb{G}}$-$C^*$-algebra $(A,\alpha)$ and all $\widehat{\mathbb{H}}$-$C^*$-algebra $(B,\beta)$ there exists a Kasparov triple
		$$\psi\in KK^{\widehat{\mathbb{F}}}\big(L'(A)\otimes L''(B), L(A\otimes B)\big)$$
		such that the following diagram is commutative
		\begin{equation}\label{eq.CommutativeDiagramDirectProduct}
		\begin{gathered}
			\xymatrix@C=20mm@!R=15mm{
				\mbox{$\widehat{\mathbb{G}}\underset{r}{\ltimes} L'(A)\otimes \widehat{\mathbb{H}}\underset{r}{\ltimes} L''(B)$}\ar[d]_{\mbox{$\widehat{\mathbb{G}}\underset{r}{\ltimes} u'\otimes \widehat{\mathbb{H}}\underset{r}{\ltimes} u''$}}\ar[r]^-{\mbox{$\Psi$}}&\mbox{$\widehat{\mathbb{F}}\underset{r}{\ltimes} L(A\otimes B)$}\ar[d]^{\mbox{$\widehat{\mathbb{F}}\underset{r}{\ltimes} u$}}\\
				\mbox{$\widehat{\mathbb{G}}\underset{r}{\ltimes} A\otimes \widehat{\mathbb{H}}\underset{r}{\ltimes} B$}\ar[r]^{\mbox{$\cong$}}&\mbox{$\widehat{\mathbb{F}}\underset{r}{\ltimes}(A\otimes B)$}}
		\end{gathered}
		\end{equation}
		where $\Psi:=\widehat{\mathbb{F}}\underset{r}{\ltimes}\psi$ and $u'$, $u''$, $u$ are the Dirac homomorphisms for $A$, $B$, $A\otimes B$, respectively.
			\item For all $\widehat{\mathbb{G}}$-$C^*$-algebra $(A_0,\alpha_0)\in \mathscr{L}_{\widehat{\mathbb{G}}}$ and all $\widehat{\mathbb{H}}$-$C^*$-algebra $(B,\beta)$ there exists an \emph{invertible} Kasparov triple
		$${_{A_0}}\psi\in KK^{\widehat{\mathbb{F}}}\big(A_0\otimes L''(B), L(A_0\otimes B)\big)$$
		such that the following diagram is commutative
		\begin{equation}\label{eq.CommutativeDiagramDirectProductbis}
		\begin{gathered}
			\xymatrix@C=20mm@!R=15mm{
				\mbox{$\widehat{\mathbb{G}}\underset{r}{\ltimes} A_0\otimes \widehat{\mathbb{H}}\underset{r}{\ltimes} L''(B)$}\ar[d]_{\mbox{$\widehat{\mathbb{F}}\underset{r}{\ltimes} {_{A_0}}\mathcal{Z}(u'')$}}\ar[r]^-{\mbox{$\underset{\sim}{{_{A_0}}\Psi}$}}&\mbox{$\widehat{\mathbb{F}}\underset{r}{\ltimes} L(A_0\otimes B)$}\ar[d]^{\mbox{$\widehat{\mathbb{F}}\underset{r}{\ltimes} u$}}\\
				\mbox{$\widehat{\mathbb{G}}\underset{r}{\ltimes} A_0\otimes \widehat{\mathbb{H}}\underset{r}{\ltimes} B$}\ar[r]^{\mbox{$\cong$}}&\mbox{$\widehat{\mathbb{F}}\underset{r}{\ltimes}(A_0\otimes B)$}}
		\end{gathered}
		\end{equation}
		where ${_{A_0}}\Psi:=\widehat{\mathbb{F}}\underset{r}{\ltimes}{_{A_0}}\psi$ and $u''$, $u$ are the Dirac homomorphism for $B$, $A_0\otimes B$, respectively.
		\end{enumerate}
	\end{lem}
	\begin{proof}
		First of all, we recall that for all $\widehat{\mathbb{G}}$-$C^*$-algebra $(A,\alpha)$ and all $\widehat{\mathbb{H}}$-$C^*$-algebra $(B,\beta)$ we have a canonical $*$-isomorphism $\widehat{\mathbb{F}}\underset{\delta,r}{\ltimes} (A\otimes B)\cong \widehat{\mathbb{G}}\underset{\alpha,r}{\ltimes} A\otimes \widehat{\mathbb{H}}\underset{\beta,r}{\ltimes} B$ by Proposition \ref{cor.DirectProductTensorProduct}.
		\begin{enumerate}[i)]
			\item Given a $\widehat{\mathbb{G}}$-$C^*$-algebra $(A,\alpha)$, consider the corresponding $(\mathscr{L}_{\widehat{\mathbb{G}}}, \mathscr{N}_{\widehat{\mathbb{G}}})$-triangle, say $\Sigma(N'(A))\longrightarrow L'(A)\overset{u'}{\longrightarrow} A\longrightarrow N'(A)$. Given a $\widehat{\mathbb{H}}$-$C^*$-algebra $(B,\beta)$, consider the corresponding $(\mathscr{L}_{\widehat{\mathbb{H}}}, \mathscr{N}_{\widehat{\mathbb{H}}})$-triangle, say $\Sigma(N''(B))\longrightarrow L''(B)\overset{u''}{\longrightarrow} B\longrightarrow N''(B)$. Consider the $(\mathscr{L}_{\widehat{\mathbb{F}}}, \mathscr{N}_{\widehat{\mathbb{F}}})$-triangle of the $\widehat{\mathbb{F}}$-$C^*$-algebra $\mathcal{Z}(A\otimes B)=A\otimes B$, say $\Sigma(N(A\otimes B))\longrightarrow L(A\otimes B)\overset{u}{\longrightarrow} A\otimes B\longrightarrow N(A\otimes B)$.
			
			Let us fix the object $\mathcal{Z}(L'(A), L''(B))=L'(A)\otimes L''(B)=:T\in Obj(\mathscr{K}\mathscr{K}^{\widehat{\mathbb{F}}})$ and take the long exact sequence associated to the above triangle with respect to the object $T$. Namely,
		\begin{equation*}
			\begin{split}
				\ldots&\rightarrow KK^{\widehat{\mathbb{F}}}(T,\Sigma(N(A\otimes B)))\rightarrow KK^{\widehat{\mathbb{F}}}(T, L(A\otimes B))\overset{(u)_{*}}{\rightarrow}\\
				&\rightarrow KK^{\widehat{\mathbb{F}}}(T, A\otimes B)\rightarrow KK^{\widehat{\mathbb{F}}}(T, N(A\otimes B))\rightarrow \ldots				
			\end{split}
		\end{equation*}
						
		Since $(L'(A), L''(B))\in \mathscr{L}_{\widehat{\mathbb{G}}}\times \mathscr{L}_{\widehat{\mathbb{H}}}$, then $T\in \mathscr{L}_{\widehat{\mathbb{F}}}$ by Lemma \ref{lem.FunctorPreservingLDirectProduct}. But, by definition of complementary pair, we have $\mathscr{L}_{\widehat{\mathbb{F}}}\subset \mathscr{N}^{\vdash}_{\widehat{\mathbb{F}}}$. In particular, we obtain $KK^{\widehat{\mathbb{F}}}(T,\Sigma(N(A\otimes B)))=(0)=KK^{\widehat{\mathbb{F}}}(T, N(A\otimes B))$. Hence the above long exact sequence yields the isomorphism $KK^{\widehat{\mathbb{F}}}(T, L(A\otimes B))\overset{(u)_{*}}{\cong} KK^{\widehat{\mathbb{F}}}(T, A\otimes B)$. Take $\psi:=(u)^{-1}_{*}(\mathcal{Z}(u', u''))$. Consequently, we have $u\circ \psi=\mathcal{Z}(u', u'')=u'\otimes u''$, by definition.
		
		If we put $\Psi:=\widehat{\mathbb{F}}\underset{r}{\ltimes}\psi: \widehat{\mathbb{F}}\underset{r}{\ltimes} (L'(A)\otimes L''(B))\longrightarrow \widehat{\mathbb{F}}\underset{r}{\ltimes} L(A\otimes B)$, then the functoriality of constructions and the definition of the element $\psi$ yields straightforwardly the diagram $(\ref{eq.CommutativeDiagramDirectProduct})$ of the statement.
			\item Given a $\widehat{\mathbb{G}}$-$C^*$-algebra $(A_0,\alpha_0)\in \mathscr{L}_{\widehat{\mathbb{G}}}$ and a $\widehat{\mathbb{H}}$-$C^*$-algebra $(B,\beta)$, consider the corresponding $(\mathscr{L}, \mathscr{N})$-triangles as above. 
			
			The same argument as in $(i)$ by replacing $L'(A)$ by $A_0$ yields the existence of a Kasparov triple ${_{A_0}}\psi\in KK^{\widehat{\mathbb{F}}}\big(A_0\otimes L''(B), L(A_0\otimes B)\big)$ such that diagram $(\ref{eq.CommutativeDiagramDirectProductbis})$ of the statement is commutative. Namely, put ${_{A_0}}\psi:=(u)^{-1}_{*}({_{A_0}}\mathcal{Z}(u''))$.
			
			Let us show that the Kasparov triple ${_{A_0}}\psi$ is invertible. If we apply the triangulated (by Lemma \ref{lem.FunctorPreservingLDirectProduct}) functor ${_{A_0}}\mathcal{Z}$ to the $(\mathscr{L}_{\widehat{\mathbb{H}}},\mathscr{N}_{\widehat{\mathbb{H}}})$-triangle of $B$, we get the following distinguished triangle in $\mathscr{K}\mathscr{K}^{\widehat{\mathbb{F}}}$ $\Sigma(A_0\otimes N''(B))\longrightarrow A_0\otimes L''(B)\overset{{_{A_0}}\mathcal{Z}(u'')}{\longrightarrow} A_0\otimes B\longrightarrow A_0\otimes N''(B)$, where $A_0\otimes L''(B)=\mathcal{Z}(A_0, L''(B))\in \mathscr{L}_{\widehat{\mathbb{F}}}$ because $A_0\in \mathscr{L}_{\widehat{\mathbb{G}}}$, $L''(B)\in \mathscr{L}_{\widehat{\mathbb{H}}}$ and we apply Lemma \ref{lem.FunctorPreservingLDirectProduct} and $A_0\otimes N''(B)\in \mathscr{N}_{\widehat{\mathbb{F}}}$ because ${_{A_0}}\mathcal{Z}(\mathscr{N}_{\widehat{\mathbb{H}}})\subset \mathscr{N}_{\widehat{\mathbb{F}}}$ by Lemma \ref{lem.FunctorPreservingLDirectProduct}. In other words, the above is a $(\mathscr{L}_{\widehat{\mathbb{F}}},\mathscr{N}_{\widehat{\mathbb{F}}})$-triangle for $A_0\otimes B$. Hence, by uniqueness of this kind of distinguished triangles, we have the following isomorphism of distinguished triangles in $\mathscr{K}\mathscr{K}^{\widehat{\mathbb{F}}}$,
			$$
			\xymatrix@!C=30mm@R=15mm{
				\mbox{$\Sigma(A_0\otimes L''(B))$}\ar[d]^{\mbox{$\wr$}}\ar[r]&\mbox{$A_0\otimes L''(B)$}\ar[r]^{\mbox{${_{A_0}}\mathcal{Z}(u'')$}}\ar[d]^{\mbox{$\wr\ {_{A_0}}\psi$}}&\mbox{$A_0\otimes B$}\ar[r]\ar@{=}[d]^{\mbox{$id$}}&\mbox{$ A_0\otimes N''(B))$}\ar[d]^{\mbox{$\wr$}}\\
				\mbox{$\Sigma'(N(A_0\otimes B)$}\ar[r]&\mbox{$L(A_0\otimes B)$}\ar[r]_{\mbox{$u$}}&\mbox{$A_0\otimes B$}\ar[r]&\mbox{$N(A_0\otimes B)$}}
		$$
		which yields in particular the invertibility of ${_{A_0}}\psi$ as claimed.
		\end{enumerate}
	\end{proof}
	\begin{theo}\label{theo.SBCDirectProduct}
		Let $\mathbb{F}=\mathbb{G}\times \mathbb{H}$ be a quantum direct product of compact quantum groups such that $\widehat{\mathbb{F}}$, $\widehat{\mathbb{G}}$ and $\widehat{\mathbb{H}}$ are a torsion-free discrete quantum groups.
		\begin{enumerate}[i)]
			\item If $\widehat{\mathbb{G}}$ and $\widehat{\mathbb{H}}$ satisfy the strong quantum Baum-Connes property, then $\widehat{\mathbb{F}}$ satisfies the quantum Baum-Connes property with coefficients in $A\otimes B$, for every $A\in Obj(\mathscr{K}\mathscr{K}^{\widehat{\mathbb{G}}})$ and $B\in Obj(\mathscr{K}\mathscr{K}^{\widehat{\mathbb{H}}})$.
			\item If $\widehat{\mathbb{F}}$ satisfies the \emph{strong} quantum Baum-Connes property, then $\widehat{\mathbb{G}}$ and $\widehat{\mathbb{H}}$ satisfy the \emph{strong} quantum Baum-Connes property.
			\item If $\widehat{\mathbb{F}}$ satisfies the quantum Baum-Connes property, then $\widehat{\mathbb{G}}$ and $\widehat{\mathbb{H}}$ satisfy the quantum Baum-Connes property with coefficients.
		\end{enumerate}
		
			
			
	\end{theo}
	\begin{proof}
		\begin{enumerate}[i)]
			\item Given $A\in Obj(\mathscr{K}\mathscr{K}^{\widehat{\mathbb{G}}})$ and $B\in Obj(\mathscr{K}\mathscr{K}^{\widehat{\mathbb{H}}})$ consider the commutative diagram $(\ref{eq.CommutativeDiagramDirectProduct})$ of the preceding lemma,
			\begin{equation}\label{eq.CommutativeDiagramDirectProductProof}
			\begin{gathered}
			\xymatrix@C=20mm@!R=15mm{
				\mbox{$\widehat{\mathbb{G}}\underset{r}{\ltimes} L'(A)\otimes \widehat{\mathbb{H}}\underset{r}{\ltimes} L''(B)$}\ar[d]_{\mbox{$\widehat{\mathbb{G}}\underset{r}{\ltimes} u'\otimes \widehat{\mathbb{H}}\underset{r}{\ltimes} u''$}}\ar[r]^-{\mbox{$\Psi$}}&\mbox{$\widehat{\mathbb{F}}\underset{r}{\ltimes} L(A\otimes B)$}\ar[d]^{\mbox{$\widehat{\mathbb{F}}\underset{r}{\ltimes} u$}}\\
				\mbox{$\widehat{\mathbb{G}}\underset{r}{\ltimes} A\otimes \widehat{\mathbb{H}}\underset{r}{\ltimes} B$}\ar[r]^{\mbox{$\cong$}}&\mbox{$\widehat{\mathbb{F}}\underset{r}{\ltimes}(A\otimes B)$}}
			\end{gathered}
			\end{equation}
			where $\Psi=\widehat{\mathbb{F}}\underset{r}{\ltimes}\psi$ with $\psi=(u)^{-1}_{*}(\mathcal{Z}(u', u''))$.
			
			Since $\widehat{\mathbb{G}}$ satisfies the \emph{strong} quantum Baum-Connes property by assumption, then any Dirac homomorphism for $A$ is an isomorphism, that is, $L'(A)\overset{u'}{\cong} A\in \mathscr{L}_{\widehat{\mathbb{G}}}$. In other words, $u'\in KK^{\widehat{\mathbb{G}}}(L'(A), A)$ is an invertible Kasparov triple. Accordingly, $\tau^{L''(B)}_R(u')\in KK^{\widehat{\mathbb{F}}}(L'(A)\otimes L''(B), A\otimes L''(B))$ is also an invertible Kasparov triple.
			
			Recall that by definition we have $\mathcal{Z}(u', u'')=u'\otimes u''=\tau^{L''(B)}_R(u')\underset{A_0\otimes L''(B)}{\otimes}\tau^{A}_L(u'')\in KK^{\widehat{\mathbb{F}}}(L'(A)\otimes L''(B), A\otimes B)$ and ${_{A}}\mathcal{Z}(u'')=\tau^{A}_L(u'')\in KK^{\widehat{\mathbb{F}}}(A\otimes L''(B), A\otimes B)$.
			
			These two elements can be identified via the Kasparov multiplication $\tau^{L''(B)}_R(u')\underset{A\otimes L''(B)}{\otimes}(\ \cdot\ )$. In other words, the element $\psi$ can be identified to the element ${_{A}}\psi$ via this Kasparov multiplication. The latter is invertible by $(ii)$ of Lemma \ref{lem.InvertibleElementDirectProduct}, which yields that $\Psi$ in diagram $(\ref{eq.CommutativeDiagramDirectProductProof})$ is invertible as well.
			
			Next, since $\widehat{\mathbb{H}}$ satisfies the \emph{strong} quantum Baum-Connes property by assumption, then any Dirac homomorphism for $B$ is an isomorphism, that is, $L''(B)\overset{u''}{\cong} B\in \mathscr{L}_{\widehat{\mathbb{H}}}$. Hence, the element $\widehat{\mathbb{G}}\underset{r}{\ltimes} u'\otimes \widehat{\mathbb{H}}\underset{r}{\ltimes} u''$ of diagram $(\ref{eq.CommutativeDiagramDirectProductProof})$ is invertible. The commutativity of $(\ref{eq.CommutativeDiagramDirectProductProof})$ yields that $\widehat{\mathbb{F}}\underset{r}{\ltimes} u$ is an isomorphism in $\mathscr{K}\mathscr{K}^{\widehat{\mathbb{F}}}$, which implies that the assembly map $\eta^{\widehat{\mathbb{F}}}_{A\otimes B}$ is invertible, that is, $\widehat{\mathbb{F}}$ satisfies the quantum Baum-Connes property with coefficients in $A\otimes B$.
			
			\item We have just to recall that $\widehat{\mathbb{G}}$ and $\widehat{\mathbb{H}}$ are divisible torsion-free discrete quantum subgroups of $\widehat{\mathbb{F}}$ as explained in Remark \ref{rem.QuantumSubgroupsDirectProduct}. Therefore, Remark \ref{rem.BCDivisibleSubgroups} yields the assertion.
			\item In this case we apply Proposition \ref{pro.BCDivisibleQuantumSubgroups}.
		\end{enumerate}
	\end{proof}
	\begin{rem}
		It is worth mentioning the following. The element $\psi$ constructed in $(i)$ of Lemma \ref{lem.InvertibleElementDirectProduct} is such that
		\begin{equation*}
		\begin{gathered}
			\xymatrix@C=20mm@!R=15mm{
				\mbox{$L'(A)\otimes L''(B)$}\ar[d]_{\mbox{$u'\otimes u''$}}\ar[r]^-{\mbox{$\psi$}}&\mbox{$ L(A\otimes B)$}\ar[d]^{\mbox{$ u$}}\\
				\mbox{$A\otimes B$}\ar@{=}[r]&\mbox{$A\otimes B$}}
		\end{gathered}
		\end{equation*}
		
		The argument followed in $(i)$ of the preceding theorem yields actually that both $\psi$ and $u'\otimes u''$ are isomorphisms, which implies that $u$ is also an isomorphism by the commutativity of the above diagram. In other words, we have proved that the $\widehat{\mathbb{F}}$-$C^*$-algebras of the form $A\otimes B$, where $A$ is a $\widehat{\mathbb{G}}$-$C^*$-algebra and $B$ is a $\widehat{\mathbb{H}}$-$C^*$-algebra, are actually in the subcategory $\mathscr{L}_{\widehat{\mathbb{F}}}$, which yields of course the conclusion given in $(i)$ of the preceding theorem.
		
		Taking crossed products in the preceding arguments has been done just for convenience of the presentation in order to make appear more clearly the corresponding assembly maps.
	\end{rem}
	
	\bigskip
	The above theorem yields immediately the connexion of the \emph{usual} quantum Baum-Connes property for $\widehat{\mathbb{F}}=\widehat{\mathbb{G}\times\mathbb{H}}$ with the \emph{K\"{u}nneth formula} as announced in the introduction. Let $A$ be a $C^*$-algebra, we say that $A$ satisfies the K\"{u}nneth formula if for every $C^*$-algebra $B$ with free abelian $K$-group $K_*(B)$, the canonical homomorphism $K_*(A)\otimes K_*(B)\longrightarrow K_*(A\otimes B)$ is an isomorphism. Observe that this homomorphism is natural in $A$ and $B$ and it can be described in terms of the Kasparov product. We refer to Section $23$ of \cite{Blackadar} for more details.
	\begin{cor}\label{theo.BCDirectProduct}
		Let $\mathbb{F}=\mathbb{G}\times \mathbb{H}$ be a quantum direct product of compact quantum groups such that $\widehat{\mathbb{F}}$, $\widehat{\mathbb{G}}$ and $\widehat{\mathbb{H}}$ are torsion-free discrete quantum groups.
		
		For all $\widehat{\mathbb{G}}$-$C^*$-algebra $(A,\alpha)$ and all $\widehat{\mathbb{H}}$-$C^*$-algebra $(B,\beta)$ the following diagram is commutative
			\begin{equation*}
				\xymatrix@!C=70mm@R=20mm{
					\mbox{$K_*\big(\widehat{\mathbb{G}}\underset{r}{\ltimes} L'(A)\otimes \widehat{\mathbb{H}}\underset{r}{\ltimes} L''(B)\big)$}\ar[d]_{\mbox{$K_*\big(\widehat{\mathbb{G}}\underset{r}{\ltimes} u'\otimes \widehat{\mathbb{H}}\underset{r}{\ltimes} u''\big)$}}\ar[r]^{\mbox{$K_*(\Psi)$}}&\mbox{$\mathbb{L}F(A\otimes B)$}\ar[d]^{\mbox{$\eta^{\widehat{\mathbb{F}}}_{A\otimes B}$}}\\
					\mbox{$K_*\big(\widehat{\mathbb{G}}\underset{\alpha,r}{\ltimes} A\otimes \widehat{\mathbb{H}}\underset{\beta,r}{\ltimes} B\big)$}\ar[r]^{\mbox{$\cong$}}&\mbox{$F(A\otimes B)$}}
			\end{equation*}
		
		Denote by $\mathcal{N}$ the class of $C^*$-algebras satisfying the K\"{u}nneth formula.
		\begin{enumerate}[i)]
			\item If either $\widehat{\mathbb{G}}$ satisfies the strong Baum-Connes property, $\widehat{\mathbb{H}}$ satisfies the Baum-Connes property with coefficients in $B$ or $\widehat{\mathbb{H}}$ satisfies the strong Baum-Connes property, $\widehat{\mathbb{G}}$ satisfies the Baum-Connes property with coefficients in $A$; either $\widehat{\mathbb{G}}\underset{\alpha,r}{\ltimes} A$ or $\widehat{\mathbb{H}}\underset{\beta,r}{\ltimes} B$ belong to the class $\mathcal{N}$ and either $\widehat{\mathbb{G}}\underset{\alpha,r}{\ltimes} L'(A)$ (and $K_*(\widehat{\mathbb{H}}\underset{\beta,r}{\ltimes} L''(B))$ is free abelian) or $\widehat{\mathbb{H}}\underset{\beta,r}{\ltimes} L''(B)$ (and $K_*(\widehat{\mathbb{G}}\underset{\alpha,r}{\ltimes} L'(A))$ is free abelian) belong to the class $\mathcal{N}$, then $\widehat{\mathbb{F}}$ satisfies the Baum-Connes property with coefficients in $A\otimes B$.
			\item If $\widehat{\mathbb{G}}$ satisfies the strong Baum-Connes property, $\widehat{\mathbb{H}}$ satisfies the Baum-Connes property with coefficients in $\mathbb{C}$, either $C_r(\mathbb{G})$ or $C_r(\mathbb{H})$ belong to the class $\mathcal{N}$ and either $\widehat{\mathbb{G}}\underset{r}{\ltimes} L'(\mathbb{C})$ (and $K_*(\widehat{\mathbb{H}}\underset{\beta,r}{\ltimes} L''(\mathbb{C}))$ is free abelian) or $\widehat{\mathbb{H}}\underset{r}{\ltimes} L''(\mathbb{C})$ (and $K_*(\widehat{\mathbb{G}}\underset{\beta,r}{\ltimes} L'(\mathbb{C}))$ is free abelian) belong to the class $\mathcal{N}$, then $\widehat{\mathbb{F}}$ satisfies the Baum-Connes property with coefficients in $\mathbb{C}$.
		\end{enumerate}
	\end{cor}
	\begin{proof}
		The commutative diagram of the statement is obtained by simply applying the functor $K_*(\cdot)$ to diagram (\ref{eq.CommutativeDiagramDirectProduct}) from Lemma \ref{lem.InvertibleElementDirectProduct}.
		
		\begin{enumerate}[i)]
			\item Let $A$ be a $\widehat{\mathbb{G}}$-$C^*$-algebra and $B$ a $\widehat{H}$-$C^*$-algebra. Assume that $\widehat{\mathbb{G}}$ satisfies the \emph{strong} Baum-Connes property, $\widehat{\mathbb{H}}$ satisfies the Baum-Connes property with coefficients in $B$, $\widehat{\mathbb{G}}\underset{\alpha,r}{\ltimes} A\in\mathcal{N}$, $\widehat{\mathbb{G}}\underset{\alpha,r}{\ltimes} L'(A)\in\mathcal{N}$ and $K_*(\widehat{\mathbb{H}}\underset{\beta,r}{\ltimes} L''(B))$ is free abelian.
				
				The last condition guarantees that $K_*(\widehat{\mathbb{H}}\underset{\beta,r}{\ltimes} B)$ is free abelian too because the Dirac homomorphism for $B$, $L''(B)\overset{u''}{\longrightarrow} B$, induces a group homomorphism $K_*(\widehat{\mathbb{H}}\underset{\beta,r}{\ltimes} L''(B))\longrightarrow K_*(\widehat{\mathbb{H}}\underset{\beta,r}{\ltimes} B)$ by functoriality. Hence, by K\"{u}nneth formula we have natural isomorphisms $K_*\big(\widehat{\mathbb{G}}\underset{r}{\ltimes} L'(A)\otimes \widehat{\mathbb{H}}\underset{r}{\ltimes} L''(B)\big)\cong K_*\big(\widehat{\mathbb{G}}\underset{r}{\ltimes} L'(A))\otimes K_*(\widehat{\mathbb{H}}\underset{r}{\ltimes} L''(B)\big)$ and $K_*\big(\widehat{\mathbb{G}}\underset{r}{\ltimes} A\otimes \widehat{\mathbb{H}}\underset{r}{\ltimes} B\big)\cong K_*\big(\widehat{\mathbb{G}}\underset{r}{\ltimes} A)\otimes K_*(\widehat{\mathbb{H}}\underset{r}{\ltimes} B\big)$, which allows to write the commutative diagram of the statement as follows
				\begin{equation*}
				\xymatrix@!C=70mm@R=20mm{
					\mbox{$K_*\big(\widehat{\mathbb{G}}\underset{r}{\ltimes} L'(A))\otimes K_*(\widehat{\mathbb{H}}\underset{r}{\ltimes} L''(B)\big)$}\ar[d]_{\mbox{$\eta^{\widehat{\mathbb{G}}}_{A}\otimes \eta^{\widehat{\mathbb{H}}}_{B}$}}\ar[r]^{\mbox{$K_*(\Psi)$}}&\mbox{$\mathbb{L}F(A\otimes B)$}\ar[d]^{\mbox{$\eta^{\widehat{\mathbb{F}}}_{A\otimes B}$}}\\
					\mbox{$K_*\big(\widehat{\mathbb{G}}\underset{\alpha,r}{\ltimes} A)\otimes K_*( \widehat{\mathbb{H}}\underset{\beta,r}{\ltimes} B\big)$}\ar[r]^{\mbox{$\cong$}}&\mbox{$F(A\otimes B)$}}
			\end{equation*}
			
			Since $\widehat{\mathbb{G}}$ satisfies the \emph{strong} Baum-Connes property, it satisfies the Baum-Connes property with coefficients in $A$. $\widehat{\mathbb{H}}$ satisfies the Baum-Connes property with coefficients in $B$ by assumption. Hence $\eta^{\widehat{\mathbb{G}}}_{A}$ and $\eta^{\widehat{\mathbb{H}}}_{B}$ are isomorphisms. Since $\widehat{\mathbb{G}}$ satisfies the \emph{strong} Baum-Connes property, the same argument as in Theorem \ref{theo.SBCDirectProduct} shows that $\Psi$ is invertible, so $K_*(\Psi)$ of the above diagram is an isomorphism. We conclude that $\eta^{\widehat{\mathbb{F}}}_{A\otimes B}$ is an isomorphism by commutativity of the above diagram, which yields the claim. 
			\item This is a particular case of $(i)$.
		\end{enumerate}
	\end{proof}
	\begin{note}
		In order to obtain a more optimal result concerning the Baum-Connes property for a quantum direct product, we would have to carry out a detailed study about the K\"{u}nneth formula in the \emph{equivariant quantum setting} (following and inspired by \cite{ChabertEchterhoffOyono}). In particular, we would like to find sufficient conditions to a crossed product to belong to the class $\mathcal{N}$.
	\end{note}

\subsection{$K$-amenablity property}
	As in the quantum semi-direct product case, we study the $K$-amenability of $\widehat{\mathbb{F}}$ in terms of the $K$-amenability of $\widehat{\mathbb{G}}$ and $\widehat{\mathbb{H}}$. Namely, we get the following
	
	\begin{theo}\label{theo.KAmenabilityQDirectProduct}
			Let $\mathbb{F}=\mathbb{G}\times \mathbb{H}$ be a quantum direct product of compact quantum groups. Then $\mathbb{F}$ is co-$K$-amenable if and only if $\mathbb{G}$ and $\mathbb{H}$ are co-$K$-amenable.
		\end{theo}
		\begin{proof}
			Assume that $\mathbb{F}$ is co-$K$-amenable. This means that there exists an element $\alpha_{\mathbb{F}}\in KK(C_r(\mathbb{F}), \mathbb{C})$ such that $[\tau_{\mathbb{F}}]\underset{C_r(\mathbb{F})}{\otimes} \alpha_{\mathbb{F}}=[\varepsilon_{\mathbb{F}}]\in KK(C_m(\mathbb{F}),\mathbb{C})$, where $\tau_{\mathbb{F}}: C_m(\mathbb{F})\twoheadrightarrow C_r(\mathbb{F})$ is the canonical surjection and $\varepsilon_{\mathbb{F}}:Pol(\mathbb{F})\longrightarrow \mathbb{C}$ is the co-unit of $\mathbb{F}$ whose extension to $C_m(\mathbb{F})$ is still denoted by $\varepsilon_{\mathbb{F}}$.
			
			By virtue of Remark \ref{rem.QuantumSubgroupsDirectProduct} we know that $\widehat{\mathbb{G}}$ and $\widehat{\mathbb{H}}$ are discrete quantum subgroups of $\widehat{\mathbb{F}}$ via the canonical injections $\iota^r_{\mathbb{G}}: C_r(\mathbb{G})\hookrightarrow C_r(\mathbb{F})$ and $\iota^r_{\mathbb{H}}: C_r(\mathbb{H})\hookrightarrow C_r(\mathbb{F})$. In this situation it is well known that $\mathbb{G}$ and $\mathbb{H}$ are co-$K$-amenable with elements $\alpha_{\mathbb{G}}:=[\iota^r_{\mathbb{G}}]\underset{C_r(\mathbb{F})}{\otimes}\alpha_{\mathbb{F}}\in KK(C_r(\mathbb{G}), \mathbb{C})\mbox{ and } \alpha_{\mathbb{H}}:=[\iota^r_{\mathbb{H}}]\underset{C_r(\mathbb{F})}{\otimes}\alpha_{\mathbb{F}}\in KK(C_r(\mathbb{H}), \mathbb{C})$.
			
			Conversely, assume that both $\mathbb{G}$ and $\mathbb{H}$ are co-$K$-amenable. This means that there exist elements $\alpha_{\mathbb{G}}\in KK(C_r(\mathbb{G}), \mathbb{C})$ and $\alpha_{\mathbb{H}}\in KK(C_r(\mathbb{H}), \mathbb{C})$ such that $[\tau_{\mathbb{G}}]\underset{C_r(\mathbb{G})}{\otimes}\alpha_{\mathbb{G}}=[\varepsilon_{\mathbb{G}}]\mbox{ and } [\tau_{\mathbb{H}}]\underset{C_r(\mathbb{H})}{\otimes}\alpha_{\mathbb{H}}=[\varepsilon_{\mathbb{H}}]$, where $\tau_{\mathbb{G}}: C_m(\mathbb{G})\twoheadrightarrow C_r(\mathbb{G})$, $\tau_{\mathbb{H}}: C_m(\mathbb{H})\twoheadrightarrow C_r(\mathbb{H})$ are the canonical surjections and $\varepsilon_{\mathbb{G}}:Pol(\mathbb{G})\longrightarrow \mathbb{C}$, $\varepsilon_{\mathbb{H}}:Pol(\mathbb{H})\longrightarrow \mathbb{C}$ are the co-units of $\mathbb{G}$ and $\mathbb{H}$, respectively whose extensions to $C_m(\mathbb{G})$ and $C_m(\mathbb{H})$ are still denoted by $\varepsilon_{\mathbb{G}}$ and $\varepsilon_{\mathbb{H}}$, respectively.
			
			By using the canonical injections $\iota^r_{\mathbb{G}}: C_r(\mathbb{G})\hookrightarrow C_r(\mathbb{F})$ and $\iota^r_{\mathbb{H}}: C_r(\mathbb{H})\hookrightarrow C_r(\mathbb{F})$, we observe that $\tau_{\mathbb{F}}=\tau_{\mathbb{G}}\times\tau_{\mathbb{H}}$ by universal property of the maximal tensor product. We have as well that $\varepsilon_{\mathbb{F}}=\varepsilon_{\mathbb{G}}\times\varepsilon_{\mathbb{H}}$. If $\pi:C_m(\mathbb{G})\underset{\max}{\otimes} C_m(\mathbb{H})\twoheadrightarrow C_m(\mathbb{G})\otimes C_m(\mathbb{H})$ denotes the canonical surjection, then by universal property of the maximal tensor product we have the following commutative diagrams
			\begin{equation*}
			\begin{gathered}
				\xymatrix@!C=22mm@R=12mm{
					\mbox{$C_m(\mathbb{F})$}\ar@{->>}[r]^{\mbox{$\tau_{\mathbb{F}}$}}\ar@{->>}[d]_{\mbox{$\pi$}}&\mbox{$C_r(\mathbb{F})$}&\mbox{$C_m(\mathbb{F})$}\ar@{->>}[d]_{\mbox{$\pi$}}\ar[r]^{\mbox{$\varepsilon_{\mathbb{F}}$}}&\mbox{$\mathbb{C}$}\\
				\mbox{$C_m(\mathbb{G})\otimes C_m(\mathbb{H})$}\ar[ur]_{\mbox{$\tau_{\mathbb{G}}\otimes \tau_{\mathbb{H}}$}}&&\mbox{$C_m(\mathbb{G})\otimes C_m(\mathbb{H})$}\ar[ur]_{\mbox{$\varepsilon_{\mathbb{G}}\otimes\varepsilon_{\mathbb{H}}$}}}
			\end{gathered}
			\end{equation*}
			where $\tau_{\mathbb{G}}\otimes\tau_{\mathbb{H}}:C_m(\mathbb{G})\otimes C_m(\mathbb{H})\twoheadrightarrow C_r(\mathbb{G})\otimes C_r(\mathbb{H})$ is the tensor product of the canonical surjections $\tau_{\mathbb{G}}$ and $\tau_{\mathbb{H}}$ and $\varepsilon_{\mathbb{G}}\otimes\varepsilon_{\mathbb{H}}: C_m(\mathbb{G})\otimes C_m(\mathbb{H})\longrightarrow \mathbb{C}$ is the tensor product of the co-units $\varepsilon_{\mathbb{G}}$ and $\varepsilon_{\mathbb{H}}$.
			
			In this way, the canonical surjection $\tau_{\mathbb{F}}: C_m(\mathbb{F})\twoheadrightarrow C_r(\mathbb{F})$ and the co-unit $\varepsilon_{\mathbb{F}}:C_m(\mathbb{F})\longrightarrow \mathbb{C}$ can be written, as Kasparov bimodules, under the following form $[\tau_{\mathbb{F}}]=[\pi]\underset{C_m(\mathbb{G})\otimes C_m(\mathbb{H})}{\otimes}[\tau_{\mathbb{G}}\otimes\tau_{\mathbb{H}}]=\pi^*\big([\tau_{\mathbb{G}}\otimes\tau_{\mathbb{H}}]\big)$, $[\varepsilon_{\mathbb{F}}]=[\pi]\underset{C_m(\mathbb{G})\otimes C_m(\mathbb{H})}{\otimes}[\varepsilon_{\mathbb{G}}\otimes\varepsilon_{\mathbb{H}}]=\pi^*\big([\varepsilon_{\mathbb{G}}\otimes\varepsilon_{\mathbb{H}}]\big)$. Define the element $\alpha_{\mathbb{F}}:=\alpha_{\mathbb{G}}\otimes\alpha_{\mathbb{H}}:=\tau^{C_r(\mathbb{H})}_R(\alpha_{\mathbb{G}})\underset{C_r(\mathbb{H})}{\otimes}\tau^{\mathbb{C}}_L(\alpha_{\mathbb{H}})\in KK(C_r(\mathbb{F}), \mathbb{C})$. Using elementary properties of the Kasparov product and the exterior tensor product of Kasparov triples (see \cite{Blackadar} for more details), a straightforward computation yields that $[\tau_{\mathbb{F}}]\underset{C_r(\mathbb{F})}{\otimes} \alpha_{\mathbb{F}}=[\varepsilon_{\mathbb{F}}]\in KK(C_m(\mathbb{F}),\mathbb{C})$.
		\end{proof}

\nocite{Sergey, Lance, Woronowicz, Jorgensen, Amnon, BrownOzawa, Williams}
\bibliographystyle{acm}
\bibliography{BCquantumsemidirectproduct}

\bigskip
\bigskip
\bigskip
\textsc{R. Martos, Institut de Mathématiques de Jussieu-Paris Rive Gauche, UMR 7586, Univ. Paris Diderot - Paris 7, France.} 

\textit{E-mail address:} \textbf{\texttt{ruben.martos@imj-prg.fr}}

 \end{document}